\documentclass[11pt]{article}

\usepackage[utf8]{inputenc}

\usepackage{color}
\usepackage{latexsym}
\usepackage{dsfont}
\usepackage{amssymb}
\usepackage{comment}
\usepackage{graphicx}
\usepackage{placeins}
\usepackage{amsmath,amsfonts,amssymb,euscript,array,enumerate,amsfonts,mathrsfs}
\usepackage{appendix}
\usepackage[T1]{fontenc}
\usepackage{babel}
\numberwithin{equation}{section}
\usepackage{bbm}
\usepackage{subfigure}
\usepackage{color}
\usepackage[authoryear]{natbib}
\usepackage{stmaryrd} 
\usepackage{amsthm}

\usepackage[algo2e,ruled,vlined]{algorithm2e} 
 \usepackage{algorithmicx}
\usepackage{algpseudocode}
\usepackage{ marvosym }
\usepackage{csquotes}
\usepackage{hyperref}
\usepackage{framed}

\usepackage{comment}

\usepackage{tikz}
\usetikzlibrary{fit,matrix,chains,positioning,decorations.pathreplacing,arrows}

\usepackage{mathtools}
\mathtoolsset{showonlyrefs}

\usepackage[a4paper,margin=1in,heightrounded]{geometry}

\usepackage[normalem]{ulem}

\def \trans{^{\scriptscriptstyle{\intercal}}}

\def \b1{\bf{1}}

\def \N{\mathbb{N}}
\def \R{\mathbb{R}}

\def \E{\mathbb{E}}

\def \d{\mathrm{d}}

\def\esssup_#1{\underset{#1}{\mathrm{ess\,sup\, }}}

\def\argmin_#1{\underset{#1}{\mathrm{argmin\, }}}
\def\argmax_#1{\underset{#1}{\mathrm{argmax\, }}}

\def \Dc{{\cal D}}

\def \eps{\varepsilon}

\def \d{\mathrm{d}}

\def \AC{\mathcal{A}_\mathcal{C}}

\def\beqs{\begin{eqnarray*}}
\def\enqs{\end{eqnarray*}}
\def\beq{\begin{eqnarray}}
\def\enq{\end{eqnarray}}

\newtheorem{Theorem}{Theorem}[section]

\newtheorem{Proposition}{Proposition}[section]
\newtheorem{Assumption}{Assumption}[section]
\newtheorem{Lemma}{Lemma}[section]
\newtheorem{Corollary}{Corollary}[section]
\newtheorem{Remark}{Remark}[section]
\newtheorem{Example}{Example}[section]
\numberwithin{equation}{section}

\title{A stochastic mirror descent algorithm for mean-field control problems}
\author{Noufel Frikha$\,^{a}$ and Aymen Haouas$\,^{b}$ \\ \\ \\ \\
\\ $^{a\,}$Universit\'e Paris 1 Panth\'eon-Sorbonne, Centre d’Economie de la Sorbonne, \\106 Boulevard de l’H\^opital, 75642 Paris Cedex 13, France, \\
email: noufel.frikha@univ-paris1.fr \\ \\
$^{b\,}$ Ecole Polytechnique, IP Paris, France,\\ email: aymen.haouas@polytechnique.edu \\\\ }
\date{October 2026}

\begin{document}
\maketitle
\begin{abstract}
We develop a stochastic mirror descent method for continuous-time
McKean--Vlasov control with measure-valued controls and general Bregman
regularisation. The dependence on the state law leads to a coupled
Wasserstein-Bregman stability problem and to a McKean-Vlasov adjoint equation,
which have no counterpart in classical stochastic control. We establish a
first-order representation of the cost and prove its relative smoothness in
two regimes: nonlinear dynamics with bounded cost derivatives, and affine
state-law dynamics with linearly growing cost derivatives, the latter covering
linear-quadratic (LQ) and more general nonquadratic models. Our regularity conditions in the
control measure are formulated through directional finite increments and do
not require second-order flat derivatives. We further introduce relative
semiconvexity along admissible forward--backward trajectories, allowing
compensation between the dynamics, running cost, and terminal cost beyond
pointwise Hamiltonian convexity. This yields energy dissipation, an
$O(1/n)$ convergence rate at the convexity threshold, and geometric
convergence when regularisation dominates relative nonconvexity. Explicit LQ
criteria and numerical experiments illustrate the scope of the results.
\end{abstract}

\noindent \textbf{Keywords: }
Mean-field control, stochastic mirror descent, measure-valued controls,
relative smoothness, Bregman divergence, linear-quadratic models.

\section{Introduction}

Mean-field control (MFC), also referred to as McKean--Vlasov (MKV)
control, studies the optimisation of large interacting populations under a
common objective. The state $X^\pi$ of a representative agent is governed by
the controlled MKV equation
\begin{align}\label{intro:SDEX}
\mathrm{d}X_s^\pi
= b_s(X_s^\pi,\mathcal{L}(X_s^\pi),\pi_s)\,\mathrm{d}s
+ \sigma_s(X_s^\pi,\mathcal{L}(X_s^\pi),\pi_s)\,\mathrm{d}W_s,
\quad s\in[0,T],
\end{align}
where $W$ is a $d'$-dimensional Brownian motion,
$\mathcal L(X_s^\pi)$ is the population distribution, and $\pi$ is a
progressively measurable relaxed control with values in a convex subset
$\mathscr C$ of $\mathcal P(A)$. Such equations arise as mean-field limits of
cooperative systems of weakly interacting agents; see
\cite{cardel:1,cardel:2}. For a regularisation parameter $\tau\geq0$, the
social planner minimises
\begin{equation}\label{intro:cost}
J^\tau(\pi)
= \mathbb{E}\!\left[
  \int_0^T\!\bigl(f_s(X_s^\pi,\mathcal{L}(X_s^\pi),\pi_s)+\tau\,h(\pi_s)\bigr)\mathrm{d}s
  + g(X_T^\pi,\mathcal{L}(X_T^\pi))
\right]
\end{equation}
over admissible controls $\pi\in\mathcal A_{\mathscr C}$. Here $f$ and $g$
are the running and terminal costs, while the convex functional $h$ induces
a geometry on the space of probability measures. Relative entropy,
$\chi^2$ divergence, and entropic optimal transport are representative
examples.

The principal difficulty is that a perturbation of the relaxed control changes both the coefficients of the
state equation and the law of its solution. The corresponding first-order
analysis involves an MKV adjoint BSDE with independent-copy terms, whereas
stability of the optimisation scheme requires simultaneous control of the
Bregman displacement of the policy and the Wasserstein displacement of the
state law. This interaction has no counterpart in classical stochastic
control.

\medskip
\noindent\textit{Numerical methods and learning for mean-field control.}
Dynamic programming leads to a Bellman equation on the Wasserstein space
\cite{phawei17}. Since direct discretisation is quickly intractable, recent
methods use particle approximations and neural representations adapted to
probability measures
\cite{gerlauphawar21,germikwar22,hanhulong22,daylauzha2024,soner:teichmann:yan25}.
The alternative Pontryagin route characterises the optimum through an
MKV-FBSDE and motivates iterative model-based solvers
\cite{cardel15,carlau22,phawar22,rei:stock:zhang}. These two approaches
provide powerful numerical tools, but they do not by themselves identify a
geometry on the space of randomised controls that yields global descent and
quantitative iteration complexity.

Randomised policies and entropy regularisation are natural in reinforcement
learning, where they promote exploration. Continuous-time policy-gradient
and actor--critic methods were developed first for controlled diffusions
\cite{wanzarzhou20,jiazhou21a,jiazhou21b,jiazhou23} and subsequently for MFC~\cite{frikha:germain:lauriere:pham:song},
\cite{bayraktar:hernandez:yan:zhu26}.
Convergence results are available for specific linear-quadratic (LQ)
mean-field models in discrete and continuous time
\cite{carmona2019linear,pmlr-v139-wang21j,frikha:pham:song}. For classical finite-horizon LQ control,
geometry-aware policy gradients can also accommodate certain indefinite
costs and Gaussian policies~\cite{giegrich:reisinger:zhang24}. These results
demonstrate the effectiveness of policy optimisation, but their convergence
analyses remain tied either to a finite-dimensional parametrisation or to a
specific LQ structure.

\medskip
\noindent\textit{Mirror descent and the remaining gap.}
Mirror descent replaces the Euclidean penalty of gradient descent by the
Bregman divergence of a convex functional $h$. Relative smoothness and
relative convexity then provide the appropriate descent and contraction
principles in measure spaces; see
\cite{NEURIPS2022_6e3daaec} for the static measure-valued setting. \cite{KERIMKULOV2025104765} establish this programme for finite-horizon
classical stochastic control with measure-valued controls. They prove
relative smoothness, dissipation, and sublinear or exponential convergence
for several Bregman geometries using sharp BMO estimates for the adjoint
BSDE. Their setting, however, does not include the endogenous law dependence
of MKV dynamics. Moreover, the bounded-cost-derivative regime used in their
relative-smoothness analysis excludes standard quadratic costs on unbounded
state spaces.

The existing literature therefore leaves a gap between two complementary
approaches. Policy-gradient and actor--critic methods for MFC provide
implementable algorithms, but their quantitative analyses are generally
tied to finite-dimensional parametrisations or specific LQ structures.
Conversely, mirror descent accommodates general measure-valued controls and
Bregman geometries, but the available stochastic-control theory does not
cover MKV interactions or the affine-growth regime needed for quadratic
costs. Pointwise Hamiltonian convexity is also poorly adapted to this
setting: when imposed for every adjoint pair, it forces the first-order
remainders of the drift and diffusion to vanish along the admissible
directions and leaves no mechanism for compensating negative curvature
between the dynamics and the costs.

\medskip
\noindent\textit{Overview of the results.}
We bridge this gap under two complementary regularity regimes. The first
allows nonlinear dynamics and bounded cost derivatives and relies on uniform
$H^\infty$ and BMO estimates for the adjoint process. The second exploits
affine dependence of the dynamics on the state and its law, replacing
boundedness of the cost derivatives by linear-growth assumptions. This
regime contains LQ models and more general nonquadratic costs. In the
entropy-regularised affine-LQ case, SMD preserves Gaussian policies, and
relative smoothness is proved directly on this invariant class rather than
on the full entropy domain.

We also replace pointwise Hamiltonian convexity by relative semiconvexity
along admissible trajectories. Its constant $\rho$ measures the negative
curvature of the unregularised objective in the Bregman geometry. Entropic
regularisation then acts as a convexification mechanism: the critical regime
$\tau=\rho$ yields an $O(1/n)$ rate, whereas $\tau>\rho$ gives geometric
convergence. A systemic-risk model with procyclical funding provides a
financially motivated example for which $\rho>0$: weak institutions
deleverage under funding pressure while better-capitalised institutions
expand, amplifying cross-sectional dispersion.

The abstract relative semiconvexity inequality is not, by itself, the main
substantive contribution: through the performance-difference identity it is
equivalent to a lower first-order bound for $J^0$. Its value here is that we
make it verifiable. For affine dynamics and quadratic costs, we reduce it to
an explicit mean--fluctuation quadratic form; this both recovers standard LQ
models and produces a computable strictly positive curvature defect in the
procyclical example. Together with the affine-growth smoothness theory and
its Gaussian-orbit specialisation, these criteria extend the framework of~\cite{KERIMKULOV2025104765}. In particular, they cover
genuinely mean-field and quadratic-growth settings.

Our main contributions are the following.

\begin{itemize}

\item \textbf{MKV first-order calculus and stability.}
We derive the first variation of the objective through an MKV adjoint BSDE
and establish the forward--backward stability estimates required by SMD. The
analysis separates a nonlinear bounded-derivative regime from an affine
state-law regime with linearly growing cost derivatives and
$S^2\times H^2$ estimates.

\item \textbf{Relative smoothness in Wasserstein--Bregman geometry.}
We control the nonlinear remainder of the cost by the Bregman displacement
of the control, including the induced Wasserstein displacement of the state
law. The regularity assumptions in the control measure are formulated
through finite directional remainders and avoid second-order flat
derivatives. For affine-LQ models, the estimate is established directly on
the invariant Gaussian orbit.

\item \textbf{Trajectory-wise semiconvexity and convexification.}
Instead of pointwise Hamiltonian convexity, we introduce relative
semiconvexity along admissible trajectories. It allows compensation between
the dynamics, running cost, and terminal cost. We prove energy dissipation
without convexity, an $O(1/n)$ rate at the threshold $\tau=\rho$, and
geometric convergence when $\tau>\rho$. For affine dynamics with quadratic
costs, we turn this abstract condition into an explicit mean--fluctuation
quadratic form. This reduction both recovers the standard LQ regime and
identifies financially meaningful models with strictly negative curvature
$\rho>0$.

\item \textbf{Affine-LQ models, regularisers, and examples.}
For entropy-regularised LQ control, we prove that a Gaussian initial policy
generates Gaussian SMD iterates and give the exact mean--precision recursion.
We verify the mirror step for relative entropy,
$\chi^2$ divergence, and entropic optimal transport, and illustrate the
theory with systemic-risk, optimal-trading, and control-dependent-volatility
models.

\end{itemize}

To the best of our knowledge, this is the first convergence analysis of
SMD algorithms for MKV control that combines measure-valued
controls, non-Euclidean Bregman geometries, and an affine-quadratic growth regime
covering entropy-regularised LQ models.

\medskip
\noindent\textit{Organisation of the paper.}
Section~\ref{sec:main-results} presents the SMD algorithm, the standing
assumptions, and the convergence theory. Section~\ref{sec:numerical-experiments}
presents the numerical experiments. A short conclusion precedes the proofs
of the main results collected in Section~\ref{sec:proofs}.

\section{Stochastic mirror descent: framework and convergence theory}
\label{sec:main-results}

\subsection{The stochastic mirror descent algorithm}

The coefficients $b$, $\sigma$, and $f$ are jointly measurable on
$[0,T]\times\mathbb R^d\times\mathcal P_2(\mathbb R^d)\times\mathscr C$,
$g$ is measurable on
$\mathbb R^d\times\mathcal P_2(\mathbb R^d)$, and
$X_0\in L^2(\Omega,\mathcal F_0,\mathbb P)$.

Assume that the convex regulariser $h$ admits a flat derivative. Thus, for
$m,m'\in\mathscr C$,
\[
\lim_{\varepsilon\downarrow0}
\frac{h(m+\varepsilon(m'-m))-h(m)}{\varepsilon}
=\int_A\frac{\delta h}{\delta m}(m,a)(m'-m)(\d a),
\]
where we use the normalisation
$\int_A(\delta h/\delta m)(m,a)m(\d a)=0$. 

Similarly, whenever a flat derivative of a time-dependent coefficient is used below, we
assume that it is chosen jointly Borel measurable in all its variables. Thus,
if $\varphi\in\{b,\sigma,f\}$ admits a flat derivative,
\[
(t,x,\mu,m,a)\longmapsto
\frac{\delta\varphi_t}{\delta m}(x,\mu,m,a)
\]
is jointly Borel measurable on its natural domain. Flat derivatives are
normalised by requiring their integral with respect to the differentiated
measure to vanish. This convention ensures that their evaluation along a
progressively measurable control and the associated state--adjoint processes
admits a progressively measurable version.

The Bregman divergence associated to $h$ is
\[
D_h(m'\mid m)
=h(m')-h(m)-\int_A\frac{\delta h}{\delta m}(m,a)(m'-m)(\d a).
\]
For progressively measurable $\mathscr C$-valued controls, set
\[
\mathcal D_h(\pi'\mid\pi)
:=\E\int_0^T D_h(\pi'_t\mid\pi_t)\,\d t.
\]
For  a fixed reference measure
$\nu_0\in\mathscr C$, the set of admissible controls is
\[
\mathcal A_{\mathscr C}
:=\left\{\pi:\Omega\times[0,T]\longrightarrow\mathscr C:
\begin{array}{l}
\pi\text{ is }\mathbb F\text{-progressively measurable},\\
\mathcal D_h(\pi\mid\nu_0)<\infty,\quad
\E\int_0^T|h(\pi_t)|\,\d t<\infty
\end{array}\right\},
\]
where the deterministic measure $\nu_0$ is identified with the constant
control $t\mapsto\nu_0$. Additional moment requirements will be stated
explicitly when needed in the linearly growing regime.

For $\tau\geq0$, define the Hamiltonian
\[
H_t^\tau(x,\mu,y,z,m)
=\langle b_t(x,\mu,m),y\rangle
+\operatorname{tr}\bigl(\sigma_t(x,\mu,m)^{\trans}z\bigr)
+f_t(x,\mu,m)+\tau h(m).
\]
For a control $\pi$, we write
\[
\Theta_t^\pi
:=(X_t^\pi,\mathcal L(X_t^\pi),Y_t^\pi,Z_t^\pi).
\]
The adjoint process solves
\begin{equation}\label{eq:adjoint-mkv-full}
\begin{cases}
\d Y_t^\pi=-\Bigl[
D_xH_t^0(\Theta_t^\pi,\pi_t)
+\widetilde\E\bigl[
(\partial_\mu H_t^0)
(\widetilde\Theta_t^\pi,\widetilde\pi_t)(X_t^\pi)
\bigr]\Bigr]\,\d t+Z_t^\pi\,\d W_t,\\[4pt]
Y_T^\pi=D_xg(X_T^\pi,\mathcal L(X_T^\pi))
+\widetilde\E\bigl[
(\partial_\mu g)(\widetilde X_T^\pi,
\mathcal L(X_T^\pi))(X_T^\pi)
\bigr],
\end{cases}
\end{equation}

\noindent with $\widetilde\Theta_t^\pi =(\widetilde X_t^\pi, \mathcal L(X_t^\pi), \widetilde Y_t^\pi, \widetilde Z_t^\pi)$ and where $(\widetilde X_t^\pi, \widetilde Y_t^\pi,  \widetilde Z_t^\pi,  \widetilde \pi_t)_{0\leq t \leq T}$ stands for an independent copy of $( X_t^\pi,  Y_t^\pi,   Z_t^\pi,  \pi_t)_{0\leq t \leq T}$.

Under the assumptions below, Lemma~\ref{lem:gateaux} gives the first-order
representation
\[
\d J^0(\pi)(\pi'-\pi):=\frac{\d}{\d\varepsilon}
J^0\bigl(\pi+\varepsilon(\pi'-\pi)\bigr)
\bigg|_{\varepsilon=0}
=\E\int_0^T\int_A
\frac{\delta H_t^0}{\delta m}(\Theta_t^\pi,\pi_t,a)
(\pi_t'-\pi_t)(\d a)\,\d t.
\]
Thus $a\mapsto(\delta H_t^0/\delta m)
(\Theta_t^\pi,\pi_t,a)$ is the stochastic action gradient. Starting from
$\pi^0$, SMD alternates between solving the MKV forward--backward system and
performing a pointwise mirror step.

\begin{center}
\begin{minipage}{1\textwidth}
\begin{framed}
\noindent\textbf{Stochastic Mirror Descent algorithm.}\\
\textbf{Input:} $\pi^0\in\mathcal A_{\mathscr C}$ and $\lambda>0$.\\
For $n=0,1,2,\ldots$:
\begin{enumerate}
\item Solve the state equation~\eqref{intro:SDEX} and the adjoint
equation~\eqref{eq:adjoint-mkv-full} under $\pi^n$.
\item For every $t\in[0,T]$, set
\begin{equation}\label{eq:MDupdate}
\pi_t^{n+1}\in\argmin_{m\in\mathscr C}
\left\{
\int_A\frac{\delta H_t^\tau}{\delta m}
(\Theta_t^{\pi^n},\pi_t^n,a)(m-\pi_t^n)(\d a)
+\lambda D_h(m\mid\pi_t^n)
\right\}.
\end{equation}
\end{enumerate}
\end{framed}
\end{minipage}
\end{center}

The parameter $\lambda$ is the inverse step size. The first term in
\eqref{eq:MDupdate} follows the first variation of the regularised objective,
whereas the Bregman term keeps the update in the geometry induced by $h$.

\subsection{Well-posedness of the state and adjoint equations}

We first collect the assumptions ensuring that, for every admissible control,
the state equation and the adjoint equation used by SMD are well posed. The
state equation is treated under a common Lipschitz framework. For the adjoint
equation, we distinguish a bounded-derivative regime, which yields estimates
uniform in the control, from a linear-growth regime covering quadratic costs,
where the $S^2\times H^2$ estimate is asserted for each fixed control.

\begin{Assumption}\label{ass:H1}
There exists $K<\infty$ such that for all $t\in[0,T]$, $x,x'\in\R^d$,
$\mu,\mu'\in\mathcal{P}_2(\R^d)$ and $m,m'\in \mathscr{C}$,
\begin{align*}
  |b_t(x,\mu,m)-b_t(x',\mu',m')|^2
  & + |\sigma_t(x,\mu,m)-\sigma_t(x',\mu',m')|^2
  \; \\
  & \le\; K\Bigl(|x-x'|^2+W_2(\mu,\mu')^2+D_h(m\mid m')\Bigr).
\end{align*}
\end{Assumption}

\begin{Assumption}\label{ass:H2}
There exists $\mu_0\in\mathcal{P}_2(\R^d)$ such that, for the reference
measure $\nu_0\in\mathscr C$ fixed above,
\[
  \int_0^T \!\bigl(|b_t(0,\mu_0,\nu_0)|^2
  + |\sigma_t(0,\mu_0,\nu_0)|^2\bigr)\, \d t < \infty.
\]
\end{Assumption}

We shall use two alternative regularity regimes for the adjoint equation. The
first one, stated next, imposes boundedness of all first-order derivatives. It
allows for general nonlinear state and law dependence in the coefficients and
is the natural setting for the uniform $H^\infty$ and BMO estimates established
below.

\begin{Assumption}[Bounded first-order derivatives]\label{ass:H3}
The maps $b$, $\sigma$, $f$ and $g$ are differentiable in $x$ and admit first order Wasserstein derivative in $\mu$.
Moreover, there exists $K<\infty$ such that, for every $\varphi\in\{b,\sigma,f\}$ and every
$(t,x,\mu,m,\xi)\in[0,T]\times\R^d\times\mathcal{P}_2(\R^d)\times \mathscr{C}\times\R^d$,
$$
  |D_x\varphi_t(x,\mu,m)| + |\partial_\mu\varphi_t(x,\mu,m)(\xi)| + |D_xg(x,\mu)| + |\partial_\mu g(x,\mu)(\xi)|
  \;\le\; K.
$$

\end{Assumption}

The preceding assumption excludes quadratic costs, whose first-order
derivatives are generally unbounded. Our second regime relaxes this condition:
the derivatives of the running and terminal costs are allowed to have linear
growth. The derivatives of the drift and diffusion coefficients are still
assumed to be bounded.

\begin{Assumption}[Linearly growing cost derivatives]\label{ass:H3gen}
The maps $b$ and $\sigma$ are differentiable in $x$ and admit a first-order
Lions derivative with respect to $\mu$. There exists $K<\infty$ such that
\[
  |D_x\varphi_t(x,\mu,m)| + |\partial_\mu\varphi_t(x,\mu,m)(\xi)| \;\le\; K,
  \qquad \varphi\in\{b,\sigma\}.
\]

The maps $f_t$ and $g$ are differentiable in $x$ and admit a first-order
Lions derivative with respect to $\mu$, and these derivatives are continuous
with respect to their entries. Denote by $d_A$ the metric on $A$, fix
$a_0\in A$, and set
\[
 M_2(\mu):=\left(\int_{\mathbb R^d}|v|^2\,\mu(\d v)\right)^{1/2},
 \qquad
 M_{2,A}(m):=\left(\int_A d_A(a,a_0)^2\,m(\d a)\right)^{1/2}.
\]
We assume that $\mathscr C\subset\mathcal P_2(A)$ and that every admissible
control considered in this regime satisfies
\[
 \E\int_0^T M_{2,A}(\pi_t)^2\,\d t<\infty.
\]
Moreover, for every
$(t,x,\mu,m,\xi)\in[0,T]\times\mathbb R^d\times
\mathcal P_2(\mathbb R^d)\times\mathscr C\times\mathbb R^d$,
\[
 |D_xf_t(x,\mu,m)|
 +|\partial_\mu f_t(x,\mu,m)(\xi)|
 \le K\bigl(1+|x|+M_2(\mu)+|\xi|+M_{2,A}(m)\bigr),
\]
and, for every $(x,\mu,\xi)\in\mathbb R^d\times
\mathcal P_2(\mathbb R^d)\times\mathbb R^d$,
\[
 |D_xg(x,\mu)|
 +|\partial_\mu g(x,\mu)(\xi)|
 \le K\bigl(1+|x|+M_2(\mu)+|\xi|\bigr).
\]

\end{Assumption}

In this second regime, the adjoint processes are naturally controlled in
$S^2\times H^2$, but they need not satisfy the uniform $H^\infty$ and BMO
bounds available under Assumption~\ref{ass:H3}. To obtain the stability estimate
of Lemma~\ref{lem:YZlipc} without invoking those stronger bounds, we complement
Assumption~\ref{ass:H3gen} with the following structural condition on the
dynamics.

\begin{Assumption}[Affine state-mean structure]\label{ass:affine-dynamics}
There exist bounded measurable maps
\[
 B:[0,T]\to\mathbb R^{d\times d},
 \qquad
 \Sigma:[0,T]\to\mathbb R^{d\times d\times d'},
\]
and
\[
 \overline B:[0,T]\to\mathbb R^{d\times d},
 \qquad
 \overline\Sigma:[0,T]\to\mathbb R^{d\times d\times d'},
\]
such that, for every
$(t,x,\mu,m,\xi)\in[0,T]\times\mathbb R^d
\times\mathcal P_2(\mathbb R^d)\times\mathscr C\times\mathbb R^d$,
\[
 D_xb_t(x,\mu,m)=B_t,
 \qquad
 D_x\sigma_t(x,\mu,m)=\Sigma_t,
\]
and
\[
 \partial_\mu b_t(x,\mu,m)(\xi)=\overline B_t,
 \qquad
 \partial_\mu\sigma_t(x,\mu,m)(\xi)
       =\overline\Sigma_t.
\]
Equivalently, up to terms depending only on $(t,m)$, the coefficients $b_t$
and $\sigma_t$ are affine functions of the state $x$ and of the mean
$\int_{\mathbb R^d}\xi\,\mu(\d\xi)$ of the state law.
\end{Assumption}

\begin{Remark}
Assumptions~\ref{ass:H3} and
\ref{ass:H3gen}--\ref{ass:affine-dynamics} describe two alternative regimes,
rather than cumulative requirements. Assumption~\ref{ass:H3} accommodates
nonlinear dynamics but requires bounded cost derivatives. The combination of
Assumptions~\ref{ass:H3gen} and~\ref{ass:affine-dynamics} allows cost derivatives
with linear growth and, in particular, covers the standard LQ framework as
well as nonquadratic costs with globally Lipschitz first-order derivatives.
Assumption~\ref{ass:affine-dynamics} concerns only the state and law derivatives
of $b$ and $\sigma$; their dependence on the control variable is governed by
the assumptions introduced later.
\end{Remark}

\begin{Proposition}\label{prop:wellposed}
Under Assumptions~\ref{ass:H1} and~\ref{ass:H2}, for every
$\pi\in\mathcal{A}_{\mathscr{C}}$ and every initial condition
$X_0^{\pi}\in L^2(\mathbb{P})$, the state MKV-SDE~\eqref{intro:SDEX}
admits a unique strong solution such that
\[
 \mathbb{E}\left[\sup_{0\leq t\leq T}|X_t^{\pi}|^2\right]<\infty.
\]
Moreover, the following assertions hold:
\begin{enumerate}
\item Under the additional Assumption~\ref{ass:H3}, the MKV
adjoint BSDE~\eqref{eq:adjoint-mkv-full} admits a unique solution in
$S^2\times H^2$, and
\begin{equation}\label{estimates:bsde:S2:H2:uniform:pi}
\sup_{\pi\in\mathcal{A}_{\mathscr{C}}}
\mathbb{E}\left[
 \sup_{0\leq t\leq T}|Y_t^{\pi}|^2
 +\int_0^T|Z_t^{\pi}|^2\,\d t
\right]<\infty.
\end{equation}
\item If, instead, Assumption~\ref{ass:H3gen} holds, then, for every fixed
$\pi\in\mathcal{A}_{\mathscr{C}}$, the adjoint
BSDE~\eqref{eq:adjoint-mkv-full} admits a unique solution in
$S^2\times H^2$; equivalently,
\begin{equation}\label{estimates:bsde:S2:H2:fixed:pi}
\mathbb{E}\left[
 \sup_{0\leq t\leq T}|Y_t^{\pi}|^2
 +\int_0^T|Z_t^{\pi}|^2\,\d t
\right]<\infty.
\end{equation}
The bound in~\eqref{estimates:bsde:S2:H2:fixed:pi} may depend on $\pi$;
in particular, no uniform estimate over
$\mathcal{A}_{\mathscr{C}}$ is asserted in this second regime.
\end{enumerate}
\end{Proposition}

The proof is omitted. It follows from the standard $S^2$ estimates for
MKV SDEs and linear MKV BSDEs; see, for instance,
\cite{cardel15}. The independent-copy terms are controlled by Jensen's
inequality. Under Assumption~\ref{ass:H3} the inhomogeneous terms of the
adjoint equation are uniformly bounded, whereas under
Assumption~\ref{ass:H3gen} their square integrability follows from the state
estimate and the control-moment condition, which explains why the latter
bound is only asserted for each fixed $\pi$.

\subsection{Relative smoothness, dissipation, and convergence}

We now turn from the well-posedness of the controlled McKean--Vlasov
forward--backward system to the stability and convergence of the SMD
iteration. The argument has three layers. We first establish relative
smoothness of the objective with respect to the Bregman geometry generated by
$h$, which yields energy dissipation. We then impose a trajectory-wise
relative semiconvexity condition under which the regularisation compensates
for the negative curvature of the unregularised objective. This makes the
first-order optimality condition sufficient for global optimality and leads
to quantitative convergence rates for the objective values. Finally, we
specialise these abstract conditions to affine and LQ models, including
situations in which pointwise convexity of the Hamiltonian fails.

Table~\ref{tab:regularity-regimes} records the three frameworks used below.
The first two are the global regimes of Theorem~\ref{thm:relsmooth}; the
third is the invariant-orbit formulation needed for entropy-regularised LQ
models.
\begin{table}[htbp]
\centering
\small
\begin{tabular}{>{\raggedright\arraybackslash}p{0.15\textwidth}
>{\raggedright\arraybackslash}p{0.28\textwidth}
>{\raggedright\arraybackslash}p{0.24\textwidth}
>{\raggedright\arraybackslash}p{0.21\textwidth}}
\hline
\textbf{Regime} & \textbf{Dynamics and costs} & \textbf{Control domain} & \textbf{Smoothness result} \\
\hline
Nonlinear bounded & General nonlinear dynamics; bounded first-order cost derivatives
& General convex $\mathscr C$ satisfying the Bregman increment assumptions
& Global, Theorem~\ref{thm:relsmooth}(i) \\
Affine growth & Dynamics affine in the state and its mean; linearly growing cost derivatives
& Convex $\mathscr C$ with the required moment and finite-increment controls
& Global, Theorem~\ref{thm:relsmooth}(ii) \\
Affine-LQ Gaussian & Affine dynamics and quadratic costs
& Invariant class of nondegenerate Gaussian iterates
& Orbit-wise, Corollary~\ref{cor:gaussian-lq-smd} \\
\hline
\end{tabular}
\caption{The three relative-smoothness frameworks.}
\label{tab:regularity-regimes}
\end{table}

\begin{Assumption}[Well-posedness of iterates]\label{ass:iterates}
For the step parameter $\lambda>0$ under consideration, there exists
$\pi^0\in\mathcal{A}_{\mathscr{C}}$ such that the iterates
$\{\pi^n\}_{n\in\N}$ generated by Algorithm~1 exist, belong to
$\mathcal{A}_{\mathscr{C}}$, and satisfy
$\mathcal{D}_h(\pi^{n+1}\mid\pi^n)<\infty$ for all $n\geq0$.
\end{Assumption}

In general, Assumption \ref{ass:iterates} needs to be verified depending on the specific choices of $h$. We refer the reader to Section~\ref{sec:examples-h} for more details on this point.

\begin{Assumption}[Regularity of the derivatives]\label{ass:deriv}
There exists $K\ge 0$ such that for all $\varphi\in\{b,f,\sigma\}$,
$t\in[0,T]$, $x,x'\in\R^d$, $\mu,\mu'\in\mathcal{P}_2(\R^d)$, $m,m'\in \mathscr{C}$ and
$\xi,\xi'\in\R^d$,
\[
  |(D_x\varphi_t)(x,\mu,m)-(D_x\varphi_t)(x',\mu',m')|^2
  \le K\bigl(|x-x'|^2 + W_2(\mu,\mu')^2+D_h(m\mid m')\bigr),
\]
\[
  |(\partial_\mu\varphi_t)(x,\mu,m)(\xi)-(\partial_\mu\varphi_t)(x',\mu',m')(\xi')|^2
  \le K\bigl(|x-x'|^2 + W_2(\mu,\mu')^2 + |\xi-\xi'|^2 + D_h(m\mid m')\bigr),
\]
and for $g$,
\[
  |(D_x g)(x,\mu)-(D_x g)(x',\mu')|^2\le K\bigl(|x-x'|^2+W_2(\mu,\mu')^2\bigr),
\]
\[
  |(\partial_\mu g)(x,\mu)(\xi)-(\partial_\mu g)(x',\mu')(\xi')|^2
  \le K\bigl(|x-x'|^2+W_2(\mu,\mu')^2+|\xi-\xi'|^2\bigr).
\]
\end{Assumption}

\begin{Assumption}[Regularity of control-measure derivatives]\label{ass:measderiv}
For each $\varphi\in\{b,\sigma,f\}$, the map
$m\mapsto\varphi_t(x,\mu,m)$ admits a flat derivative
$\frac{\delta\varphi}{\delta m}$ on $\mathscr C$. For every
$m,m'\in\mathscr C$, the map
\[
 s\longmapsto
 \int_A\frac{\delta\varphi_t}{\delta m}
 \bigl(x,\mu,m+s(m'-m),a\bigr)(m'-m)(\d a)
\]
is continuous on $[0,1]$.
Moreover, there exists $K\ge 0$ such that for all $t\in[0,T]$,
$x\in\R^d$, $\mu\in\mathcal{P}_2(\R^d)$ and $m,m',m''\in \mathscr{C}$,
\[
  \left|\int_A \frac{\delta\varphi_t}{\delta m}
  (x,\mu,m'',a)\,(m-m')(\d a)\right|^2
  \le K\,D_h(m\mid m').
\]
\end{Assumption}

The preceding assumption is sufficient for the first-order differentiation of
the state and cost functionals. The relative smoothness estimate additionally
requires stability of the directional flat derivatives with respect to the
state and its law, together with a finite-increment control on the first-order
Taylor remainders in the control variable.

For $\varphi\in\{b,\sigma,f\}$, define the control-measure remainder
\begin{align*}
 \mathcal R_t^\varphi(x,\mu;m'\mid m)
 :={}&\varphi_t(x,\mu,m')-\varphi_t(x,\mu,m)\\
 &-\int_A\frac{\delta\varphi_t}{\delta m}(x,\mu,m,a)
 (m'-m)(\d a).
\end{align*}

\begin{Assumption}[Finite-increment regularity for relative smoothness]
\label{ass:relative-smoothness-regularity}
For each $\varphi\in\{b,\sigma,f\}$, there exists $K\geq0$ such that, for all
$t\in[0,T]$, $x,x'\in\mathbb R^d$,
$\mu,\mu'\in\mathcal P_2(\mathbb R^d)$ and $m,m',\widehat m\in\mathscr C$,
\begin{align*}
&\left|\int_A
\left[
\frac{\delta\varphi_t}{\delta m}(x,\mu,\widehat m,a)
-\frac{\delta\varphi_t}{\delta m}(x',\mu',\widehat m,a)
\right](m'-m)(\d a)\right|^2
\\
&\hspace{35mm}\le
K\bigl(|x-x'|^2+W_2(\mu,\mu')^2\bigr)D_h(m'\mid m).
\end{align*}
Moreover,
\[
 |\mathcal R_t^b(x,\mu;m'\mid m)|
 +|\mathcal R_t^f(x,\mu;m'\mid m)|
 \le K D_h(m'\mid m).
\]
Finally, for every $(t,x,\mu)$, the map
$m\mapsto\sigma_t(x,\mu,m)$ is affine on $\mathscr C$, or equivalently\footnote{
This equivalence uses the convexity of $\mathscr C$ and the existence of the
flat derivative from Assumption~\ref{ass:measderiv}. Indeed, fixing
$(t,x,\mu)$ and writing $F(m)=\sigma_t(x,\mu,m)$, if
$\mathcal R^F(m'\mid m)=0$ and $m_s=(1-s)m+sm'$, then
$F(m_s)-F(m)=s(F(m')-F(m))$, so $F$ is affine. Conversely, differentiating
the affine identity along the segment from $m$ to $m'$ and using the
definition of the flat derivative gives $\mathcal R^F(m'\mid m)=0$. This
intrinsic equivalence should not be confused with the pointwise condition
$\delta^2F/\delta m^2\equiv0$ for an iteratively normalized representative.}
\linebreak
$\mathcal R_t^\sigma(x,\mu;m'\mid m)=0$ for all $m,m'\in\mathscr C$.
\end{Assumption}

\begin{Remark}
Assumption~\ref{ass:relative-smoothness-regularity} is formulated entirely in
terms of directional finite increments. In particular, it neither requires
the existence of second order flat derivatives nor imposes a pointwise condition on
a chosen representative of such derivatives. With respect to the
control-measure variable, this is strictly weaker than the regularity framework
of~\cite[Assumption~2.8]{KERIMKULOV2025104765}, which assumes the existence of
$\delta^2\varphi/\delta m^2$ for $\varphi\in\{b,\sigma,f\}$ and imposes
$\delta^2\sigma/\delta m^2=0$ pointwise. That assumption also requires
differentiability in $x$ of $\delta\varphi/\delta m$ and a directional bound
on $D_x\delta\varphi/\delta m$, whereas the first condition above is stated
directly in finite-difference form. The present formulation retains only the
increment and remainder estimates used in the relative smoothness argument,
which constitutes a substantial weakening of the regularity in the
control-measure variable. The additional directional stability with respect
to $\mu$ is specific to the MKV setting and has no counterpart in
the state-dependent framework of~\cite{KERIMKULOV2025104765}.

Whenever the second order flat derivatives from the former formulation exist, the
fundamental theorem of calculus gives, with $q=m'-m$ and
$m_s=m+s(m'-m)$,
\[
 \mathcal R_t^\varphi(x,\mu;m'\mid m)
 =\int_0^1(1-s)\int_A\int_A
 \frac{\delta^2\varphi_t}{\delta m^2}
 (x,\mu,m_s,a,a')q(\d a)q(\d a')\,\d s.
\]
Thus the former directional second-derivative bound implies the present
finite-remainder bound, whereas the converse does not require a second flat
derivative to exist.

The difference is genuine. 
%Let $A=\{0,1\}$, identify $m\in\mathcal P(A)$
%with $p=m(\{1\})$, and take $h(m)=p^2/2$. For $p_0\in(0,1)$, the functional
%\[
% \varphi(m)=\frac12(p-p_0)|p-p_0|
%\]
%has a Lipschitz first derivative and therefore satisfies
%$|\mathcal R^\varphi(m'\mid m)|\le D_h(m'\mid m)$, but it has no second flat
%derivative at $p=p_0$. 
As an example, if $A\subseteq\mathbb R^k$ and
\[
 \sigma_t(x,\mu,m)=\sigma_t^0(x,\mu)
   +C_t\int_A a\,m(\d a),
\]
then $\mathcal R_t^\sigma=0$. Under our normalization convention, however,
\[
 \frac{\delta\sigma_t}{\delta m}(x,\mu,m,a)
 =C_t\left(a-\int_A a'\,m(\d a')\right),
\]
whose iterated normalized flat derivative does not vanish pointwise. Thus the
affine-remainder condition is both intrinsic and compatible with the standard
control-dependent volatility appearing e.g. in relaxed LQ models.
\end{Remark}

\begin{Remark}

Moving from standard stochastic control to the MKV framework introduces the state law $\mu$ as a fundamentally new, dynamically evolving variable alongside the state $x$ and the relaxed control $m$. To manage this complexity, our assumptions deliberately decouple how we measure variations in these different spaces. We pair the state law $\mu$ with the Wasserstein metric $W_2$, and the control measure $m$ with the Bregman divergence $D_h$.

Structuring the bounds this way is essential for the proofs. When the algorithm updates a policy, it triggers two distinct effects: a direct change in the chosen actions (the control $m$) and an indirect effect that shifts the entire population distribution (the law $\mu$). By isolating these variables with their respective natural metrics, we can decouple the errors in our convergence analysis. Specifically, the $D_h$ terms seamlessly feed into the standard regret bounds of the Mirror Descent algorithm, while the $W_2$ terms allow us to use standard SDE stability arguments (such as Gr\"onwall's inequality) to bound the drifting of the population law.

\end{Remark}

Under the above assumptions, we establish a performance difference bound for any two admissible controls $\pi$ and $\pi'$, in which the cost difference is controlled by the first-order variation of the cost functional and a Bregman divergence. More precisely, the first-order variation of $J^\tau$ at $\pi$ in the direction $\pi'-\pi$ is given by
$$
\d J^\tau(\pi)(\pi'-\pi):=\E\Big[\int_0^T\int_A
\frac{\delta H^{\tau}_t}{\delta m}
\bigl(\Theta_t^\pi,\pi_t,a\bigr)
(\pi'_t-\pi_t)(\d a), \d t
\Big].
$$
The following result shows that the remainder is controlled by the Bregman divergence generated by $h$. This corresponds to the notion of $L$-smoothness of $J^\tau$ relative to the reference function $h$, introduced in the static setting in \cite{NEURIPS2022_6e3daaec} and subsequently considered in a dynamic setting in \cite{KERIMKULOV2025104765}.

\begin{Theorem}[Relative smoothness of $J^{\tau}$]\label{thm:relsmooth}
Suppose Assumptions~\ref{ass:H1},~\ref{ass:H2},~\ref{ass:deriv},
\ref{ass:measderiv} and~\ref{ass:relative-smoothness-regularity} hold. Assume,
in addition, that one of the following two alternative regimes is satisfied:
\begin{enumerate}
\item[(i)] Assumption~\ref{ass:H3} holds;
\item[(ii)] Assumptions~\ref{ass:H3gen} and~\ref{ass:affine-dynamics} hold,
and, for every $(t,x,\mu)$, the map $m\mapsto b_t(x,\mu,m)$ is affine on
$\mathscr C$; equivalently,
\[
\mathcal R_t^b(x,\mu;m'\mid m)=0,
\qquad m,m'\in\mathscr C.
\]
\end{enumerate}
Let $\tau\ge0$. Then there exists $L\ge \tau$ such that for all
$\pi,\pi'\in\mathcal{A}_{\mathscr{C}}$,
\[
  J^{\tau}(\pi')-J^{\tau}(\pi)
   \le \E\!\int_0^T\!\int_A\frac{\delta H^{\tau}_t}{\delta m}\bigl(\Theta_t^\pi,\pi_t,a\bigr)(\pi'_t-\pi_t)(\d a)\,\d t
   +L\, \mathcal{D}_h(\pi'\mid \pi) .
\]
\end{Theorem}

The proof of Theorem~\ref{thm:relsmooth} is given in
Section~\ref{sec:proof-relsmooth}. The constant $L$ depends on the
regularisation parameter $\tau$, the time horizon $T$, the initial condition,
the dimension $d'$ of the Brownian motion, and the constants appearing in the
assumptions of the applicable regime, but not on $\pi$ or $\pi'$.
The relative smoothness estimate established above provides the key descent property for the mirror descent scheme. In particular, when the step-size parameter is chosen sufficiently large, the cost functional is non-increasing along the sequence of iterates generated by Algorithm 1. More precisely, we have the following energy dissipation result.

\begin{Theorem}[Energy dissipation]\label{thm:dissipation}
Suppose the assumptions of Theorem~\ref{thm:relsmooth} and
Assumption~\ref{ass:iterates} hold, and let $\tau\geq0$. Let
$\lambda\geq L$, where $L$ is the constant from Theorem~\ref{thm:relsmooth},
and let $(\pi^n)_{n\in\N}$ be the sequence generated by Algorithm 1. Then,
for every integer $n$,
$$
J^\tau(\pi^{n+1})-J^\tau(\pi^n)
\leq
-(\lambda-L) 
\mathcal{D}_h(\pi^{n+1} \mid \pi^n).
$$
In particular, the sequence $(J^\tau(\pi^n))_{n\in\N}$ is non-increasing.
If, in addition,
$$
\inf_{\pi\in\mathcal A_{\mathscr C}}J^\tau(\pi)>-\infty,
$$
then there exists $J^\tau_\infty\in\mathbb{R}$ such that
$$
J^\tau(\pi^n)\longrightarrow J^\tau_\infty.
$$
Moreover, if $\lambda>L$, then
$$
\sum_{n=0}^{\infty}
\mathcal{D}_h(\pi^{n+1} \mid \pi^n)
<\infty.
$$

\end{Theorem}

The proof of Theorem~\ref{thm:dissipation} is given in Section~\ref{sec:proof-dissipation}. The result has two noteworthy features. First, the descent property does not require convexity of the control problem; in particular, the cost functional $J^\tau$ may be non-convex. Secondly, the regularisation parameter $\tau$ may be equal to zero. Thus, the energy dissipation estimate applies to both the regularised and unregularised problems.

On the other hand, monotonicity of the cost along the iterates does not, by
itself, imply convergence to an optimal control. To obtain a quantitative
convergence result, we impose below a relative semiconvexity condition: the
negative part of the first-order remainder of the unregularised objective must
be controlled by the Bregman divergence. 
%Once the regularisation compensates
%for this negative curvature, the Pontryagin first-order condition becomes
%sufficient for global optimality.

For square-integrable random variables $X,X'$, set
$\mu=\mathcal L(X)$ and $\mu'=\mathcal L(X')$, and let
$(\widetilde X,\widetilde X')$ be an independent copy of $(X,X')$. Define the
averaged first-order remainder of the terminal cost by
\begin{align}
\mathfrak R_g(X'\mid X)
:={}&\E\Big[
g(X',\mu')-g(X,\mu)
-\big\langle (D_xg)(X,\mu),X'-X\big\rangle
\nonumber\\
&\hspace{15mm}
-\widetilde\E\big[
\big\langle
(\partial_\mu g)(X,\mu)(\widetilde X),
\widetilde X'-\widetilde X
\big\rangle
\big]
\Big].
\label{eq:terminal-remainder}
\end{align}
For $\pi,\pi'\in\mathcal A_{\mathscr C}$, write
$\mu_t^\pi=\mathcal L(X_t^\pi)$ and let
$(\widetilde X^\pi,\widetilde X^{\pi'})$ be an independent copy of
$(X^\pi,X^{\pi'})$. We define the averaged Hamiltonian remainder along the
forward--backward system associated with $\pi$ by
\begin{align}
&\mathfrak R_{H,t}^\pi(\pi'\mid\pi)
\nonumber\\
:={}&\E\Big[
H_t^0(X_t^{\pi'},\mu_t^{\pi'},Y_t^\pi,Z_t^\pi,\pi_t')
-H_t^0(X_t^\pi,\mu_t^\pi,Y_t^\pi,Z_t^\pi,\pi_t)
\nonumber\\
&\quad
-\big\langle
(D_xH_t^0)(\Theta_t^\pi,\pi_t),
X_t^{\pi'}-X_t^\pi
\big\rangle
\nonumber\\
&\quad
-\widetilde\E\Big[
\big\langle
(\partial_\mu H_t^0)
(X_t^\pi,\mu_t^\pi,Y_t^\pi,Z_t^\pi,\pi_t)
(\widetilde X_t^\pi),
\widetilde X_t^{\pi'}-\widetilde X_t^\pi
\big\rangle
\Big]
\nonumber\\
&\quad
-\int_A
\frac{\delta H_t^0}{\delta m}
(\Theta_t^\pi,\pi_t,a)
(\pi_t'-\pi_t)(\d a)
\Big].
\label{eq:Hamiltonian-remainder}
\end{align}

The next condition is designed to remove a structural obstruction in the
classical stochastic-control analysis of~\cite{KERIMKULOV2025104765}.
Pointwise convexity of the Hamiltonian, when required for every adjoint pair
$(y,z)$, suppresses the first-order remainders of the drift and diffusion and
therefore essentially forces affine dynamics along admissible directions.
In the MKV setting this is unnecessarily restrictive: curvature generated by
the dynamics may be compensated, after expectation and time integration, by
the running and terminal costs. We therefore impose a lower bound only on
the aggregate remainder evaluated along admissible forward--backward
trajectories. This trajectory-wise formulation is the key new convexity
mechanism of the paper; its verifiable affine and LQ reductions are developed
below.

\begin{Assumption}[Relative semiconvexity along admissible trajectories]
\label{ass:relative-semiconvexity}
There exists a constant $\rho\geq0$ such that, for every
$\pi,\pi'\in\mathcal A_{\mathscr C}$ with
$\mathcal D_h(\pi'\mid\pi)<\infty$,
\begin{equation}\label{eq:trajectory-semiconvexity}
\mathfrak R_g(X_T^{\pi'}\mid X_T^\pi)
+\int_0^T\mathfrak R_{H,t}^\pi(\pi'\mid\pi)\,\d t
\geq-\rho\mathcal D_h(\pi'\mid\pi).
\end{equation}
\end{Assumption}

The constant $\rho$ measures the negative curvature of the unregularised
functional relative to the geometry generated by $h$. The performance
difference identity established in the proof of Theorem~\ref{thm:relconvex}
shows that Assumption~\ref{ass:relative-semiconvexity} is equivalent to
\[
J^0(\pi')-J^0(\pi)
\geq \d J^0(\pi)(\pi'-\pi)
-\rho\mathcal D_h(\pi'\mid\pi).
\]
The usefulness of the assumption therefore rests on tractable structural
criteria for the aggregate remainder. This is developed below for affine
dynamics and quadratic costs: Proposition~\ref{prop:lq-semiconvexity}
reduces the condition to an explicit mean--fluctuation quadratic form,
Corollary~\ref{cor:lq-positivity} recovers a broad nonnegative-curvature LQ
regime, and Example~\ref{ex:procyclical-systemic-risk} exhibits a financially
motivated case with a computable constant $\rho>0$.

The standard pointwise convexity condition on $g$ and $H^0_t(., ., y, z,.)$ (see e.g. \cite{KERIMKULOV2025104765}) is a sufficient, but substantially stronger, criterion for
Assumption~\ref{ass:relative-semiconvexity} with $\rho=0$.

\begin{Proposition}[Pointwise convexity criterion]
\label{prop:pointwise-convexity}
Suppose that $g$ is convex in the sense that, for all $x,x'$, $\mu,\mu'$ and
all couplings $(\xi,\xi')$ with respective marginals $\mu,\mu'$,
\begin{align*}
g(x',\mu')-g(x,\mu)
&\geq
\langle(D_xg)(x,\mu),x'-x\rangle
+\widetilde\E\big[
\langle(\partial_\mu g)(x,\mu)(\xi),\xi'-\xi\rangle
\big].
\end{align*}
Suppose also that, for every $(t,y,z)$, the map
$(x,\mu,m)\mapsto H_t^0(x,\mu,y,z,m)$ satisfies the analogous first-order
convexity inequality in $(x,\mu,m)$. Then
Assumption~\ref{ass:relative-semiconvexity} holds with $\rho=0$.
\end{Proposition}

\begin{proof}
Apply the terminal convexity inequality to
$(X_T^\pi,X_T^{\pi'})$. At each time $t$, apply the Hamiltonian convexity
inequality with $(y,z)=(Y_t^\pi,Z_t^\pi)$ and with base point
$(X_t^\pi,\mu_t^\pi,\pi_t)$. Taking expectations and integrating in time
gives
\[
\mathfrak R_g(X_T^{\pi'}\mid X_T^\pi)
+\int_0^T\mathfrak R_{H,t}^\pi(\pi'\mid\pi)\,\d t\geq0.
\]
\end{proof}

\begin{Remark}
The pointwise convexity criterion of
Proposition~\ref{prop:pointwise-convexity} is especially restrictive because
it is imposed for every $y$ and $z$. To make this explicit, fix $t$, two
triples $(x,\mu,m)$ and $(x',\mu',m')$, and a coupling $(\xi,\xi')$ of
$(\mu,\mu')$. For $\varphi\in\{b,\sigma,f\}$, denote by
\begin{align*}
\mathcal R_t^\varphi
:={}&\varphi_t(x',\mu',m')-\varphi_t(x,\mu,m)
-(D_x\varphi_t)(x,\mu,m)(x'-x)
\\
&-\widetilde\E\left[
(\partial_\mu\varphi_t)(x,\mu,m)(\widetilde\xi)
(\widetilde\xi'-\widetilde\xi)
\right]
-\int_A\frac{\delta\varphi_t}{\delta m}(x,\mu,m,a)
(m'-m)(\d a)
\end{align*}
the full first-order remainder of $\varphi_t$ between these two triples. For
vector- or matrix-valued coefficients, this definition is understood
componentwise, and all products denote the natural contraction of a derivative
with its increment. Since the Hamiltonian is affine in $(y,z)$, its
corresponding remainder is
\[
\mathcal R_t^f
+\langle y,\mathcal R_t^b\rangle
+\langle z,\mathcal R_t^\sigma\rangle_F,
\]
where $\langle\cdot,\cdot\rangle_F$ denotes the Frobenius inner product. The
pointwise convexity assumption on the Hamiltonian requires this expression to be nonnegative for
every $y$ and $z$. Here $\mathcal R_t^f$ is fixed, finite and independent of
$(y,z)$. Taking $z=0$ and
$y=-\lambda\mathcal R_t^b$, with $\lambda>0$, gives
\[
\mathcal R_t^f-\lambda|\mathcal R_t^b|^2\geq0.
\]
Letting $\lambda\to\infty$ shows that $\mathcal R_t^b=0$. Similarly, taking
$y=0$ and $z=-\lambda\mathcal R_t^\sigma$ yields
\[
\mathcal R_t^f-\lambda|\mathcal R_t^\sigma|_F^2\geq0,
\]
and hence $\mathcal R_t^\sigma=0$. Taking $y=z=0$ also shows that
$\mathcal R_t^f\geq0$.

Since the two triples and their coupling were arbitrary, pointwise
Hamiltonian convexity therefore forces the full first-order remainders of
$b_t$ and $\sigma_t$ to vanish in every admissible direction. On a convex
domain, integrating their first derivatives along line segments then shows
that these coefficients are affine along those directions. In contrast,
Assumption~\ref{ass:relative-semiconvexity} only controls the sum of the
remainders along the endogenous trajectories, with
$(y,z)=(Y_t^\pi,Z_t^\pi)$, after expectation and time integration. It thus
allows compensation between the dynamics, the running cost and the terminal
cost without requiring the remainders of $b_t$ and $\sigma_t$ to vanish
pointwise.
\end{Remark}

\begin{Theorem}[Relative semiconvexity and convexification]
\label{thm:relconvex}
Suppose Assumptions~\ref{ass:H1},~\ref{ass:H2},~\ref{ass:deriv},
\ref{ass:measderiv} and~\ref{ass:relative-semiconvexity} hold. Assume, in
addition, that either Assumption~\ref{ass:H3} or
Assumption~\ref{ass:H3gen} holds. Then, for every $\tau\geq0$ and all
$\pi,\pi'\in\mathcal A_{\mathscr C}$ with
$\mathcal D_h(\pi'\mid\pi)<\infty$,
\begin{align}
J^\tau(\pi')-J^\tau(\pi)
\geq{}&
\E\int_0^T\int_A
\frac{\delta H_t^\tau}{\delta m}
(\Theta_t^\pi,\pi_t,a)
(\pi_t'-\pi_t)(\d a)\,\d t
\nonumber\\
&+(\tau-\rho)\mathcal D_h(\pi'\mid\pi).
\label{eq:relative-semiconvexity-Jtau}
\end{align}
In particular, $J^\tau$ is relatively convex when $\tau\geq\rho$, and
strongly relatively convex, with modulus $\tau-\rho$, when
$\tau>\rho$.
\end{Theorem}

The relative smoothness and semiconvexity estimates are complementary. Set
\[
\alpha:=\tau-\rho.
\]
When $\alpha=0$, the recursion gives an $\mathcal O(1/n)$ rate. When
$\alpha>0$, relative strong convexity yields a contraction with modulus
$\alpha$. The case $\rho=0$ recovers the unregularised sublinear rate at
$\tau=0$ and the previously stated regularised exponential rate.

\begin{Theorem}[Sublinear/Exponential convergence]\label{thm:main}
Suppose the assumptions of Theorem~\ref{thm:relsmooth} and
Assumptions~\ref{ass:relative-semiconvexity} and~\ref{ass:iterates} hold. Let
$\tau\geq\rho$, set
$\alpha=\tau-\rho$, let $\lambda\geq L$, and let
$(\pi^n)_{n\in\mathbb N}$ be generated by Algorithm~1.
\begin{enumerate}
\item \textbf{(Sublinear, $\alpha=0$).} For every
$\pi\in\mathcal A_{\mathscr C}$ such that
$\mathcal D_h(\pi\mid\pi^n)<\infty$ for all $n\geq0$,
\[
J^\tau(\pi^n)-J^\tau(\pi)
\leq\frac{\lambda}{n}\mathcal D_h(\pi\mid\pi^0),
\qquad n\geq1.
\]
\item \textbf{(Exponential, $\alpha>0$).} Suppose
$\pi^*\in\arg\min_{\pi\in\mathcal A_{\mathscr C}}J^\tau(\pi)$ and
$\mathcal D_h(\pi^*\mid\pi^n)<\infty$ for all $n\geq0$. Then
\[
\mathcal D_h(\pi^*\mid\pi^n)
\leq
\left(1-\frac{\alpha}{\lambda}\right)^n
\mathcal D_h(\pi^*\mid\pi^0),
\qquad n\geq0,
\]
and
\[
0\leq J^\tau(\pi^n)-J^\tau(\pi^*)
\leq
\lambda\left(1-\frac{\alpha}{\lambda}\right)^n
\mathcal D_h(\pi^*\mid\pi^0),
\qquad n\geq1.
\]
\end{enumerate}
\end{Theorem}

\begin{Remark}[Orbit-wise curvature is sufficient]
The global quantification in
Assumption~\ref{ass:relative-semiconvexity} is convenient for stating a
relative-convexity property of the objective, but it is stronger than what is
used in Theorem~\ref{thm:main}. The proof only applies the remainder lower
bound with the base point $\pi=\pi^n$. For the sublinear estimate, it is
therefore enough that this bound hold for every $n$ and for the comparator
$\pi$ appearing in the conclusion (or for every admissible comparator if the
estimate is to be asserted uniformly over that class). For the exponential
estimate, it is enough that it holds for the pairs $(\pi^n,\pi^\star)$ for all $n\geq0$.
Thus relative semiconvexity is only required along the SMD orbit and in the
directions of the relevant comparators, rather than uniformly for all pairs
in $\mathcal A_{\mathscr C}\times\mathcal A_{\mathscr C}$.
\end{Remark}

%\begin{Remark}
%Lemma~\ref{lem:reg} and the fundamental theorem of calculus yield, under the
%assumptions of Theorem~\ref{thm:relsmooth},
%\[
%\left|
%J^0(\pi')-J^0(\pi)-\d J^0(\pi)(\pi'-\pi)
%\right|
%\leq C\mathcal D_h(\pi'\mid\pi).
%\]
%Thus Assumption~\ref{ass:relative-semiconvexity} is automatically satisfied
%with the generic choice $\rho=C$ in the relative-smoothness framework. Its
%interest is that structural criteria may provide a smaller constant, and in
%particular $\rho=0$. This also shows that sufficiently strong regularisation,
%namely $\tau>C$, convexifies every problem covered by
%Theorem~\ref{thm:relsmooth}.
%\end{Remark}

\subsection{Affine dynamics and linear-quadratic models}\label{sec:affine-lq}

Affine and LQ models provide a natural benchmark for the SMD
algorithm. They retain the genuinely MKV features of the problem,
including interactions through the state law and randomized controls, while
remaining sufficiently tractable to yield explicit optimal policies. They
therefore make it possible to compare the SMD iterates with a known solution,
to distinguish discretisation and regression errors from optimisation error,
and to test the convergence rates of Theorem~\ref{thm:main}. At the same time,
the class is broad enough to include state-control cross terms and
control-dependent volatility, for which pointwise convexity of the
Hamiltonian need not be the appropriate criterion.
More importantly, this subsection makes the trajectory-wise curvature
assumption operational. We first isolate its form for general costs under
affine dynamics and then derive, for quadratic costs, an explicit decomposition
into mean and fluctuation components. The standard LQ setting and the
negative-curvature systemic-risk model are two distinct consequences of this
same reduction.

\paragraph{Affine dynamics.}
For square-integrable $X,X'$ and $\mathscr C$-valued random variables $m,m'$,
define
\begin{align*}
&\mathfrak R_{f,t}\big((X',m')\mid(X,m)\big)
\\
:={}&\E\Big[
f_t(X',\mu',m')-f_t(X,\mu,m)
-\langle(D_xf_t)(X,\mu,m),X'-X\rangle
\\
&\quad
-\widetilde\E\big[
\langle
(\partial_\mu f_t)(X,\mu,m)(\widetilde X),
\widetilde X'-\widetilde X
\rangle
\big]
\\
&\quad
-\int_A\frac{\delta f_t}{\delta m}(X,\mu,m,a)
(m'-m)(\d a)
\Big].
\end{align*}
Remark that if Assumption~\ref{ass:affine-dynamics} holds and, for every $(t,x,\mu)$,
the maps $m\mapsto b_t(x,\mu,m)$ and
$m\mapsto\sigma_t(x,\mu,m)$ are affine on $\mathscr C$, then
\[
\mathfrak R_{H,t}^\pi(\pi'\mid\pi)
=\mathfrak R_{f,t}\big(
(X_t^{\pi'},\pi_t')\mid(X_t^\pi,\pi_t)
\big).
\]
Consequently, Assumption~\ref{ass:relative-semiconvexity} reduces to
\begin{equation}\label{eq:affine-cost-semiconvexity}
\mathfrak R_g(X_T^{\pi'}\mid X_T^\pi)
+\int_0^T
\mathfrak R_{f,t}\big(
(X_t^{\pi'},\pi_t')\mid(X_t^\pi,\pi_t)
\big)\,\d t
\geq-\rho\mathcal D_h(\pi'\mid\pi).
\end{equation}

\paragraph{The LQ benchmark.}
We next make this condition explicit for quadratic costs. Besides providing a
transparent test of Assumption~\ref{ass:relative-semiconvexity}, this
specialisation is the basis of the numerical examples in
Section~\ref{sec:numerical-experiments}. Let
$A\subseteq\mathbb R^k$, and write
\[
\overline\mu=\int_{\mathbb R^d}v\,\mu(\d v),
\qquad
\overline m=\int_Aa\,m(\d a).
\]
Consider the coefficients
\begin{align*}
b_t(x,\mu,m)
&=B_tx+\overline B_t\overline\mu+C_t\overline m,\\
\sigma_t(x,\mu,m)
&=\gamma_t+D_tx+\overline D_t\overline\mu+F_t\overline m,
\end{align*}
and the costs
\begin{align*}
f_t(x,\mu,m)
={}&x^{\trans}Q_tx
+\overline\mu^{\trans}\overline Q_t\overline\mu
+\int_Aa^{\trans}N_ta\,m(\d a)
\\
&+2\overline m^{\trans}I_tx
+2\overline m^{\trans}\overline I_t\overline\mu
+2M_t^{\trans}x+2H_t^{\trans}\overline m,\\
g(x,\mu)
={}&x^{\trans}Px
+\overline\mu^{\trans}\overline P\,\overline\mu
+2L^{\trans}x.
\end{align*}

The time-dependent coefficients are bounded measurable maps with values in
the following spaces:
\begin{align*}
B_t,\overline B_t&\in\mathbb R^{d\times d},
&C_t&\in\mathbb R^{d\times k},
&\gamma_t&\in\mathbb R^{d\times d'},\\
D_t,\overline D_t&\in
\mathcal L(\mathbb R^d,\mathbb R^{d\times d'}),
&F_t&\in\mathcal L(\mathbb R^k,\mathbb R^{d\times d'}),\\
Q_t,\overline Q_t&\in\mathbb S^d,
&N_t&\in\mathbb S^k,
&I_t,\overline I_t&\in\mathbb R^{k\times d},\\
M_t&\in\mathbb R^d,
&H_t&\in\mathbb R^k.
\end{align*}
Recall that $d'$ is the dimension of the Brownian motion,
$\mathcal L(E,F)$ denotes the space of linear maps from $E$ to $F$, and
$\mathbb S^r$ denotes the space of symmetric $r\times r$ matrices. We equip
$\mathbb R^{d\times d'}$ with the Frobenius inner product and denote by
$F_t^*:\mathbb R^{d\times d'}\to\mathbb R^k$ the adjoint of $F_t$,
characterised by
\[
\langle F_ta,Z\rangle_F=a^{\trans}F_t^*Z,
\qquad a\in\mathbb R^k,\quad Z\in\mathbb R^{d\times d'}.
\]
The terminal coefficients satisfy
\[
P,\overline P\in\mathbb S^d,
\qquad L\in\mathbb R^d.
\]

\paragraph{Entropy-regularised solution.}
The key point is that the preceding model gives an explicit description of the optimal
randomised policy. We first clarify the value-function interpretation, which
is slightly different from that of the scalar function appearing in the
quadratic ansatz. For $t\in[0,T]$ and $\mu\in\mathcal P_2(\mathbb R^d)$, let
$J_{t,\mu}^\tau(\pi)$ denote the time-shifted social cost when the initial
state at time $t$ has law $\mu$, and set
\[
\mathcal V^\tau(t,\mu):=\inf_\pi J_{t,\mu}^\tau(\pi).
\]
The scalar function $v(t,x,\mu)$ introduced below is the corresponding
\emph{decoupled value function}: the variable $x$ denotes a tagged initial
state, whereas $\mu$ describes the population law. Consequently, the social
value is obtained by averaging the decoupled value against the initial law,
namely
\begin{equation}\label{eq:social-decoupled-value}
\mathcal V^\tau(t,\mu)
=\int_{\mathbb R^d}v(t,x,\mu)\,\mu(\d x).
\end{equation}
In particular, if $\mu_0=\mathcal L(X_0)$ and $\pi^\star$ is optimal, then
\[
J^\tau(\pi^\star)
=\mathcal V^\tau(0,\mu_0)
=\E\bigl[v(0,X_0,\mu_0)\bigr].
\]

Suppose now that $A=\mathbb R^k$, that $N_t$ is uniformly positive
definite, and take as regulariser the relative entropy with respect to a
fixed probability measure $\varrho$:
\begin{equation}\label{eq:relative:entropy}
h(m)=
\begin{cases}
\displaystyle\int_{\mathbb R^k}
\log\!\left(\frac{\d m}{\d\varrho}(a)\right)m(\d a),
&m\ll\varrho,\\[6pt]
+\infty,&\text{otherwise}.
\end{cases}
\end{equation}
For a general reference measure, the optimal policy is given by the Gibbs
formula below, but it need not be Gaussian and the value function need not
remain quadratic. A finite-dimensional LQ closure is obtained when
$\varrho$ is Gaussian. The negative differential-entropy case is obtained
instead by using Lebesgue measure as the underlying sigma-finite measure.

For completeness, a transparent sufficient condition for the solvability of
the associated Riccati system is that, for some $\delta>0$,
\begin{align}
N_t&\succeq\delta I_k,
&P&\succeq0,
&P+\overline P&\succeq0,
\label{eq:LQ-Riccati-coercivity}\\
Q_t-I_t^{\trans}N_t^{-1}I_t&\succeq0,
&Q_t+\overline Q_t
-(I_t+\overline I_t)^{\trans}N_t^{-1}(I_t+\overline I_t)&\succeq0.
\nonumber
\end{align}

These conditions are sufficient, rather than necessary. They ensure that the
matrix Riccati equations admit global solutions with
$K_t,\Lambda_t\succeq0$ and hence that
$S_t=N_t+F_t^*K_tF_t$ is uniformly positive definite. Under these conditions,
the verification argument gives the decoupled value
\[
v(t,x,\mu)
=(x-\overline\mu)^{\trans}K_t(x-\overline\mu)
+\overline\mu^{\trans}\Lambda_t\overline\mu
+2r_t^{\trans}x+s_t.
\]
The corresponding matrix, vector, and scalar ODEs have terminal conditions
\[
K_T=P,\qquad \Lambda_T=P+\overline P,
\qquad r_T=L,\qquad s_T=0.
\]
We do not reproduce the Riccati system. Its derivation follows the same
completion-of-squares and entropy-minimisation argument as
\cite[Section~4 and Appendix~B]{frikha:germain:lauriere:pham:song}, after
replacing the sampled control in that reference by the policy mean in the
drift and diffusion. Integrating the preceding expression against $\mu$ proves
\eqref{eq:social-decoupled-value} and yields
\[
\mathcal V^\tau(t,\mu)
=\int_{\mathbb R^d}(x-\overline\mu)^{\trans}
K_t(x-\overline\mu)\,\mu(\d x)
+\overline\mu^{\trans}\Lambda_t\overline\mu
+2r_t^{\trans}\overline\mu+s_t.
\]
 Writing
\[
u_t^*:=\int_{\mathbb R^k}a\,\pi_t^*(\d a),
\qquad \overline X_t^*:=\E[X_t^*],
\]
the adjoint processes satisfy
\begin{align*}
Y_t^*
&=2K_t(X_t^*-\overline X_t^*)
+2\Lambda_t\overline X_t^*+2r_t,\\
Z_t^*
&=2K_t\bigl(
\gamma_t+D_tX_t^*+\overline D_t\overline X_t^*+F_tu_t^*
\bigr).
\end{align*}
Introduce
\begin{align*}
S_t&:=N_t+F_t^*K_tF_t,\\
U_t&:=I_t+C_t^{\trans}K_t+F_t^*K_tD_t,\\
\widehat U_t
&:=I_t+\overline I_t+C_t^{\trans}\Lambda_t
+F_t^*K_t(D_t+\overline D_t),\\
O_t&:=H_t+C_t^{\trans}r_t+F_t^*K_t\gamma_t,
\end{align*}
and set, for $x,\overline x\in\mathbb R^d$,
\[
q_t(x,\overline x)
:=U_tx+(\widehat U_t-U_t)\overline x+O_t,
\qquad
q_t^*:=q_t(X_t^*,\overline X_t^*).
\]
Substitution of the preceding expressions for $Y_t^*$ and $Z_t^*$ gives
\begin{equation}\label{eq:LQ-linear-energy}
\ell_t^*
:=C_t^{\trans}Y_t^*+F_t^*Z_t^*
+2I_tX_t^*+2\overline I_t\overline X_t^*+2H_t
=2q_t^*+2(S_t-N_t)u_t^*.
\end{equation}
The part of the Hamiltonian depending on the action is then generated by the
quadratic energy
\[
\mathcal E_t^*(a)=a^{\trans}N_ta+a^{\trans}\ell_t^*.
\]
More precisely, the LQ verification argument reduces the optimisation over
the control measure, pointwise in $(t,x,\overline x)$, to the minimisation of
\begin{align*}
\Phi_t(m;x,\overline x)
:={}&\int_{\mathbb R^k}a^{\trans}N_ta\,m(\d a)
+2q_t(x,\overline x)^{\trans}\overline m
+\overline m^{\trans}(S_t-N_t)\overline m
+\tau h(m),
\end{align*}
where $\overline m=\int_{\mathbb R^k}a\,m(\d a)$. Its first-order condition
at $m=\pi_t^*$ involves the energy $\mathcal E_t^*$ above. Since $N_t$ is
positive definite, the functional is coercive. Moreover, $K_t$ is positive
semidefinite under the stated Riccati solvability conditions, so that
$S_t-N_t=F_t^*K_tF_t$ is positive semidefinite. Since the relative entropy is
strictly convex, $\Phi_t$ is strictly convex for $\tau>0$, and its unique
minimiser is the Gibbs policy
\begin{equation}\label{eq:LQ-Gibbs-policy}
\pi_t^*(\d a)
=\frac{
\exp\bigl(-\mathcal E_t^*(a)/\tau\bigr)
}{
\int_{\mathbb R^k}\exp\bigl(-\mathcal E_t^*(a')/\tau\bigr)\varrho(\d a')
}\,\varrho(\d a),
\qquad \tau>0.
\end{equation}
In particular, if $\varrho=\mathcal N(m_\varrho,\Sigma_\varrho)$, the Gibbs
formula first expresses the mean $u_t^*$ in terms of $\ell_t^*$ as
\[
u_t^*
=\left(\Sigma_\varrho^{-1}+\frac{2}{\tau}N_t\right)^{-1}
\left(\Sigma_\varrho^{-1}m_\varrho-\frac{1}{\tau}\ell_t^*\right).
\]
Combining this identity with~\eqref{eq:LQ-linear-energy} gives
\begin{align}\label{eq:LQ-Gaussian-policy}
\pi_t^*&=\mathcal N(u_t^*,\Gamma_t^*),
\end{align}
\noindent with
\begin{align}
u_t^*
&=\left(\Sigma_\varrho^{-1}+\frac{2}{\tau}S_t\right)^{-1}
\left(\Sigma_\varrho^{-1}m_\varrho-\frac{2}{\tau}q_t^*\right), \qquad 
\Gamma_t^*
=\left(\Sigma_\varrho^{-1}+\frac{2}{\tau}N_t\right)^{-1}.
\nonumber\\
\end{align}
If $\varrho$ is the Lebesgue measure, then $h$ is the usual negative
differential entropy,
\[
h(m)=\int_{\mathbb R^k}p(a)\log p(a)\,\d a, \quad m(\d a)=p(a)\,\d a.
\]
In this case, the Gibbs formula gives
\begin{equation}\label{eq:LQ-Gaussian-policy-lebesgue}
\pi_t^*
=\mathcal N\left(
-S_t^{-1}q_t^*,
\frac{\tau}{2}N_t^{-1}
\right).
\end{equation}
Thus the optimal mean is an affine feedback of
$(X_t^*,\overline X_t^*)$ and the covariance is deterministic. The same
Gaussian-feedback mechanism for entropy-regularised LQ control is derived in
\cite{wanzarzhou20} and, in the mean-field setting, in
\cite[Section~4 and Theorem~B.1]{frikha:germain:lauriere:pham:song}. This
finite-dimensional representation makes the model particularly convenient
for testing SMD: with a Gaussian reference and Gaussian initialization, the
iterates can be tracked through their means and covariance matrices.

\paragraph{Aggregate curvature.}
We finally identify the relative-semiconvexity condition in this model. For
$\pi,\pi'\in\mathcal A_{\mathscr C}$, set
\begin{align*}
\delta X_t&:=X_t^{\pi'}-X_t^\pi,
&\delta\overline X_t&:=\E[\delta X_t],\\
u_t^\pi&:=\int_Aa\,\pi_t(\d a),
&\delta u_t&:=u_t^{\pi'}-u_t^\pi,
\qquad
\delta\overline u_t:=\E[\delta u_t].
\end{align*}
We further introduce the centered increments
\[
\widehat{\delta X}_t:=\delta X_t-\delta\overline X_t,
\qquad
\widehat{\delta u}_t:=\delta u_t-\delta\overline u_t.
\]

\begin{Proposition}[Mean-fluctuation LQ reduction]\label{prop:lq-semiconvexity}
For the preceding LQ model, it holds
\begin{align}
&\mathfrak R_g(X_T^{\pi'}\mid X_T^\pi)
+\int_0^T\mathfrak R_{H,t}^\pi(\pi'\mid\pi)\,\d t
=\mathcal Q_{\mathrm{LQ}}(\pi',\pi),
\label{eq:LQ-remainder}
\end{align}
\noindent where
\begin{align}
\mathcal Q_{\mathrm{LQ}}(\pi',\pi)
:={}&
\E[\widehat{\delta X}_T^{\trans}P\widehat{\delta X}_T]
+\delta\overline X_T^{\trans}(P+\overline P)\delta\overline X_T
\nonumber\\
&+\int_0^T\Big(
\E[\widehat{\delta X}_t^{\trans}Q_t\widehat{\delta X}_t
+2\widehat{\delta u}_t^{\trans}I_t\widehat{\delta X}_t]
\label{eq:LQ-polynomial}\\
&\hspace{24mm}
+\delta\overline X_t^{\trans}(Q_t+\overline Q_t)\delta\overline X_t
+2\delta\overline u_t^{\trans}(I_t+\overline I_t)
\delta\overline X_t
\Big)\,\d t.
\nonumber
\end{align}
Hence Assumption~\ref{ass:relative-semiconvexity} is equivalent in this case
to
\[
\mathcal Q_{\mathrm{LQ}}(\pi',\pi)
\geq-\rho\mathcal D_h(\pi'\mid\pi),
\qquad \pi,\pi'\in\mathcal A_{\mathscr C}.
\]
In particular, positivity of the aggregate quadratic form
$\mathcal Q_{\mathrm{LQ}}$ gives $\rho=0$.
\end{Proposition}

\begin{proof}
The drift and diffusion are affine in $(x,\overline\mu,m)$, so their
first-order remainders vanish. The affine terms in the running and terminal
costs cancel from their first-order remainders. The same is true of the
control cost because the map
$m\mapsto\int_Aa^{\trans}N_ta\,m(\d a)$ is affine on the space of probability
measures.

It remains to decompose the quadratic terms. For every symmetric matrix $M$,
the centered increments satisfy
\[
\E[\delta X_t^{\trans}M\delta X_t]
=\E[\widehat{\delta X}_t^{\trans}M\widehat{\delta X}_t]
+\delta\overline X_t^{\trans}M\delta\overline X_t,
\]
because $\E[\widehat{\delta X}_t]=0$. Similarly,
\[
\E[\delta u_t^{\trans}I_t\delta X_t]
=\E[\widehat{\delta u}_t^{\trans}I_t\widehat{\delta X}_t]
+\delta\overline u_t^{\trans}I_t\delta\overline X_t.
\]
Applying the first identity with $M=P$ at time $T$ and with $M=Q_t$ at
time $t$, and then combining the resulting mean terms with those involving
$\overline P$, $\overline Q_t$, and $\overline I_t$, yields
\eqref{eq:LQ-remainder}.
\end{proof}

The proof of the following elementary sufficient condition is immediate from
Proposition~\ref{prop:lq-semiconvexity}.

\begin{Corollary}[Elementary LQ positivity criterion]
\label{cor:lq-positivity}
Suppose that
\[
P\succeq0,
\qquad P+\overline P\succeq0,
\]
and, for almost every $t\in[0,T]$,
\[
Q_t\succeq0,
\qquad Q_t+\overline Q_t\succeq0,
\qquad I_t=0,
\qquad I_t+\overline I_t=0.
\]
Then $\mathcal Q_{\mathrm{LQ}}(\pi',\pi)\geq0$ for every admissible pair
$(\pi,\pi')$. Consequently, Assumption~\ref{ass:relative-semiconvexity}
holds with $\rho=0$. Notice that the conditions do not require
$\overline P\succeq0$ or $\overline Q_t\succeq0$. Thus the criterion may hold
even when the terminal cost or the Hamiltonian is not pointwise convex in the
measure argument; this is the case, for instance, when
$\overline P=-P$ and $\overline Q_t=-Q_t$.
\end{Corollary}

\begin{Remark}
The absence of $N_t$ from $\mathcal Q_{\mathrm{LQ}}$ is specific to relaxed
controls: although $a\mapsto a^{\trans}N_ta$ is quadratic, its integral is
affine in the probability measure $m$. Pointwise joint convexity in $(x,m)$
would therefore typically force the cross matrices $I_t$ and
$\overline I_t$ to vanish. Condition~\eqref{eq:LQ-remainder} is weaker because
it only requires positivity, or a controlled negative part, of the aggregate
form along state variations generated by admissible controls. In particular,
Corollary~\ref{cor:lq-positivity} is only a sufficient criterion: nonzero
cross terms may still give a nonnegative aggregate form once the dynamics of
the state variation are taken into account.
\end{Remark}

\begin{Example}[Procyclical funding and negative curvature]
\label{ex:procyclical-systemic-risk}
Consider the scalar systemic-risk model
\begin{align*}
\d X_t
&=\bigl[\beta(\E[X_t]-X_t)+\overline a(\pi_t)\bigr]\,\d t
+\gamma\,\d W_t,\\
f(x,\mu,m)
&=Q(x^2-\overline\mu^2)+\frac12\int_{\mathbb R}a^2m(\d a)
+2I(x-\overline\mu)\overline a(m),\\
g(x,\mu)&=P(x^2-\overline\mu^2).
\end{align*}
The usual stabilising specification has $I>0$: a bank below the average
reserve level is encouraged to borrow, whereas a bank above the average
level lends. A negative value of $I$ models a procyclical funding channel.
Better-capitalised banks then expand their balance sheets further, while
weaker banks deleverage because of tighter funding or margin constraints.
Such a mechanism amplifies cross-sectional dispersion and is a genuine
source of negative curvature for the social planner.

For two controls, let
$\xi_t=\widehat{\delta X}_t$ and
$v_t=\widehat{\delta u}_t$. Then
\[
\d\xi_t=(-\beta\xi_t+v_t)\,\d t,
\qquad \xi_0=0,
\]
and Proposition~\ref{prop:lq-semiconvexity} reduces to
\begin{equation}\label{eq:procyclical-curvature}
\mathcal Q_{\mathrm{LQ}}(\pi',\pi)
=(P+I)\E|\xi_T|^2
+(Q+2\beta I)\E\int_0^T|\xi_t|^2\,\d t.
\end{equation}
Suppose $P+I\geq0$, $Q+2\beta I<0$, and that the Gaussian covariance of the
base control is bounded by $\overline\Gamma$. Since
\[
\int_0^T\E|\xi_t|^2\,\d t
\leq c_{\beta,T}\E\int_0^T|v_t|^2\,\d t,
\qquad
c_{\beta,T}:=\left(\frac{1-e^{-\beta T}}{\beta}\right)^2,
\]
and, writing $\delta u_t=u_t^{\pi'}-u_t^\pi$, the Gaussian
relative-entropy formula gives
\[
\E|v_t|^2
=\operatorname{Var}(\delta u_t)
\leq\E|\delta u_t|^2
\leq 2\overline\Gamma\,
\E\bigl[D_{\mathrm{KL}}(\pi_t'\mid\pi_t)\bigr],
\]
it follows that
\begin{equation}\label{eq:procyclical-rho}
\mathcal Q_{\mathrm{LQ}}(\pi',\pi)
\geq-\rho_{\rm pc}\mathcal D_h(\pi'\mid\pi),
\qquad
\rho_{\rm pc}
=2\overline\Gamma\,[-(Q+2\beta I)]c_{\beta,T}>0.
\end{equation}
Thus entropy regularisation convexifies this financially motivated
nonconvex model as soon as $\tau\geq\rho_{\rm pc}$.
\end{Example}

\subsection{Examples of regularisers}\label{sec:examples-h}

The purpose of this subsection is to verify that Assumption~\ref{ass:iterates}
holds, that is, the mirror step is closed in the domain of the chosen
regulariser. The proofs of Propositions~\ref{prop:KL}--\ref{prop:eot} and Corollary~\ref{cor:gaussian-lq-smd} are
collected in Section~\ref{sec:proofs-regularisers}. When $A$ is a compact metric space, $C(A)$ denotes the Banach
space of real-valued continuous functions on $A$, equipped with the supremum
norm. The relevant quantity is the action-dependent stochastic gradient
\[
a \mapsto \mathscr G_t^\pi(a)
:=\frac{\delta H_t^0}{\delta m}
  (\Theta_t^\pi,\pi_t,a).
\]
Its required integrability depends on the dual geometry of the regulariser:
$L^\infty(A)$ for relative entropy under bounded derivatives, $L^2(\varrho)$
for the $\chi^2$ geometry, and $C(A)$ for entropic optimal transport on a
compact action space. Proposition~\ref{prop:wellposed} supplies the required
integrability of $(X^\pi,Y^\pi,Z^\pi)$ in both the bounded-derivative and the
linearly growing regimes. The genuine LQ framework deserves separate treatment:
although $\mathscr G_t^\pi(a)$ is quadratic in $a$ and hence unbounded, the
entropy mirror step is explicit and preserves Gaussian policies.

\begin{Proposition}[Entropy regularisation]\label{prop:KL}
The following three regimes guarantee well-posedness of the entropy mirror
step.
\begin{enumerate}
\item[(i)] Fix $\varrho\in\mathcal P(A)$ and let
\[
\mathscr C_\infty(\varrho)
:=\left\{m\ll\varrho:
\left\|\log\frac{\d m}{\d\varrho}\right\|_{L^\infty(A)}<\infty\right\}.
\]
Let $h$ be given by \eqref{eq:relative:entropy}. Take $\mathscr C = \mathscr C_\infty(\varrho)$, and $\nu_0=\varrho$ in the definition of
$\mathcal A_{\mathscr C}$. Suppose
Assumptions~\ref{ass:H1}--\ref{ass:H3} hold and
\[
  K_\delta:=\sup_{t,x,\mu,m,a}\Bigl(\bigl|\tfrac{\delta b_t}{\delta m}(x,\mu,m,a)\bigr|
   +\bigl|\tfrac{\delta\sigma_t}{\delta m}(x,\mu,m,a)\bigr|
   +\bigl|\tfrac{\delta f_t}{\delta m}(x,\mu,m,a)\bigr|\Bigr)<\infty.
\]
Then Assumption~\ref{ass:iterates} holds for every $\pi^0\in\AC$ with
$\E\int_0^T\|\ln\tfrac{\d\pi^0_t}{\d\varrho}\|_{L^\infty(A)}\,\d t<\infty$, and moreover
$\E\int_0^T\|\ln\tfrac{\d\pi^n_t}{\d\varrho}\|_{L^\infty(A)}\,\d t<\infty$ for all $n\ge 0$.
\item[(ii)] Let $A=\mathbb R^k$, fix a full-support measure
$\varrho\in\mathcal P_2(\mathbb R^k)$, and set
\[
\mathscr C_{\rm qent}(\varrho)
:=\left\{m=f\varrho\in\mathcal P_2(\mathbb R^k):
f>0\ \varrho\text{-a.e.},\quad
\int_{\mathbb R^k}f(a)|\log f(a)|\,\varrho(\d a)<\infty\right\}.
\]
Let $h$ be given by \eqref{eq:relative:entropy}. Take $\mathscr C=\mathscr C_{\rm qent}(\varrho)$, and $\nu_0=\varrho$ in the definition of
$\mathcal A_{\mathscr C}$. Suppose Assumptions~\ref{ass:H1},
\ref{ass:H2}, \ref{ass:H3gen}, and~\ref{ass:affine-dynamics} hold. Assume
that, for every $\pi\in\mathcal A_{\mathscr C}$, the stochastic action gradient satisfies, up to an $a$-independent progressively measurable term,
\begin{equation}\label{eq:quadratic-action-gradient}
\kappa^{-1} |a|^2-\zeta_t^\pi
\leq \mathscr G_t^\pi(a)
\leq \kappa |a|^2+\zeta_t^\pi,
\qquad a\in\mathbb R^k,
\end{equation}
for a constant $\kappa\geq 1$ and a nonnegative process satisfying
$\E\int_0^T\zeta_t^\pi\,\d t<\infty$. If $\lambda>\tau>0$, then
Assumption~\ref{ass:iterates} holds for every
$\pi^0\in\mathcal A_{\mathscr C}$.

\item[(iii)] Let $A=\mathbb R^k$ and let $h$ be the negative differential
entropy
\[
h(m)=\int_{\mathbb R^k}p(a)\log p(a)\,\d a,
\qquad m(\d a)=p(a)\,\d a,
\]
on the convex domain
\[
\mathscr C_{\rm ent}^{,2}
:=\left\{m=p\,\mathcal L^k\in\mathcal P_2(\mathbb R^k):
p>0\ \mathcal L^k\text{-a.e.},\quad
\int_{\mathbb R^k}p(a)|\log p(a)|\,\d a<\infty\right\}.
\]
Thus, in this regime, $\mathscr C=\mathscr C_{\rm ent}^{,2}$. Assume that we are in
the affine-LQ setting of Section~\ref{sec:affine-lq}, fix a nondegenerate
Gaussian measure $\nu_0\in\mathscr C_{\rm ent}^{,2}$, and use it as the
reference measure in the definition of $\mathcal A_{\mathscr C}$. Under the Riccati
solvability conditions stated there and when $\tau>0$, the optimal policy
$\pi^\star$ is given by~\eqref{eq:LQ-Gaussian-policy-lebesgue}, with
Lebesgue measure as the underlying sigma-finite reference measure, and is
Gaussian. The Gaussian
measures nevertheless form a nonconvex subset of
$\mathscr C_{\rm ent}^{,2}$: Gaussianity is an invariant property of the SMD
orbit, not a restriction of the minimisation domain.

Assume that
$N_t\succeq\nu I_k$ for some $\nu>0$, and let $\lambda\geq\tau$. Suppose that for some constant $\gamma\geq 1$
\[
\pi_t^0=\mathcal N(u_t^0,\Gamma_t^0), \qquad
0< \gamma^{-1} I_k\preceq\Gamma_t^0
\preceq \gamma I_k, \qquad
\E\int_0^T|u_t^0|^2\,\d t<\infty.
\]
Then every SMD iterate is a nondegenerate Gaussian policy
$\pi_t^n=\mathcal N(u_t^n,\Gamma_t^n)$. Writing
$R_t^n=(\Gamma_t^n)^{-1}$, $\eta_t^n=R_t^nu_t^n$, and
\[
\ell_t^n=C_t^{\trans}Y_t^{\pi^n}+F_t^*Z_t^{\pi^n}
+2I_tX_t^{\pi^n}+2\overline I_t\E[X_t^{\pi^n}]+2H_t,
\]
the SMD update is
\begin{align}
R_t^{n+1}
&=\left(1-\frac{\tau}{\lambda}\right)R_t^n
+\frac{2}{\lambda}N_t,
&
\eta_t^{n+1}
&=\left(1-\frac{\tau}{\lambda}\right)\eta_t^n
-\frac1\lambda\ell_t^n,
\label{eq:abstract-gaussian-update}
\end{align}
so that $\Gamma_t^{n+1}=(R_t^{n+1})^{-1}$ and
$u_t^{n+1}=\Gamma_t^{n+1}\eta_t^{n+1}$. In particular,
Assumption~\ref{ass:iterates} holds and
$\mathcal D_h(\pi^{n+1}\mid\pi^n)<\infty$ for all $n$.
\end{enumerate}
\end{Proposition}

\begin{Remark}[From entropy to relative smoothness]
In regime (i), Pinsker's inequality converts total-variation estimates into
relative-entropy estimates. Thus bounded flat derivatives of $b$, $\sigma$,
and $f$, together with the analogous bounds for the derivatives entering
Assumptions~\ref{ass:deriv} and~\ref{ass:measderiv}, provide convenient
sufficient conditions for the control-measure estimates used in the relative
smoothness theorem. Boundedness of the first flat derivatives alone does not,
however, imply relative convexity: the lower bound on the first-order
remainder in Assumption~\ref{ass:relative-semiconvexity} remains a separate
requirement.

The situation is genuinely different in regime (ii). The convex set
$\mathscr C_{\rm qent}(\varrho)$ is chosen to make the entropic mirror step
well posed and stable under the update. Its defining conditions, namely a
finite second moment and finite relative entropy with respect to $\varrho$,
do not imply a quadratic transport-entropy inequality, let alone one with a
constant uniform over $\mathscr C_{\rm qent}(\varrho)$.
% Consequently, the
%relative-smoothness estimates cannot in general be deduced from the choice of
%this domain, even when the coefficient remainders are controlled by $W_2^2$.

%Such a deduction becomes possible only under the additional orbit-wise
%assumption that, uniformly in $(n,t)$,
%\[
%W_2^2(m,\pi_t^n)
%\leq 2C_{\rm T}D_{\rm KL}(m\mid\pi_t^n),
%\qquad m\in\mathscr C_{\rm qent}(\varrho),
%\]
%or under structural conditions implying this estimate, such as suitable
%uniform log-concavity of the iterates. Wasserstein estimates then yield the
%corresponding relative-entropy estimates along the SMD orbit. 

Thus Proposition~\ref{prop:KL}(ii) verifies the well-posedness requirement in
Assumption~\ref{ass:iterates}, but does not by itself place that regime under
the relative-smoothness theorem. As in regime (i), relative convexity still
requires a separate lower bound on the aggregate remainder.
\end{Remark}

\begin{Remark}[Limitations of regime (ii)]
The condition $\lambda>\tau$ excludes the endpoint $\lambda=\tau>0$. At that
endpoint the new density no longer contains the preceding density, and
$D_{\rm KL}(\pi^{n+1}\mid\pi^n)$ may be infinite unless an additional lower
tail bound is imposed on $\pi^n$. The unregularised endpoint $\tau=0$ is
also excluded: the variational estimate still controls the second moment and
$D_{\rm KL}(\pi^{n+1}\mid\pi^n)$, but it does not control
$D_{\rm KL}(\pi^{n+1}\mid\varrho)$ without an additional entropy-stability
condition on the exponential tilt. This endpoint remains covered in the
bounded regime (i) and in the Gaussian affine-LQ regime (iii). Moreover,
regime (ii) does not yield a
finite-dimensional parametrisation as in regime (iii) which avoids the first issue
through the explicit Gaussian structure and provides such a parametrisation.
\end{Remark}

\medskip

\begin{Proposition}[$\chi^2$-divergence]\label{prop:chi2}
Let $\varrho\in\mathcal P(A)$, set $\nu_0=\varrho$ and
$\mathscr C_{2}(\varrho)=\{m\in\mathcal{P}(A):m\ll\varrho,\
\tfrac{\d m}{\d\varrho}\in L^2_\varrho(A)\}$ and
set $\mathscr C=\mathscr C_2(\varrho)$. Define
$h(m)=\int_A\tfrac12\bigl(\tfrac{\d m}{\d\varrho}(a)-1\bigr)^2\varrho(\d a)$, so that
$D_h(m'\mid m)=\tfrac12\|\tfrac{\d m'}{\d\varrho}-\tfrac{\d m}{\d\varrho}\|_{L^2_\varrho(A)}^2$.
Suppose that one of the following two regimes holds.
\begin{enumerate}
\item[(i)] Assumptions~\ref{ass:H1}, \ref{ass:H3} hold and, for each
$\varphi\in\{b,\sigma,f\}$,
\[
  a\longmapsto\sup_{t,x,\mu,m}\bigl|\tfrac{\delta\varphi_t}{\delta m}(x,\mu,m,a)\bigr|\in L^2_\varrho(A).
\]
\item[(ii)] $A=\mathbb R^k$, $\varrho\in\mathcal P_4(\mathbb R^k)$,
Assumptions~\ref{ass:H1}, \ref{ass:H2}, \ref{ass:H3gen},
and~\ref{ass:affine-dynamics} hold, and, for every
$\pi\in\mathcal A_{\mathscr C}$, the stochastic action gradient satisfies,
up to an $a$-independent progressively measurable term,
\begin{equation}\label{eq:chi2-quadratic-action-gradient}
 |\mathscr G_t^\pi(a)|
 \leq \kappa |a|^2+\zeta_t^\pi(1+|a|),
 \qquad a\in\mathbb R^k,
\end{equation}
for a constant $\kappa < \infty$ and a nonnegative progressively measurable
process satisfying $\E\int_0^T|\zeta_t^\pi|^2\,\d t<\infty$.
\end{enumerate}
Then, in both regimes, Assumption~\ref{ass:iterates} holds for every $\pi^0\in\AC$ with
$\E\int_0^T\|\tfrac{\d\pi^0_t}{\d\varrho}\|_{L^2_\varrho(A)}^2\,\d t<\infty$. Moreover, one has
$\E\int_0^T\|\tfrac{\d\pi^n_t}{\d\varrho}\|_{L^2_\varrho(A)}^2\,\d t<\infty$ for all $n\ge 0$.
In particular, regime~(ii) contains the affine-LQ model of
Section~\ref{sec:affine-lq}.
\end{Proposition}

\begin{Proposition}[Entropic optimal transport]\label{prop:eot}
Let $(A,d_A)$ be a compact separable metric space, let
$c\in C(A\times A)$ be Lipschitz continuous, fix
$\varrho\in\mathcal P(A)$ and $\varepsilon>0$, and set
$\mathscr C=\mathcal P(A)$. Define
\begin{equation}\label{eq:entropic-transport-regulariser}
h(m):=\min_{\gamma\in\Pi(m,\varrho)}
\left\{
\int_{A\times A}c(a,a')\,\gamma(\d a,\d a')
+\varepsilon D_{\rm KL}(\gamma\mid m\otimes\varrho)
\right\},
\end{equation}
where $\Pi(m,\varrho)$ is the set of couplings with first marginal $m$ and
second marginal $\varrho$, and $D_{\rm KL}$ is understood as $+\infty$ in
the absence of absolute continuity.

Fix $a_0\in A$. For every $m\in\mathcal P(A)$, let
$(\phi[m],\psi[m])\in C(A)\times C(A)$ be the uniquely normalised pair of
Schr\"odinger potentials satisfying $\phi[m](a_0)=0$ given by
\begin{align}
\phi[m](a)
&=-\varepsilon\log\int_A
\exp\left(\frac{\psi[m](a')-c(a,a')}{\varepsilon}\right)
\varrho(\d a'),
\nonumber\\
\psi[m](a')
&=-\varepsilon\log\int_A
\exp\left(\frac{\phi[m](a)-c(a,a')}{\varepsilon}\right)
m(\d a).
\label{eq:schrodinger-potentials}
\end{align}
Then $h$ is finite and convex on $\mathcal P(A)$ and admits the normalised
flat derivative
\begin{equation}\label{eq:eot-flat-derivative}
\frac{\delta h}{\delta m}(m,a)
=\phi[m](a)-\int_A\phi[m](a')\,m(\d a').
\end{equation}

Suppose Assumptions~\ref{ass:H1}--\ref{ass:H2} and either
Assumption~\ref{ass:H3} or Assumption~\ref{ass:H3gen} hold. Assume, for each
$\varphi\in\{b,\sigma,f\}$, that
$\delta\varphi_t/\delta m(x,\mu,m,\cdot)\in C(A)$ and that, for some
$K<\infty$,
\[
\left\|\frac{\delta\varphi_t}{\delta m}
(x,\mu,m,\cdot)\right\|_{C(A)}
\leq K\bigl(1+|x|+M_2(\mu)+M_{2,A}(m)\bigr).
\]
Then Assumption~\ref{ass:iterates} holds for every
$\pi^0\in\mathcal A_{\mathscr C}$.
\end{Proposition}

The entropy statement of Proposition~\ref{prop:KL} guarantees that the SMD iterates are well defined, but
one further point is needed before applying the convergence theorems to the
LQ model. On the unbounded action space $\mathbb R^k$, relative entropy does
not control the mean of two arbitrary probability measures. Thus the global
control-measure assumptions of Theorem~\ref{thm:relsmooth} need not hold on
the whole entropy domain. They do hold in the form actually needed by the
proof along the invariant Gaussian orbit, as recorded next.

\begin{Corollary}[Convergence framework for Gaussian LQ iterates]
\label{cor:gaussian-lq-smd}
Suppose the conditions of Proposition~\ref{prop:KL}(iii) hold. Then there
exists a constant $L_{\rm LQ}<\infty$, independent of the iteration index
and of the comparator, such that, for every admissible nondegenerate Gaussian
control $\pi'$,
\begin{align}
J^\tau(\pi')-J^\tau(\pi^n)
\leq{}&\d J^\tau(\pi^n)(\pi'-\pi^n)
+L_{\rm LQ}\mathcal D_h(\pi'\mid\pi^n).
\label{eq:gaussian-lq-relative-smoothness}
\end{align}
Consequently, if $\lambda\geq L_{\rm LQ}$, the energy-dissipation conclusions
of Theorem~\ref{thm:dissipation} hold for the Gaussian SMD sequence, without
any curvature assumption. If, in addition,
$$
\mathcal Q_{\mathrm{LQ}}(\pi',\pi^n)
\geq-\rho\mathcal D_h(\pi'\mid\pi^n)
$$
along this class of admissible nondegenerate Gaussian controls and
$\tau\geq\rho$, then the convergence conclusions of
Theorem~\ref{thm:main} hold. In particular, under~\eqref{eq:LQ-Riccati-coercivity} and
$\tau>0$, the optimal policy is Gaussian and is an admissible comparator, so these convergence bounds apply directly to $J^\tau(\pi^n)-J^\tau(\pi^\star)$.
\end{Corollary}

%\begin{Remark}
%When $\tau=0$ on $A=\mathbb R^k$, the optimal relaxed LQ control is typically
%a Dirac measure and hence has infinite relative entropy with respect to a
%Gaussian iterate. The dissipation result and the comparison estimate of
%Theorem~\ref{thm:main} remain valid for every fixed Gaussian comparator, but
%the stated $O(1/n)$ estimate cannot be applied directly to the singular
%optimizer. Approximating the latter by Gaussians of vanishing covariance
%generally introduces an additional logarithmic term. This obstruction does
%not occur for $\tau>0$, or on a finite action grid with a full-support
%reference measure.
%\end{Remark}

\newpage
\section{Numerical Experiments}\label{sec:numerical-experiments}

Throughout this section, both the state space and the action space are
one-dimensional. We report four scalar LQ experiments, including the benchmark
with financially motivated negative curvature given by Example \ref{ex:procyclical-systemic-risk}. We first describe the Gaussian
SMD implementation and the subsequent evaluation protocol, and then
present each example by separating its modelling interpretation, its exact
solution, and the observed or theoretically predicted convergence behaviour.

\subsection{Discretisation and SMD implementation}
\label{sec:numerical-implementation}

We discretize the interval $[0,T]$ using a regular time grid
$t_i=i\Delta$, $i=0,\ldots,N_{\rm time}$, with
$\Delta=T/N_{\rm time}$, and let
$j=1,\ldots,M$ index the simulated trajectories. At SMD iteration $n$, the
forward--backward system associated with the current policy is approximated
using the Euler-Maruyama scheme. All experiments reported below use a Gaussian
initial policy and exploit preservation of Gaussianity to evolve only its
affine mean and variance. Thus this subsection describes how the SMD iterates
$\pi^n$ are learned from simulated forward--backward trajectories, whereas
Section~\ref{sec:numerical-objectives} describes the deterministic evaluation
of the cost of the policy obtained at each iteration. For completeness,
Remark~\ref{rem:action-grid}
records an alternative action-grid implementation that is not used in the
reported experiments.
The Brownian increments $\Delta W_i^{(j)}$ are independent
$\mathcal N(0,\Delta)$ random variables across $i$ and $j$.

\paragraph{Forward simulation.}
For the LQ coefficients
of Section~\ref{sec:affine-lq},
\begin{align}
X_{i+1}^{n,(j)}={}&X_i^{n,(j)}
+\bigl(B_iX_i^{n,(j)}+\overline B_i\overline X_i^n
+C_iu_i^{n,(j)}\bigr)\Delta
\nonumber\\
&+\bigl(\gamma_i+D_iX_i^{n,(j)}+\overline D_i\overline X_i^n
+F_iu_i^{n,(j)}\bigr)\Delta W_i^{(j)}, \quad i\geq 0,
\label{eq:numerical-forward}
\end{align}
\noindent where we set
\[
 u_i^{n,(j)}:=\int_{\mathbb R}a\,\pi_i^{n,(j)}(\d a),
 \qquad
 \overline X_i^n:=\frac1M\sum_{j=1}^M X_i^{n,(j)},
 \qquad
 \overline u_i^n:=\frac1M\sum_{j=1}^M u_i^{n,(j)}.
\]

It is important here that the coefficients are evaluated at the policy mean
$u_i^{n,(j)}$, rather than at an action sampled from the policy: this is the
measure-valued dynamics studied throughout the paper.  The second moment of
the policy still enters the running cost through
$\int a^2\pi_i^{n,(j)}(\d a)$.

\paragraph{Backward regression.}
The terminal condition is
\[
Y_{N_{\rm time}}^{n,(j)}
=2P X_{N_{\rm time}}^{n,(j)}
+2\overline P\,\overline X_{N_{\rm time}}^n+2L.
\]
Starting from $i=N_{\rm time}-1$, least-squares regressions on the affine basis
$\{1,x\}$ approximate
\[
\widehat E_i^{n,(j)}
=\widehat{\mathbb E}
 [Y_{i+1}^n\mid X_i^n=X_i^{n,(j)}],
\qquad
Z_i^{n,(j)}
=\frac1\Delta\widehat{\mathbb E}
 [Y_{i+1}^n\Delta W_i\mid X_i^n=X_i^{n,(j)}].
\]
This basis is compatible with the affine form of the exact LQ adjoint
process. Set
\[
\overline{\widehat E}_i^n:=\frac1M\sum_{j=1}^M\widehat E_i^{n,(j)},
\qquad
\overline Z_i^n:=\frac1M\sum_{j=1}^M Z_i^{n,(j)}.
\]
The explicit backward Euler step is
\begin{align}
Y_i^{n,(j)}={}&\widehat E_i^{n,(j)}+\Delta\Bigl(
B_i^{\trans}\widehat E_i^{n,(j)}
+\overline B_i^{\trans}\overline{\widehat E}_i^n
+D_i^{*}Z_i^{n,(j)}
+\overline D_i^{*}\overline Z_i^n
\nonumber\\
&\quad
+2Q_iX_i^{n,(j)}+2\overline Q_i\overline X_i^n
+2I_i^{\trans}u_i^{n,(j)}
+2\overline I_i^{\trans}\overline u_i^n+2M_i
\Bigr).
\label{eq:numerical-backward}
\end{align}
Thus the first two examples require only the regression of $Y$, whereas the
control-dependent-volatility example also tests the numerical estimation of
$Z$.

\paragraph{Action-dependent energy.}
For a state trajectory $j$ and time index $i$, the action-dependent part of
the Hamiltonian is
\begin{align*}
\mathcal E_i^{n,(j)}(a)
&=a^{\trans}N_ia+a^{\trans}\Xi_i^{n,(j)},\\
\Xi_i^{n,(j)}
&=C_i^{\trans}Y_i^{n,(j)}+F_i^{*}Z_i^{n,(j)}
+2I_iX_i^{n,(j)}+2\overline I_i\overline X_i^n+2H_i.
\end{align*}

\paragraph{Gaussian implementation used in the experiments.}
For the experiments below, we use negative differential entropy, with
Lebesgue measure as the underlying measure, and a nondegenerate Gaussian
initialisation. Since $\mathcal E_i^{n,(j)}$ is quadratic, the SMD iterates
remain Gaussian. Writing
\[
\pi_i^{n,(j)}=\mathcal N(\mu_i^{n,(j)},\Gamma_i^n),
\qquad
\mathsf R_i^n=(\Gamma_i^n)^{-1},
\qquad
\eta_i^{n,(j)}=\mathsf R_i^n\mu_i^{n,(j)},
\]
the moments used in the forward scheme and running cost are
\[
 u_i^{n,(j)}=\mu_i^{n,(j)},
 \qquad
 \int_{\mathbb R}a^2\pi_i^{n,(j)}(\d a)
 =|\mu_i^{n,(j)}|^2+\Gamma_i^n.
\]
The SMD step reduces to
\begin{align}
\mathsf R_i^{n+1}
&=\left(1-\frac{\tau}{\lambda}\right)\mathsf R_i^n
+\frac{2}{\lambda}N_i,
&
\Gamma_i^{n+1}&=(\mathsf R_i^{n+1})^{-1},
\label{eq:numerical-gaussian-precision}\\
\eta_i^{n+1,(j)}
&=\left(1-\frac{\tau}{\lambda}\right)\eta_i^{n,(j)}
-\frac1\lambda\Xi_i^{n,(j)},
&
\mu_i^{n+1,(j)}&=\Gamma_i^{n+1}\eta_i^{n+1,(j)}.
\label{eq:numerical-gaussian-mean}
\end{align}
This finite-dimensional recursion avoids an additional action discretisation
and makes it possible to compare the learned Gaussian parameters directly
with their exact values. In the implementation, the regressions of $Y$ and
$Z$ on $\{1,x\}$ make $\Xi_i^n$ affine in the state. The update is therefore
performed directly on the common intercept and slope of
$x\mapsto\mu_i^n(x)$, rather than independently at each simulated particle.
More precisely, write
\[
 \mu_i^n(x)=\alpha_i^n+k_i^nx,\qquad
 \Xi_i^n(x)=\xi_{i,0}^n+\xi_{i,1}^nx.
\]
Since $\eta_i^n(x)=\mathsf R_i^n(\alpha_i^n+k_i^nx)$, identifying the
constant and linear coefficients in
\eqref{eq:numerical-gaussian-mean} gives
\begin{align}
\alpha_i^{n+1}
&=\Gamma_i^{n+1}\left[
 \left(1-\frac{\tau}{\lambda}\right)\mathsf R_i^n\alpha_i^n
 -\frac{\xi_{i,0}^n}{\lambda}\right],
\nonumber\\
k_i^{n+1}
&=\Gamma_i^{n+1}\left[
 \left(1-\frac{\tau}{\lambda}\right)\mathsf R_i^nk_i^n
 -\frac{\xi_{i,1}^n}{\lambda}\right].
\label{eq:numerical-affine-mean-update}
\end{align}
Thus \eqref{eq:numerical-gaussian-mean} and
\eqref{eq:numerical-affine-mean-update} are the same update expressed,
respectively, particlewise and in the affine parametrisation used in the code.

%For a scalar Gaussian $\mathcal N(\mu,\Gamma)$, the entropy contribution is
%$h=-\frac12\log(2\pi e\Gamma)$.

\begin{Remark}[Alternative action-grid implementation]
\label{rem:action-grid}
Instead of exploiting Gaussian invariance, one may discretise the action
space. Let $(a_\ell)_{1\leq \ell\leq N_a}$ be a finite action grid, write
$\pi_{i,\ell}^{n,(j)}$ for the mass assigned to $a_\ell$, and let $\varrho_\ell>0$
denote the reference weights. The moments entering the forward scheme and
the running cost are
\[
 u_i^{n,(j)}=\sum_{\ell=1}^{N_a} a_\ell\pi_{i,\ell}^{n,(j)},
 \qquad
 \int_{\mathbb R}a^2\pi_i^{n,(j)}(\d a)
 =\sum_{\ell=1}^{N_a} a_\ell^2\pi_{i,\ell}^{n,(j)}.
\]
After solving the discretised forward--backward system with these moments,
the entropy mirror step updates the weights according to
\begin{equation}
 \pi_{i,\ell}^{n+1,(j)}
 \propto
 \bigl(\pi_{i,\ell}^{n,(j)}\bigr)^{1-\tau/\lambda}
 \varrho_\ell^{\tau/\lambda}
 \exp\!\left(-\frac{\mathcal E_i^{n,(j)}(a_\ell)}{\lambda}\right).
 \label{eq:numerical-grid-update}
\end{equation}
with the proportionality constant chosen so that the weights sum to one. The
mean used in the next Euler-Maruyama step is therefore
$u_i^{n+1,(j)}=\sum_\ell a_\ell\pi_{i,\ell}^{n+1,(j)}$.
This alternative introduces an additional action-discretisation error and is
not used in the experiments below.
\end{Remark}

\subsection{Evaluation protocol and cost gap}
\label{sec:numerical-objectives}

The purpose of these experiments is not to establish a new numerical solver
for LQ control problems, which can already be solved through Riccati
equations. Rather, the LQ class of Section~\ref{sec:affine-lq} provides a
controlled test environment for SMD: the optimal randomized policy and its
value are available in closed form, while the models retain mean-field
interactions, state-control cross terms, and control-dependent volatility.
The experiments can therefore separate the error in evaluating a given
policy and the bias of the time-discretised benchmark from the error affecting
the learned policy. The latter still contains the regression and sampling
errors incurred when estimating the action gradient.

More specifically, we examine three questions. First, do the SMD iterates
recover the affine mean and the covariance of the optimal Gaussian policy?
Second, does the cost gap exhibit the two regimes predicted by
Theorem~\ref{thm:main}: geometric decay when $\tau-\rho>0$ and sublinear decay
at the threshold $\tau=\rho$? Third, how is the numerical behaviour affected
when the policy enters the volatility and the update consequently depends on
the martingale integrand $Z$?

All examples have horizon $T=1$ and use the negative differential entropy
\[
 h(m)=\int_{\mathbb R}p(a)\log p(a)\,\d a,
 \qquad m(\d a)=p(a)\,\d a.
\]
The mirror parameter is $\lambda=1$. The regularised experiments use
$\tau=0.2$; the comparison regime uses $\tau=0$, except in the procyclical
model, where it uses the critical value $\tau=\rho_{\rm pc}$. Write the
learned Gaussian policy at time $t_i$ as
\(
 \pi_i^n(x)=\mathcal N(\alpha_i^n+k_i^n x,\Gamma_i^n)
\), and let $m_i^n$ and $v_i^n$ denote the mean and variance of its Euler
state. With $\overline u_i^n=\alpha_i^n+k_i^nm_i^n$, these moments are
propagated deterministically according to
\begin{align*}
m_{i+1}^n
&=\bigl[1+\Delta(B_i+\overline B_i)\bigr]m_i^n
  +\Delta C_i\overline u_i^n,\\
v_{i+1}^n
&=\bigl[1+\Delta(B_i+C_ik_i^n)\bigr]^2v_i^n
+\Delta\left[
 \bigl(\gamma_i+(D_i+\overline D_i)m_i^n+F_i\overline u_i^n\bigr)^2
 +(D_i+F_ik_i^n)^2v_i^n\right].
\end{align*}
The convergence figures report the Euler-grid gap
\[
 G_{\Delta,n}^\tau
 :=J_\Delta^\tau(\pi^n)-J_{\Delta,\star}^\tau.
\]
Consequently, no additional simulation is used to evaluate the policy cost:
\begin{equation}\label{eq:exact-euler-policy-cost}
 J_\Delta^\tau(\pi^n)
 =\Delta\sum_{i=0}^{N_{\rm time}-1}\mathcal F_i^n+\mathcal G_N^n,
\end{equation}
where
\begin{align*}
\mathcal F_i^n={}&(Q_i+\overline Q_i)|m_i^n|^2+Q_iv_i^n
+N_i\bigl(|\overline u_i^n|^2+|k_i^n|^2v_i^n+\Gamma_i^n\bigr)
+2I_i(m_i^n\overline u_i^n+k_i^nv_i^n)\\
&+2\overline I_i m_i^n\overline u_i^n+2M_im_i^n
+2H_i\overline u_i^n-\frac{\tau}{2}\log(2\pi e\Gamma_i^n),\\
\mathcal G_N^n={}&(P+\overline P)|m_N^n|^2+Pv_N^n+2Lm_N^n.
\end{align*}
The benchmark $J_{\Delta,\star}^\tau$ is obtained independently by dynamic
programming for the same Euler model. Inserting a quadratic value function
into the discrete Bellman recursion and minimizing the resulting one-step
quadratic Hamiltonian yields the exact affine Gaussian policy on the grid,
which is then evaluated through \eqref{eq:exact-euler-policy-cost}. Thus the
gap contains neither Monte Carlo noise from cost evaluation nor the bias
arising from comparison with the continuous-time optimum. It does, however,
retain the regression and finite-sample errors entering the learned policy.

For $\tau>0$, the discrete optimum is nondegenerate. At $\tau=0$, it is the
Dirac mass at the minimizing affine action; in the code this comparator is
represented by its mean and zero variance, rather than by a density. By
contrast, every finite SMD iterate remains nondegenerate: in the present
scalar setting, $\mathsf R_i^{n+1}=\mathsf R_i^n+2N_i/\lambda$, so that
$\Gamma_i^n\downarrow0$. The unregularised curves therefore describe this
singular Gaussian limit. The dashed curves display the profiles
$C(1-(\tau-\rho)/\lambda)^n$ and $C/n$ and are normalised reference slopes,
not finite-sample bounds. In particular, when $\tau=0$ the Gaussian iterates
cannot be compared directly in relative entropy with the singular optimal
relaxed control; see the remark following
Corollary~\ref{cor:gaussian-lq-smd}. The unregularised plots therefore provide
an empirical limiting benchmark, rather than a direct application of the
$O(1/n)$ conclusion to $\pi^\star$. The same initial policy and numerical
discretisation are used when comparing the two values of $\tau$ within each
example.

\paragraph{Numerical protocol.}
We use $T=1$, $\lambda=1$, an initial Gaussian policy with
$\mu^0=0$, $\Gamma^0=1$, and an Euler grid of
$N_{\rm time}=50$ steps (100 in the procyclical experiment).
The publication experiments use $M_{\rm train}=20,000$ independent
forward--backward trajectories (40,000 for controlled volatility) and 30
independent repetitions. We perform 35 SMD updates at $\tau=0.2$ and
80 updates at $\tau=0$ or $\tau=\rho_{\rm pc}$.
The continuous-time Riccati value is used only to quantify the deterministic
time-discretisation bias and does not enter the displayed gap.
The dashed $C(1-(\tau-\rho)/\lambda)^n$ and $C/n$ profiles are
normalised reference slopes and are not numerical upper bounds.

\subsection{Mean-field systemic risk}

\paragraph{Model and interpretation.}
The first example is the interbank lending model introduced by
Carmona, Fouque, and Sun~\cite{carfousun15}.  The state represents
the log-monetary reserve of a representative bank.  Mean reversion toward the
average reserve models interbank borrowing and lending, while the control
changes the bank's borrowing rate:
\begin{align*}
\d X_t
&=\bigl[\beta(\overline X_t-X_t)
+\overline a(\pi_t)\bigr]\d t+\gamma\,\d W_t,\\
f(x,\mu,\pi)
&=Q(x^2-\overline\mu^2)
+\frac12\int_{\mathbb R}a^2\pi(\d a)
+2I(x-\overline\mu)\overline a(\pi),\\
g(x,\mu)&=P(x^2-\overline\mu^2).
\end{align*}
Here $\overline a(\pi)=\int a\pi(\d a)$.  In the notation of
Section~\ref{sec:affine-lq}, this corresponds to
\[
B=-\beta,\quad \overline B=\beta,\quad C=1,\quad
D=\overline D=F=0,
\]
\[
\overline Q=-Q,\quad N=\frac12,\quad
\overline I=-I,\quad M=H=L=0,
\qquad \overline P=-P.
\]
We take
\[
\beta=0.6,\qquad I=0.4,\qquad Q=P=\gamma=1,
\qquad X_0\sim\mathcal N(0,1).
\]
The energy used in the mirror step is
\[
\mathcal E_t(a)=\frac12a^2+a\ell_t,
\qquad
\ell_t=Y_t+2I(X_t-\overline X_t).
\]

\paragraph{Exact benchmark and curvature condition.}
Set
\[
\kappa:=\sqrt{(\beta+2I)^2+2Q-4I^2},
\qquad s_t:=\kappa(T-t).
\]
The regularised optimal policy is
\[
\pi_t^\star
=\mathcal N\bigl(\varphi(t)(x-\overline x),\tau\bigr),
\qquad \varphi(t)=-2(K(t)+I),
\]
where
\[
K(t)=-\frac12\left[
\beta+2I
-\kappa\,
\frac{\kappa\sinh(s_t)+(\beta+2I+2P)\cosh(s_t)}
{\kappa\cosh(s_t)+(\beta+2I+2P)\sinh(s_t)}
\right].
\]
The scalar term in the value function is
\[
R(t)=\frac{\gamma^2}{2}
\log\!\left(
\cosh(s_t)+\frac{\beta+2I+2P}{\kappa}\sinh(s_t)
\right)
-\frac{\gamma^2}{2}(\beta+2I)(T-t)
-\frac{\tau(T-t)}2\log(2\pi\tau).
\]
In particular, $\varphi(0)=-1.251$ and
$\varphi(T/2)=-1.409$.

This example also satisfies Assumption~\ref{ass:relative-semiconvexity} with
$\rho=0$.  Indeed, $\overline P=-P$, $\overline Q=-Q$, and
$\overline I=-I$, so all mean terms in
$\mathcal Q_{\mathrm{LQ}}$ vanish.  If
$\xi_t=\widehat{\delta X}_t$ and
$v_t=\widehat{\delta u}_t$, then
\[
\d\xi_t=(-\beta\xi_t+v_t)\,\d t,
\qquad \xi_0=0,
\]
and an integration by parts gives
\[
\mathcal Q_{\mathrm{LQ}}(\pi',\pi)
=(P+I)\mathbb E|\xi_T|^2
+(Q+2\beta I)\int_0^T\mathbb E|\xi_t|^2\,\d t.
\]
For the chosen parameters the two coefficients are $1.4$ and $1.48$.

\paragraph{Results.}
Figure~\ref{fig:rates-systemic} compares the two regularisation regimes.  For
$\tau=0.2$, the cost gap decreases geometrically until it reaches a floor of
approximately $6\times10^{-5}$. Since the ordinate is the exact Euler-grid
gap, this floor is not caused by Monte Carlo evaluation of the cost or by
comparison with the continuous-time optimum; it reflects the finite-sample
and regression errors in the policy update. A separate refinement experiment
confirms that this residual decreases as either the number of training paths
or the number of time steps is increased. For
$\tau=0$, the log-log
plot is consistent with a $1/n$ slope over the displayed iterations.  Thus
the experiment exhibits the qualitative transition suggested by the
regularised convergence theory and by its singular limit at $\tau=0$; the
dashed curves are not finite-sample error bounds.

\begin{figure}[htbp]
\centering
\includegraphics[width=\textwidth]{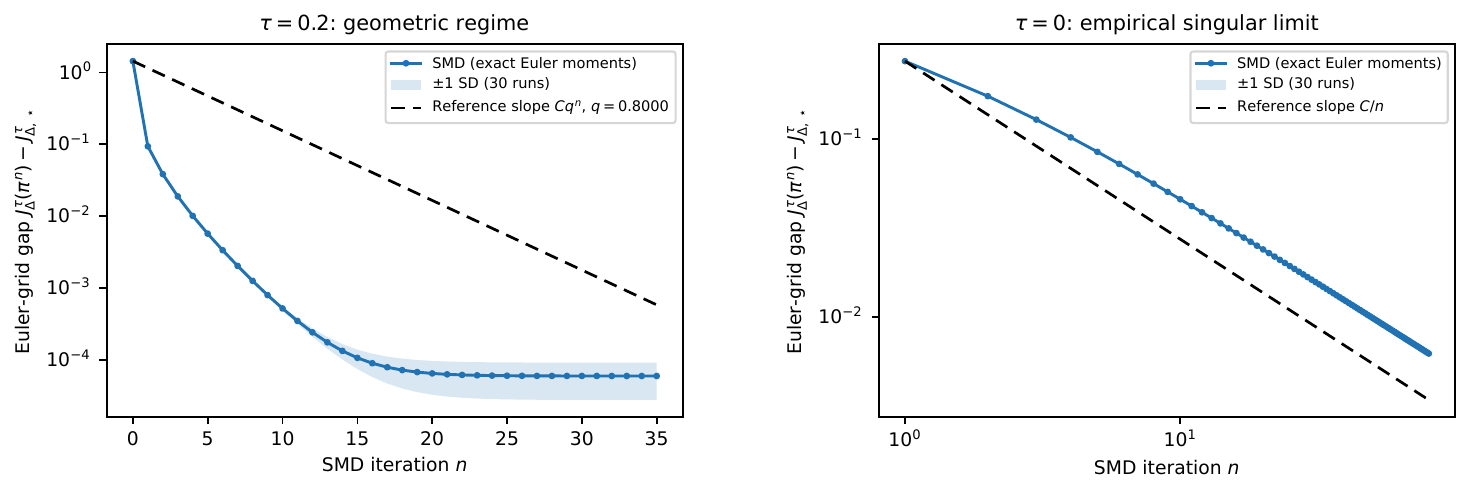}
\caption{Mean-field systemic risk. Euler-grid gap
$G_{\Delta,n}^\tau=J_\Delta^\tau(\pi^n)-J_{\Delta,\star}^\tau$ for
$\tau=0.2$ (left) and $\tau=0$ (right), averaged over 30 independent runs;
the shaded region represents one standard deviation. The dashed curves are
normalised geometric and $1/n$ reference profiles.}
\label{fig:rates-systemic}
\end{figure}

\FloatBarrier

\subsection{Procyclical systemic risk: a negative-curvature benchmark}

This variant of the preceding experiment is designed to test the
convexification mechanism of Theorem~\ref{thm:main}. It keeps the same
dynamics and costs but replaces the stabilising liquidity incentive by the
procyclical coefficient
\[
T=1,\qquad \beta=0.6,\qquad I=-0.9,\qquad
Q=P=\gamma=1,\qquad N=\frac12.
\]
As explained in Example~\ref{ex:procyclical-systemic-risk}, $I<0$ models
balance-sheet expansion by relatively well-capitalised banks and forced
deleveraging by weaker banks. The curvature identity becomes
\[
\mathcal Q_{\mathrm{LQ}}(\pi',\pi)
=0.1\,\E|\xi_T|^2
-0.08\,\E\int_0^T|\xi_t|^2\,\d t.
\]
With $\lambda=1$, Gaussian initial covariance $\Gamma^0=1$, and
$\tau\leq1$, the covariance recursion gives
$\overline\Gamma=1$. Consequently,
\[
\rho_{\rm pc}
=0.16\left(\frac{1-e^{-0.6}}{0.6}\right)^2
\simeq0.0905.
\]
Figure~\ref{fig:rates-pro-systemic} compares the threshold value
$\tau=\rho_{\rm pc}$ with $\tau=0.2$. In the latter case the gap initially
decreases geometrically, consistently with the factor
$1-(0.2-\rho_{\rm pc})/\lambda\simeq0.8905$, before reaching the numerical
floor generated by the estimated policy update. At the critical value, the
gap is compatible with the predicted $1/n$ reference profile and in fact
decreases faster over part of the displayed range. This does not identify an
exact asymptotic rate, but it clearly distinguishes the critical regime from
the geometric one and confirms that convergence persists despite the negative
curvature of the unregularised objective.

\begin{figure}[htbp]
\centering
\includegraphics[width=\textwidth]{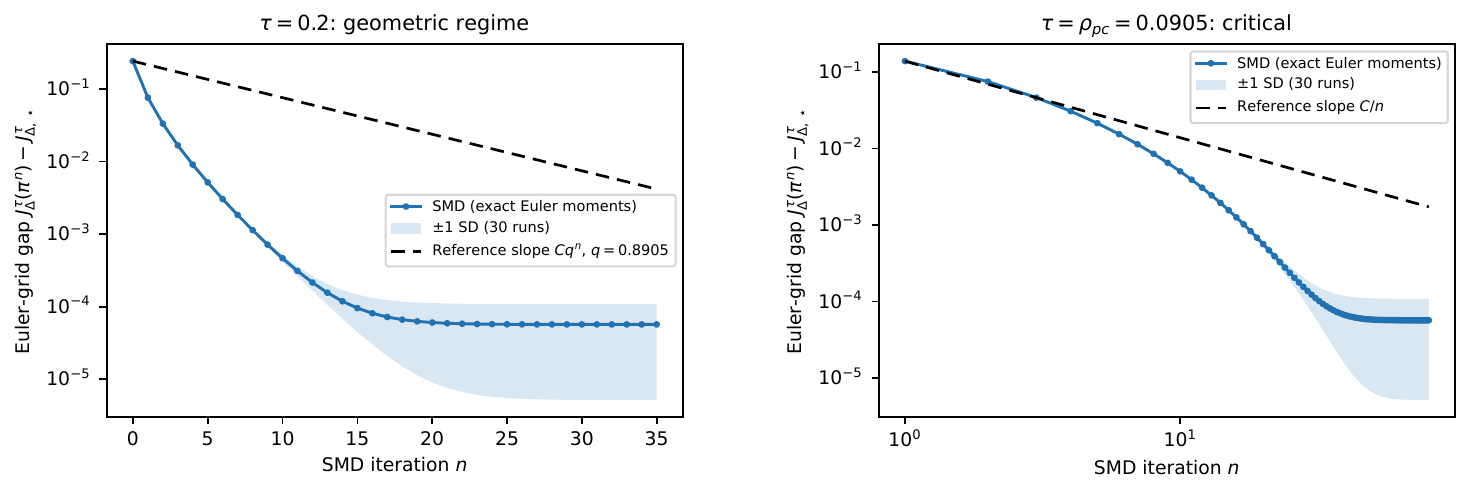}
\caption{Procyclical systemic risk. Euler-grid gap
$G_{\Delta,n}^\tau$ for $\tau=0.2$ (left) and
$\tau=\rho_{\rm pc}\simeq0.0905$ (right), averaged over 30 independent runs;
the shaded region represents one standard deviation. The dashed curves are
normalised geometric and $1/n$ reference profiles.}
\label{fig:rates-pro-systemic}
\end{figure}
\FloatBarrier

\subsection{Optimal trading}

\paragraph{Model and interpretation.}
This example is a stylised mean-field formulation of the classical
optimal-execution trade-off between transaction costs and inventory risk
introduced by Almgren and Chriss~\cite{almgrenchriss01}.  The inventory is
denoted by $X_t$.  The mean of the randomized control is the trading rate, the
quadratic term penalises fast execution, and the terminal variance penalises
dispersion of the remaining inventory:
\begin{align*}
\d X_t&=\overline a(\pi_t)\,\d t+\gamma\,\d W_t,\\
f(x,\mu,\pi)
&=\int_{\mathbb R}a^2\pi(\d a)+2H\overline a(\pi),\\
g(x,\mu)&=P(x^2-\overline\mu^2).
\end{align*}
In the notation of Section~\ref{sec:affine-lq}, this means
\[
C=1,\qquad N=1,\qquad \overline P=-P,
\]
with constant diffusion $\gamma$, affine action-cost coefficient $H$, and
$B=\overline B=D=\overline D=F=0$; all the remaining running-cost
coefficients vanish.
We use
\[
P=3,\qquad H=2,\qquad \gamma=1,
\qquad X_0\sim\mathcal N(1,1).
\]
The corresponding action energy is
\[
\mathcal E_t(a)=a^2+a(Y_t+2H).
\]

\paragraph{Exact benchmark and curvature condition.}
The exact coefficients and optimal policy are
\begin{align*}
K(t)&=\frac{P}{1+P(T-t)},\\
R(t)&=\gamma^2\log(1+P(T-t))
-\left(H^2+\frac\tau2\log(\pi\tau)\right)(T-t),\\
\pi_t^\star
&=\mathcal N\bigl(-K(t)(x-\overline x)-H,\tau/2\bigr).
\end{align*}
Thus the optimal feedback coefficient is $-0.750$ at time $0$ and $-1.200$
at time $T/2$, while the intercept is the constant $-H=-2$.  The variance
$\tau/2$ offers a direct check of the covariance recursion
\eqref{eq:numerical-gaussian-precision}.

Here $\overline P=-P$ and there are no running state or state--control cross
terms.  Proposition~\ref{prop:lq-semiconvexity} gives
\[
\mathcal Q_{\mathrm{LQ}}(\pi',\pi)
=P\mathbb E|\widehat{\delta X}_T|^2\geq0.
\]
The action cost and the linear execution term are affine in the relaxed
control measure, and hence have zero first-order remainder.  Assumption
\ref{ass:relative-semiconvexity} therefore holds with $\rho=0$.

\paragraph{Results.}
The left panel of Figure~\ref{fig:rates-trading} again exhibits geometric
decay for $\tau=0.2$, with small variability across the 30 runs. In the
unregularised case, the gap decreases throughout the displayed range and is
compatible with the sublinear conclusion, although the finite-iteration
curve should not be interpreted as identifying the exact asymptotic rate.
In addition to the objective value, the Gaussian
recursion recovers the structural features of the exact control: covariance
$\tau/2$, constant intercept $-H$, and increasing liquidation intensity
$K(t)$ as the terminal date approaches.

\begin{figure}[htbp]
\centering
\includegraphics[width=0.98\textwidth]{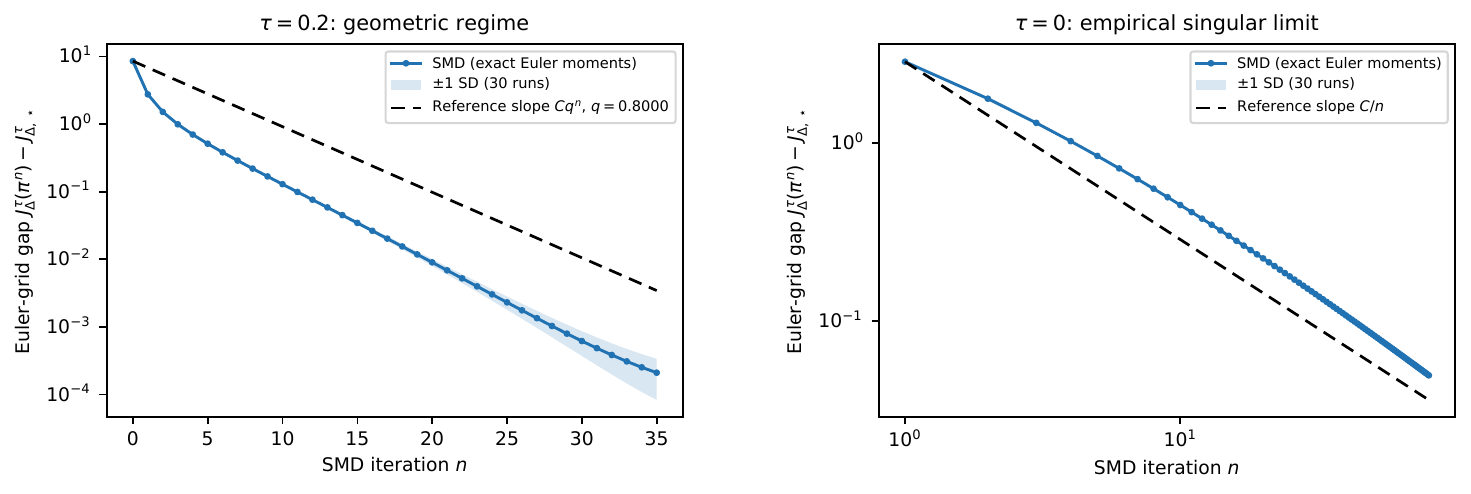}
\caption{Optimal trading. Euler-grid gap $G_{\Delta,n}^\tau$ in the
regularised geometric regime (left) and in the unregularised empirical
singular limit (right), averaged over 30 independent runs. The shaded region
represents one standard deviation and the dashed curves are normalised
reference slopes.}
\label{fig:rates-trading}
\end{figure}

\FloatBarrier

\subsection{Control-dependent volatility}

\paragraph{Model and interpretation.}
The final example belongs to the classical family of stochastic LQ problems
in which the control enters both the drift and the diffusion; see, in
particular, Ait Rami, Moore, and Zhou~\cite{aitramimoorezhou02} for the
associated generalized Riccati structure.  Here we adapt this setting to
randomized controls.  The mean control affects both the drift and the
volatility,
\begin{align*}
\d X_t
&=(BX_t+C\overline a(\pi_t))\,\d t
+(\gamma+F\overline a(\pi_t))\,\d W_t,\\
f(x,\pi)&=Qx^2+N\int_{\mathbb R}a^2\pi(\d a),\\
g(x)&=Px^2.
\end{align*}
The coefficients are
\[
\begin{gathered}
B=0.3,\qquad C=\gamma=1,\qquad F=0.5,\\
Q=N=0.5,\qquad P=1,\qquad
X_0\sim\mathcal N(1,0.25).
\end{gathered}
\]
The mirror update now uses
\[
\mathcal E_t(a)=Na^2+a(CY_t+FZ_t).
\]
Consequently, errors in the regression estimate of $Z$ directly perturb the
policy update; this makes the example a more demanding test of the
forward-backward implementation.

\paragraph{Exact benchmark and curvature condition.}
Let $K_t$ solve
\[
\dot K_t+2BK_t+Q-\frac{(CK_t)^2}{N+F^2K_t}=0,
\qquad K_T=P,
\]
and write $S_t=N+F^2K_t$, consistently with the notation of
Section~\ref{sec:affine-lq}.  The remaining affine coefficient is
determined by
\[
\dot r_t+Br_t+CK_tc_t=0,\qquad r_T=0,
\qquad
c_t=-\frac{Cr_t+FK_t\gamma}{S_t}.
\]
The optimal randomized policy is
\[
\pi_t^\star
=\mathcal N(k_tx+c_t,\tau/(2N)),
\qquad
k_t=-\frac{CK_t}{S_t}.
\]
Numerically, $k_0=-1.224$, $k_{T/2}=-1.260$,
$c_0=-0.094$, and $c_{T/2}=-0.298$.

There are no mean-field cost terms or state--control cross terms.  The
coefficients above are already those of the LQ model in
Section~\ref{sec:affine-lq}.  Consequently,
\[
\mathcal Q_{\mathrm{LQ}}(\pi',\pi)
=P\mathbb E|\delta X_T|^2
+Q\int_0^T\mathbb E|\delta X_t|^2\,\d t\geq0.
\]
The control-dependent diffusion contributes no remainder because it is
affine in the policy mean.  Thus Assumption
\ref{ass:relative-semiconvexity} again holds with $\rho=0$.

\paragraph{Results.}
Figure~\ref{fig:rates-ctrlvol} shows the same qualitative transition between
geometric and sublinear convergence.  The regularised curve has a longer
initial transient and reaches a higher floor than in the systemic-risk
example. This is consistent with the additional statistical error carried by
the estimate of $Z$ in the coefficient $CY+FZ$. In the final backward sweep,
the affine $Y$ regressions have coefficients of determination between about
$0.96$ and $0.98$. If the fitted adjoint is
$Y_{i+1}=a_{i+1}+b_{i+1}X_{i+1}$, the corresponding LQ oracle is
\[
 Z_i^{\rm LQ}(X_i)
 =b_{i+1}\bigl(\gamma_i+D_iX_i+\overline D_im_i+F_iu_i\bigr).
\]
The relative root-mean-square error
\[
 \left(
 \frac{\sum_j|Z_i^{n,(j)}-Z_i^{\rm LQ}(X_i^{n,(j)})|^2}
      {\sum_j|Z_i^{\rm LQ}(X_i^{n,(j)})|^2}
 \right)^{1/2}
\]
is between $5\%$ and $10\%$ over the reported time indices. These diagnostics
support this interpretation without identifying the dominant contribution to
the error floor. Once the policy-estimation error dominates, further SMD
iterations no longer reduce the measured gap.

\begin{figure}[htbp]
\centering
\includegraphics[width=\textwidth]{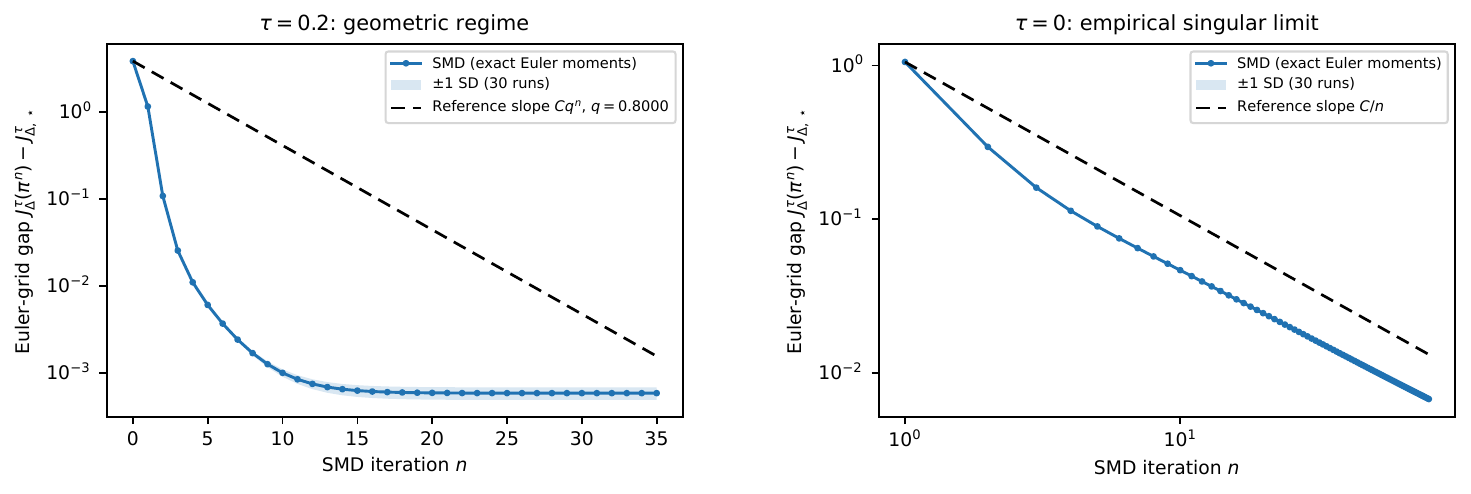}
\caption{Control-dependent volatility. Euler-grid gap
$G_{\Delta,n}^\tau$ for $\tau=0.2$ (left) and $\tau=0$ (right), averaged over
30 independent runs. The shaded region represents one standard deviation;
the higher regularised floor is consistent with the additional
$Z$-regression error entering the mirror update.}
\label{fig:rates-ctrlvol}
\end{figure}

\FloatBarrier

\section{Conclusion}

We have developed a stochastic mirror descent framework for continuous-time
McKean--Vlasov control with measure-valued controls. The analysis combines
the first-order calculus of the MKV forward--backward system with the
Bregman geometry of the control space. Relative smoothness is obtained both
for nonlinear models with bounded cost derivatives and for affine dynamics
with linearly growing cost derivatives. The latter regime includes standard
quadratic costs, while the entropy-regularised LQ specialisation closes on an
invariant Gaussian orbit.

Trajectory-wise relative semiconvexity provides the complementary lower
curvature estimate. Unlike pointwise Hamiltonian convexity, it permits
compensation between the dynamics and the costs along admissible
trajectories. Its mean--fluctuation reduction makes the condition explicit
for affine-quadratic models and identifies a financially motivated
procyclical example with a strictly positive curvature defect. Together,
relative smoothness and semiconvexity yield energy dissipation, a sublinear
rate at the critical regularisation level, and geometric convergence above
that threshold.

Several questions remain open. The constants entering the global smoothness
estimate are not optimised, and outside invariant parametric classes the
closure of the mirror iterates must still be verified for each regulariser.
It would also be natural to combine the present optimisation analysis with
non-asymptotic error bounds for the numerical approximation of the MKV
forward--backward system, and to investigate non-Gaussian invariant families
for nonlinear mean-field models.

\newpage
\section{Proofs}\label{sec:proofs}
\subsection{Technical results on the adjoint process}

\begin{Lemma}[Uniform $H^\infty$ and BMO bounds on the adjoint process]\label{lem:Ybound}
Suppose Assumption~\ref{ass:H3} holds and that the adjoint
BSDE~\eqref{eq:adjoint-mkv-full} is well posed in the sense of
Proposition~\ref{prop:wellposed}. 
Then, it holds
\[
  \sup_{\pi\in\mathcal{A}_{\mathscr{C}}}\ \esssup_{(t,\omega)} |Y_t^{\pi}|
  \;=\;\sup_{\pi\in\mathcal{A}_{\mathscr{C}}}\|Y^{\pi}\|_{H^\infty}\;<\; \infty.
\]

Moreover, it holds
\[
\sup_{\pi\in\mathcal{A}_{\mathscr{C}}}\|Z^\pi\cdot W\|_{\mathrm{BMO}}\;<\; \infty.
\]

\end{Lemma}

\begin{proof}
Fix $\pi\in\mathcal{A}_{\mathscr{C}}$, write $(Y,Z)=(Y^\pi,Z^\pi)$, and set
$\mu_s=\mathcal L(X_s^\pi)$. The terminal condition is
\[
 \xi^\pi
 :=D_xg(X_T^\pi,\mu_T)
   +\widetilde\E\bigl[(\partial_\mu g)(\widetilde X_T^\pi,\mu_T)(X_T^\pi)\bigr],
 \qquad |\xi^\pi|\le 2K.
\]
The adjoint equation is a linear BSDE in $(Y,Z)$. Its inhomogeneous term is
\begin{align*}
 \bar{\Dc}_s^\pi
 := {}&D_xf_s(X_s^\pi,\mu_s,\pi_s)\\
 &+\widetilde\E\Big[
   (\partial_\mu b_s)(\widetilde X_s^\pi,\mu_s,\widetilde\pi_s)(X_s^\pi)\trans
       \widetilde Y_s^\pi
   +(\partial_\mu\sigma_s)(\widetilde X_s^\pi,\mu_s,\widetilde\pi_s)(X_s^\pi)\trans
       \widetilde Z_s^\pi\\
 &\hspace{39mm}
   +(\partial_\mu f_s)(\widetilde X_s^\pi,\mu_s,\widetilde\pi_s)(X_s^\pi)
 \Big].
\end{align*}
By Assumption~\ref{ass:H3}, with a constant $C$ depending only on the dimensions and on $K$,
\begin{equation}\label{eq:barD-deterministic-bound}
 |\bar{\Dc}_s^\pi|
 \le C\bigl(1+\E|Y_s^\pi|+\E|Z_s^\pi|\bigr).
\end{equation}
Notice that the right-hand side is deterministic. Therefore, the uniform
$S^2\times H^2$ estimate in Proposition~\ref{prop:wellposed} yields
\begin{equation}\label{eq:barD-L2-bound}
 \sup_{\pi\in\mathcal A_{\mathscr C}}
 \int_0^T|\bar{\Dc}_s^\pi|^2\,\d s<\infty.
\end{equation}

Let $S^\pi$ be the matrix-valued stochastic flow defined by
\[
 \d S_s^\pi
 =S_s^\pi D_xb_s(X_s^\pi,\mu_s,\pi_s)\,\d s
  +S_s^\pi D_x\sigma_s(X_s^\pi,\mu_s,\pi_s)\,\d W_s,
 \qquad S_0^\pi=I_d.
\]
The representation formula for multidimensional linear BSDEs
\cite[Proposition~3.1]{harterrichou19} gives
\begin{equation}\label{expres:y:bsde-new}
 Y_t^\pi
 =\E_t\left[(S_t^\pi)^{-1}S_T^\pi\xi^\pi
   +\int_t^T(S_t^\pi)^{-1}S_s^\pi\bar{\Dc}_s^\pi\,\d s\right].
\end{equation}
Since $D_xb$ and $D_x\sigma$ are uniformly bounded, the standard conditional
moment estimate for this linear stochastic flow (see also
\cite[Section~3]{harterrichou19}) gives a constant $C_S$, independent of $\pi$, such that
\[
 \sup_{\pi\in\mathcal A_{\mathscr C}}\sup_{0\le t\le T}
 \left\|\E_t\left[\sup_{t\le s\le T}
 |(S_t^\pi)^{-1}S_s^\pi|^2\right]\right\|_{L^\infty}
 \le C_S.
\]
Combining this estimate with \eqref{eq:barD-L2-bound}, the conditional
Cauchy--Schwarz inequality, and $|Y_T^\pi|\le2K$ in
\eqref{expres:y:bsde-new}, we obtain a constant $M<\infty$ such that
\begin{equation}\label{eq:Y-uniform-bound}
 \sup_{\pi\in\mathcal A_{\mathscr C}}\|Y^\pi\|_{H^\infty}\le M.
\end{equation}

It remains to prove the BMO estimate. Denote the full driver of the adjoint
BSDE by $\Dc_s^\pi$. Assumption~\ref{ass:H3} and
\eqref{eq:Y-uniform-bound} imply
\begin{equation}\label{eq:driver-bound-bmo}
 |\Dc_s^\pi|
 \le C\bigl(1+M+|Z_s^\pi|+\E|Z_s^\pi|\bigr).
\end{equation}
Let $\eta$ be an arbitrary stopping time with values in $[0,T]$.
Applying It\^o's formula to $|Y^\pi|^2$ on $[\eta,T]$ and taking the
conditional expectation gives
\begin{equation}\label{eq:ito-bmo}
 |Y_\eta^\pi|^2+\E_\eta\int_\eta^T|Z_s^\pi|^2\,\d s
 =\E_\eta|Y_T^\pi|^2
   +2\E_\eta\int_\eta^T (Y_s^\pi)\trans\Dc_s^\pi\,\d s.
\end{equation}
Indeed, the stochastic integral is a true martingale because $Y^\pi$ is
bounded and $Z^\pi\in H^2$. By \eqref{eq:driver-bound-bmo} and the bound
$|Y_s^\pi|\le M$, we first have
\[
 2\bigl|(Y_s^\pi)\trans\Dc_s^\pi\bigr|
 \le C_M\bigl(1+|Z_s^\pi|+\E|Z_s^\pi|\bigr).
\]
Applying Young's inequality only to the pathwise term $C_M|Z_s^\pi|$
and leaving the deterministic term $C_M\E|Z_s^\pi|$ unchanged gives
\[
 2\bigl|(Y_s^\pi)\trans\Dc_s^\pi\bigr|
 \le \frac12|Z_s^\pi|^2+C_M\bigl(1+\E|Z_s^\pi|\bigr),
\]
where $C_M$ is independent of $\pi$. Proposition~\ref{prop:wellposed} also
implies, for some $C_0<\infty$ independent of $\pi$,
\[
 \int_0^T\E|Z_s^\pi|\,\d s
 \le \sqrt{T}\left(\E\int_0^T|Z_s^\pi|^2\,\d s\right)^{1/2}
 \le C_0.
\]
Using these bounds in \eqref{eq:ito-bmo}, dropping $|Y_\eta^\pi|^2$, and
absorbing the term $\frac12\E_\eta\int_\eta^T|Z_s^\pi|^2\,\d s$ into the
left-hand side yields
\[
 \E_\eta\int_\eta^T|Z_s^\pi|^2\,\d s
 \le C_{\mathrm{BMO}}
 \qquad\text{a.s.},
\]
where $C_{\mathrm{BMO}}$ is independent of both $\eta$ and $\pi$.
Taking the essential supremum over $\eta$ and then the supremum over $\pi$
proves
\[
 \sup_{\pi\in\mathcal A_{\mathscr C}}
 \|Z^\pi\cdot W\|_{\mathrm{BMO}}^2\le C_{\mathrm{BMO}}.
\]
\end{proof}

\begin{Lemma}[Stability of the state process]\label{lem:Xstab}
Under Assumption~\ref{ass:H1}, for all $T$, there exists $C<\infty$ such that for all
$\pi,\pi'\in\AC$,
\[
  \E\Bigl[\sup_{0\le t\le T}|X_t^\pi-X_t^{\pi'}|^2\Bigr]
  \;\le\; C\, \mathcal{D}_h(\pi \mid \pi').
\]
\end{Lemma}

\begin{proof}
Fix $\pi,\pi'\in\mathcal{A}_{\mathscr{C}}$. 
\noindent From Assumption~\ref{ass:H1}, the Burkholder-Davis-Gundy inequality, it is standard to get  
\begin{align*}
     \E\Bigl[\sup_{0\le t\le T}|X_t^\pi-X_t^{\pi'}|^2\Bigr]
   &\le C\!\int_0^{T}\!\E\Bigl[\sup_{0\le t\le s}|X_t^{\pi}-X_t^{\pi'}|^2\Bigr]\d s
   + C\, \mathcal{D}_{h}(\pi \mid \pi').
\end{align*}
 
\noindent By Gr\"onwall's lemma, up to a modifcation of $C$, we deduce
\[
\E\Bigl[\sup_{0\le t\le T}|X_t^{\pi}-X_t^{\pi'}|^2\Bigr] \le C\,\mathcal{D}_{h}(\pi \mid \pi').
\]

\end{proof}

\begin{Lemma}[Fefferman's inequality]\label{lem:fefferman}
Let $M$ be a BMO martingale and $N$ a martingale with
$\E\bigl[\langle N\rangle_T^{1/2}\bigr]<\infty$. Then
\[
  \E\!\left[\int_0^T |\d\langle M,N\rangle_t|\right]
  \;\le\; \sqrt{2}\,\|M\|_{\mathrm{BMO}}\,\E\!\left[\langle N\rangle_T^{1/2}\right].
\]
\end{Lemma}

\begin{Lemma}\label{lem:bmoL1L2}
Let $f,g:\Omega\times[0,T]\to\R$ be square-integrable progressively measurable
processes such that $\|f\cdot W\|_{\mathrm{BMO}}<\infty$, where $W$ is a Brownian motion.
Then
\[
  \E\!\left[\int_0^T |f_t\,g_t|\,\d t\right]
  \;\le\; \sqrt{2}\,\|f\cdot W\|_{\mathrm{BMO}}\;
          \E\!\left[\Bigl(\int_0^T |g_t|^2\,\d t\Bigr)^{1/2}\right].
\]
\end{Lemma}

\begin{proof}
By It\^o's isometry for the cross-variation,
$\E[\int_0^T |f_t g_t|\,\d t]=\E[\int_0^T \d\langle |f|\cdot W,\,|g|\cdot W\rangle_t]$.
Applying Lemma~\ref{lem:fefferman} with $M=|f|\cdot W$ and $N=|g|\cdot W$, together with
$\langle |g|\cdot W\rangle_T=\int_0^T|g_t|^2\,\d t$, gives the claim.
\end{proof}

\begin{Lemma}[Lipschitz continuity of the adjoint process]\label{lem:YZlipc}
Suppose Assumptions~\ref{ass:H1},~\ref{ass:H2}
and~\ref{ass:deriv} hold. Assume, in addition, that one of the following two
alternative regimes is satisfied:
\begin{enumerate}
\item[(i)] \emph{Bounded-cost regime:} Assumption~\ref{ass:H3} holds;
\item[(ii)] \emph{Linear-growth affine regime:}
Assumptions~\ref{ass:H3gen} and~\ref{ass:affine-dynamics} hold.
\end{enumerate}
Then there exists a constant $C>0$, independent of $\pi$ and $\pi'$, such that
for all $\pi,\pi'\in\AC$,
\[
  \E\Bigl[\sup_{t\in[0,T]}|Y_t^{\pi}-Y_t^{\pi'}|^2\Bigr]
  + \E \Big[\int_0^T |Z_t^{\pi}-Z_t^{\pi'}|^2\,\d t \Big]
  \;\le\; C\,\mathcal{D}_{h}(\pi \mid \pi').
\]
\end{Lemma}

\begin{proof}
Fix $\pi,\pi'\in\mathcal A_{\mathscr C}$ and set
\[
 \delta X:=X^{\pi'}-X^\pi,\qquad
 \delta Y:=Y^{\pi'}-Y^\pi,\qquad
 \delta Z:=Z^{\pi'}-Z^\pi.
\]
We use the same notation with a tilde for an independent copy. Also set
$\mu_t^\pi:=\mathcal L(X_t^\pi)$ and
$\Xi_t^\pi:=(X_t^\pi,\mu_t^\pi,\pi_t)$.
To avoid ambiguity, denote the full generator of the adjoint BSDE by
\begin{align*}
 G_t^\pi
 :={}&D_xH_t^0\bigl(\Theta_t^\pi,\pi_t\bigr) +\widetilde\E\left[
   \partial_\mu H_t^0
   \bigl(\widetilde\Theta_t^\pi,\widetilde\pi_t\bigr)
   (X_t^\pi)
 \right],
\end{align*}
and set $\delta G_t:=G_t^{\pi'}-G_t^\pi$. Thus
\[
 \d\delta Y_t=-\delta G_t\,\d t+\delta Z_t\,\d W_t.
\]

\medskip
\noindent\emph{Common estimates.}
By Proposition~\ref{prop:wellposed}, $(Y^\pi,Z^\pi)$ and
$(Y^{\pi'},Z^{\pi'})$ belong to $S^2\times H^2$ in both regimes. Hence, for
every $\beta\geq0$,
\[
 M_t:=\int_0^t e^{\beta r}\delta Y_r\trans\delta Z_r\,\d W_r
\]
is an $\mathcal H^1$-martingale. Indeed,
\begin{align*}
 \E\bigl[\langle M\rangle_T^{1/2}\bigr]
 &\le e^{\beta T}\E\left[
   \sup_{0\leq t\leq T}|\delta Y_t|
   \left(\int_0^T|\delta Z_t|^2\,\d t\right)^{1/2}
 \right]\\
 &\le e^{\beta T}\|\delta Y\|_{S^2}\|\delta Z\|_{H^2}<\infty.
\end{align*}
Applying It\^o's formula to $e^{\beta t}|\delta Y_t|^2$ gives
\begin{align*}
 \d\bigl(e^{\beta t}|\delta Y_t|^2\bigr)
 ={}&e^{\beta t}\bigl(
   \beta|\delta Y_t|^2+|\delta Z_t|^2
   -2\delta Y_t\trans\delta G_t
 \bigr)\,\d t\\
 &+2e^{\beta t}\delta Y_t\trans\delta Z_t\,\d W_t.
\end{align*}
Since the last term is an $\mathcal H^1$-martingale, we get
\begin{align}
 &\E\int_0^T e^{\beta t}
 \bigl(|\delta Z_t|^2+\beta|\delta Y_t|^2\bigr)\,\d t
 \nonumber\\
 &\qquad\le e^{\beta T}\E|\delta Y_T|^2
 +2\E\int_0^T e^{\beta t}\delta Y_t\trans
 \delta G_t\,\d t,
 \label{eq:YZ-Ito-expectation}
\end{align}
where the nonpositive term $-\E|\delta Y_0|^2$ has been dropped.

Assumption~\ref{ass:deriv}, Jensen's inequality and the coupling
$(X_T^\pi,X_T^{\pi'})$ yield
\begin{equation}\label{eq:stability:terminal:cond}
 \E|\delta Y_T|^2\le C\E|\delta X_T|^2 \le C\mathcal D_h(\pi\mid\pi').
\end{equation}
\noindent where we used Lemma~\ref{lem:Xstab} for the last inequality.

To decompose the generator, define
\begin{align*}
 B_t^\pi&:=(D_xb_t)(\Xi_t^\pi),
 &\Sigma_t^\pi&:=(D_x\sigma_t)(\Xi_t^\pi),
 &F_t^\pi&:=(D_xf_t)(\Xi_t^\pi),\\
 \overline B_t^\pi
 &:=(\partial_\mu b_t)(\widetilde X_t^\pi,\mu_t^\pi,
       \widetilde\pi_t)(X_t^\pi),
 &\overline\Sigma_t^\pi
 &:=(\partial_\mu\sigma_t)(\widetilde X_t^\pi,\mu_t^\pi,
       \widetilde\pi_t)(X_t^\pi),\\
 \overline F_t^\pi
 &:=(\partial_\mu f_t)(\widetilde X_t^\pi,\mu_t^\pi,
       \widetilde\pi_t)(X_t^\pi).
\end{align*}
Then
\begin{align*}
 \delta G_t
 ={}&(B_t^\pi)\trans\delta Y_t
    +(\Sigma_t^\pi)\trans\delta Z_t\\
 &+\widetilde\E\bigl[(\overline B_t^\pi)\trans
          \delta\widetilde Y_t
        +(\overline\Sigma_t^\pi)\trans\delta\widetilde Z_t\bigr]
    +R_t^f+R_t^b+R_t^\sigma,
\end{align*}
where
\begin{align*}
 R_t^f
 &:=(F_t^{\pi'}-F_t^\pi)
   +\widetilde\E[\overline F_t^{\pi'}-\overline F_t^\pi],\\
 R_t^b
 &:=(B_t^{\pi'}-B_t^\pi)\trans Y_t^{\pi'}
   +\widetilde\E\bigl[
     (\overline B_t^{\pi'}-\overline B_t^\pi)\trans
     \widetilde Y_t^{\pi'}\bigr],\\
 R_t^\sigma
 &:=(\Sigma_t^{\pi'}-\Sigma_t^\pi)\trans Z_t^{\pi'}
   +\widetilde\E\bigl[
     (\overline\Sigma_t^{\pi'}-\overline\Sigma_t^\pi)\trans
     \widetilde Z_t^{\pi'}\bigr].
\end{align*}
Assumption~\ref{ass:deriv}, Jensen's inequality and
Lemma~\ref{lem:Xstab} imply
\begin{equation}\label{eq:Rf-stability}
 \E\int_0^T|R_t^f|^2\,\d t
 \le C\mathcal D_h(\pi\mid\pi').
\end{equation}
Consequently, for every $\eta>0$, Young's inequality gives
\begin{align}
2 \E\int_0^T e^{\beta t}|\delta Y_t|\,|R_t^f|\,\d t
 &\le \eta \E\int_0^T |\delta Y_t|^2\,\d t
   +C_{\eta}\mathcal D_h(\pi\mid\pi').
 \label{eq:Rf-cross-term}
\end{align}
This is the estimate used below for the contribution of $R^f$ to the
last integral in~\eqref{eq:YZ-Ito-expectation}.

It remains to estimate the linear part of $\delta G$. Set
\[
 L_t:=(B_t^\pi)\trans\delta Y_t+(\Sigma_t^\pi)\trans\delta Z_t
 +\widetilde\E\left[
   (\overline B_t^\pi)\trans\delta\widetilde Y_t
   +(\overline\Sigma_t^\pi)\trans\delta\widetilde Z_t
 \right].
\]
Since the four coefficients in this display are uniformly bounded in both
regimes, Cauchy--Schwarz and Young's inequalities yield
\begin{equation}\label{eq:linear-generator-bound}
 2\E\int_0^T e^{\beta t}
   \bigl|\delta Y_t\trans L_t\bigr|\,\d t
 \le C_0\E\int_0^T e^{\beta t}|\delta Y_t|^2\,\d t
 +\frac12\E\int_0^T e^{\beta t}|\delta Z_t|^2\,\d t.
\end{equation}
Indeed, the terms involving the independent copy are controlled using
\[
 \E\left[|\delta Y_t|\,\widetilde\E|\delta\widetilde Y_t|\right]
 =\bigl(\E|\delta Y_t|\bigr)^2
 \le \E|\delta Y_t|^2
\]
and
\[
 \E\left[|\delta Y_t|\,\widetilde\E|\delta\widetilde Z_t|\right]
 =\E|\delta Y_t|\,\E|\delta Z_t|
 \le \rho\E|\delta Z_t|^2+C_\rho\E|\delta Y_t|^2,
\]
where $\rho>0$ is chosen so that all the resulting
$|\delta Z_t|^2$ terms are bounded by the last term in
\eqref{eq:linear-generator-bound}.

\medskip
\noindent\emph{Case (i): bounded-cost regime.}
Lemma~\ref{lem:Ybound} gives, uniformly in the control,
\[
 \|Y^\pi\|_{H^\infty}
 +\|Z^\pi\cdot W\|_{\mathrm{BMO}}\le C.
\]
Using the Lipschitz bounds in Assumption~\ref{ass:deriv}, the uniform bound
on $Y^{\pi'}$, and Lemma~\ref{lem:Xstab}, we obtain
\begin{equation}\label{eq:Rb-stability}
 \E\int_0^T|R_t^b|^2\,\d t
 \le C\mathcal D_h(\pi\mid\pi').
\end{equation}
Hence, exactly as in~\eqref{eq:Rf-cross-term},
\begin{equation}\label{eq:Rb-cross-term}
 2\E\int_0^T e^{\beta t}|\delta Y_t|\,|R_t^b|\,\d t
 \le \eta \E\int_0^T |\delta Y_t|^2\,\d t
 +C_{\eta} \mathcal D_h(\pi\mid\pi').
\end{equation}
Set
\[
 \delta\Sigma_t:=\Sigma_t^\pi-\Sigma_t^{\pi'},
 \qquad
 \delta\overline\Sigma_t
 :=\overline\Sigma_t^\pi-\overline\Sigma_t^{\pi'}.
\]
Assumption~\ref{ass:deriv} and Lemma~\ref{lem:Xstab} imply
\begin{equation}\label{eq:diffusion-derivatives-stability}
 \E\int_0^T\left(
   |\delta\Sigma_t|^2
   +\widetilde\E|\delta\overline\Sigma_t|^2
 \right)\,\d t
 \le C\mathcal D_h(\pi\mid\pi').
\end{equation}
By Lemma~\ref{lem:bmoL1L2} and the uniform BMO bound on $Z^{\pi'}$,
\begin{align*}
 &\E\int_0^T |\delta Y_t|\,|\delta\Sigma_t|\,|Z_t^{\pi'}|\,\d t\\
 &\quad\le C\E\left[\left(
   \int_0^T|\delta Y_t|^2|\delta\Sigma_t|^2\,\d t
 \right)^{1/2}\right]\\
 &\quad\le C
 \left(\E\sup_{0\leq t\leq T}|\delta Y_t|^2\right)^{1/2}
 \mathcal D_h(\pi\mid\pi')^{1/2}.
\end{align*}
Applying the same argument on the product space, with $\widetilde Z^{\pi'}$
integrated against the copied Brownian motion, gives
\begin{align*}
 &\E\int_0^T |\delta Y_t|\,
 \left|\widetilde\E\left[
   (\delta\overline\Sigma_t)\trans\widetilde Z_t^{\pi'}
 \right]\right|\,\d t\\
 &\quad\le C
 \left(\E\sup_{0\leq t\leq T}|\delta Y_t|^2\right)^{1/2}
 \mathcal D_h(\pi\mid\pi')^{1/2}.
\end{align*}
Therefore, Young's inequality yields, for every $\varepsilon>0$,
\begin{equation}\label{eq:Rsigma-bmo}
 \E\int_0^T e^{\beta t}|\delta Y_t|\,|R_t^\sigma|\,\d t
 \le \varepsilon\E\left[\sup_{0\leq t\leq T}|\delta Y_t|^2\right]
 +C_\varepsilon\mathcal D_h(\pi\mid\pi').
\end{equation}

Now fix $\eta>0$. Substituting the decomposition
$\delta G=L+R^f+R^b+R^\sigma$ into
\eqref{eq:YZ-Ito-expectation}, and using
\eqref{eq:linear-generator-bound}, \eqref{eq:Rf-cross-term},
\eqref{eq:Rb-cross-term}, \eqref{eq:Rsigma-bmo}, and
\eqref{eq:stability:terminal:cond}, gives
\begin{align*}
 &\frac12\E\int_0^T e^{\beta t}|\delta Z_t|^2\,\d t
 +(\beta-C_0-2\eta)
   \E\int_0^T e^{\beta t}|\delta Y_t|^2\,\d t\\
 &\qquad\le
 C_{\eta,\beta,\varepsilon}\mathcal D_h(\pi\mid\pi')
 +2\varepsilon\E\left[
   \sup_{0\leq t\leq T}|\delta Y_t|^2
 \right].
\end{align*}
Choose $\beta\ge C_0+2\eta+1$. Since $e^{\beta t}\ge1$, the preceding
inequality, after renaming $\varepsilon$ and the constants, gives
\begin{equation}\label{eq:YZ-energy-bounded}
 \E\int_0^T\bigl(|\delta Y_t|^2+|\delta Z_t|^2\bigr)\,\d t
 \le C_\varepsilon\mathcal D_h(\pi\mid\pi')
 +\varepsilon \E\left[\sup_{0\leq t\leq T}|\delta Y_t|^2\right].
\end{equation}
We retain this estimate and derive a bound for the supremum before closing the
two inequalities.
For the supremum estimate, apply It\^o's formula to $|\delta Y|^2$ on
$[t,T]$:
\begin{equation}\label{eq:estimate:delta:Y:squared}
 |\delta Y_t|^2+\int_t^T|\delta Z_s|^2\,\d s
 =|\delta Y_T|^2
 +2\int_t^T\delta Y_s\trans\delta G_s\,\d s
 -2\int_t^T\delta Y_s\trans\delta Z_s\,\d W_s.
\end{equation}

We first control the finite-variation term. Since
\[
 \sup_{0\leq t\leq T}
 \left|\int_t^T\delta Y_s\trans\delta G_s\,\d s\right|
 \le \int_0^T|\delta Y_s\trans\delta G_s|\,\d s,
\]
the decomposition $\delta G=L+R^f+R^b+R^\sigma$, together with
\eqref{eq:linear-generator-bound}, \eqref{eq:Rf-cross-term},
\eqref{eq:Rb-cross-term}, and \eqref{eq:Rsigma-bmo}, applied with $\beta=0$,
yields, for every $\varepsilon>0$,
\begin{align}
 &2\E\left[
 \sup_{0\leq t\leq T}
 \left|\int_t^T\delta Y_s\trans\delta G_s\,\d s\right|
 \right]\nonumber\\
 &\qquad\le
 C\E\int_0^T\bigl(|\delta Y_s|^2+|\delta Z_s|^2\bigr)\,\d s
 +\varepsilon\E\left[\sup_{0\leq s\leq T}|\delta Y_s|^2\right]
 +C_\varepsilon\mathcal D_h(\pi\mid\pi').
 \label{eq:drift-supremum-bound}
\end{align}
Here $L$ contributes to the integral term, $R^f$ and $R^b$ contribute to
the integral and Bregman terms, and the BMO estimate for $R^\sigma$ produces
the supremum term together with another Bregman term.

For the martingale term in \eqref{eq:estimate:delta:Y:squared}, writing
$N_t:=\int_0^t\delta Y_s\trans\delta Z_s\,\d W_s$, then using $
 \sup_{0\leq t\leq T}|N_T-N_t|
 \le 2\sup_{0\leq t\leq T}|N_t|$, together with the Burkholder--Davis--Gundy inequality and Young's inequality gives
$$
 \E\left[\sup_{0\leq s\leq T}
 \left|\int_0^s\delta Y_t\trans\delta Z_t\,\d W_t\right|\right]
 \le \frac14\E\left[\sup_{0\leq t\leq T}|\delta Y_t|^2\right]
 +C\E\int_0^T|\delta Z_t|^2\,\d t.
$$
Taking the supremum over $t\in[0,T]$ in
\eqref{eq:estimate:delta:Y:squared}, and using
\eqref{eq:stability:terminal:cond}, \eqref{eq:drift-supremum-bound}, and the
preceding martingale estimate, gives
\[
 \E\left[\sup_{0\leq t\leq T}|\delta Y_t|^2\right]
 \le C_\varepsilon\mathcal D_h(\pi\mid\pi')
 +C\E\int_0^T\bigl(|\delta Y_t|^2+|\delta Z_t|^2\bigr)\,\d t
 +\varepsilon C\E\left[\sup_{0\leq t\leq T}|\delta Y_t|^2\right].
\]

To conclude, we combine this inequality with~\eqref{eq:YZ-energy-bounded} and choose
$\varepsilon$ sufficiently small so that the supremum terms is absorbed.
This proves simultaneously the desired estimates for
$\E[\sup_{0\leq t\leq T}|\delta Y_t|^2]$ and
$\E[\int_0^T|\delta Z_t|^2\,\d t]$ in case~(i).

\medskip
\noindent\emph{Case (ii): linear-growth affine regime.}
Under Assumption~\ref{ass:affine-dynamics},
\begin{align*}
 B_t^\pi=B_t^{\pi'}=B_t,\qquad
 &\Sigma_t^\pi=\Sigma_t^{\pi'}=\Sigma_t,\\
 \overline B_t^\pi=\overline B_t^{\pi'}=\overline B_t,\qquad
 &\overline\Sigma_t^\pi=\overline\Sigma_t^{\pi'}
   =\overline\Sigma_t.
\end{align*}
Therefore
\[
 R_t^b=R_t^\sigma=0.
\]
Thus $\delta G=L+R^f$. Using the estimate of the linear part
\eqref{eq:linear-generator-bound}, the estimate of the inhomogeneous part
\eqref{eq:Rf-cross-term}, and the terminal estimate
\eqref{eq:stability:terminal:cond} in~\eqref{eq:YZ-Ito-expectation}, we obtain,
for every $\eta>0$,
\[
 \frac12\E\int_0^T e^{\beta t}|\delta Z_t|^2\,\d t
 +(\beta-C_0-\eta)
   \E\int_0^T e^{\beta t}|\delta Y_t|^2\,\d t
 \le C_{\eta,\beta}\mathcal D_h(\pi\mid\pi').
\]
Choosing $\beta\ge C_0+\eta+1$ and using $e^{\beta t}\ge1$ gives
\begin{equation}\label{eq:YZ-energy-affine}
 \E\int_0^T\bigl(|\delta Y_t|^2+|\delta Z_t|^2\bigr)\,\d t
 \le C\mathcal D_h(\pi\mid\pi').
\end{equation}

For completeness, the finite-variation term in the supremum estimate satisfies
\[
 2\E\left[
 \sup_{0\leq t\leq T}
 \left|\int_t^T\delta Y_s\trans\delta G_s\,\d s\right|
 \right]
 \le C\E\int_0^T\bigl(|\delta Y_s|^2+|\delta Z_s|^2\bigr)\,\d s
 +C\mathcal D_h(\pi\mid\pi'),
\]
by~\eqref{eq:linear-generator-bound} and~\eqref{eq:Rf-cross-term} with
$\beta=0$. Applying the terminal estimate and the martingale estimate derived
in case~(i) to~\eqref{eq:estimate:delta:Y:squared}, and then using
\eqref{eq:YZ-energy-affine}, we obtain
\[
 \E\left[\sup_{0\leq t\leq T}|\delta Y_t|^2\right]
 \le C\mathcal D_h(\pi\mid\pi').
\]
Together with~\eqref{eq:YZ-energy-affine}, this proves the result in
case~(ii) and completes the proof.
\end{proof}

\subsection{Relative smoothness: proof of Theorem \ref{thm:relsmooth}}\label{sec:proof-relsmooth}
\begin{Lemma}[G\^ateaux derivative of $J^0$]\label{lem:gateaux}
Suppose Assumptions~\ref{ass:H1},~\ref{ass:H2},~\ref{ass:deriv} and
\ref{ass:measderiv} hold. Assume, in addition, that either
Assumption~\ref{ass:H3} or Assumption~\ref{ass:H3gen} holds.
Then, for all $\pi,\pi'\in\mathcal A_{\mathscr C}$ such that
$\mathcal D_h(\pi'\mid\pi)<\infty$,
\[
\d J^0(\pi)(\pi'-\pi):=\frac{\d}{\d\varepsilon}
J^0\bigl(\pi+\varepsilon(\pi'-\pi)\bigr)
\bigg|_{\varepsilon=0}
=
\E\int_0^T\int_A
\frac{\delta H_t^0}{\delta m}
\bigl(\Theta_t(\pi),\pi_t,a\bigr)
(\pi'_t-\pi_t)(\d a)\,\d t,
\]
where
\[
\Theta_t(\pi)
=
\bigl(
X_t^{\pi},
\mathcal L(X_t^{\pi}),
Y_t^{\pi},
Z_t^{\pi}
\bigr),
\]
and
\begin{align*}
\frac{\delta H_t^0}{\delta m}
\bigl(\Theta_t(\pi),\pi_t,a\bigr)
&=
\frac{\delta b_t}{\delta m}
\bigl(
X_t^{\pi},
\mathcal L(X_t^{\pi}),
\pi_t,a
\bigr)^{\trans}
Y_t^{\pi}
\\
&\quad
+
\operatorname{tr}\left(
\frac{\delta\sigma_t}{\delta m}
\bigl(
X_t^{\pi},
\mathcal L(X_t^{\pi}),
\pi_t,a
\bigr)^{\trans}
Z_t^{\pi}
\right)
\\
&\quad
+
\frac{\delta f_t}{\delta m}
\bigl(
X_t^{\pi},
\mathcal L(X_t^{\pi}),
\pi_t,a
\bigr).
\end{align*}
\end{Lemma}

\begin{proof}
Fix $\pi,\pi'\in\mathcal A_{\mathscr C}$ with
$\mathcal D_h(\pi'\mid\pi)<\infty$, and set
\[
\delta\pi_t:=\pi'_t-\pi_t,
\qquad
\pi_t^\varepsilon:=\pi_t+\varepsilon\delta\pi_t,
\]
\[
X:=X^{\pi},
\qquad
\mu_t:=\mathcal L(X_t),
\qquad
X^\varepsilon:=X^{\pi^\varepsilon},
\qquad
\mu_t^\varepsilon:=\mathcal L(X_t^\varepsilon).
\]
By convexity of $\mathscr C$,
$\pi^\varepsilon\in\mathcal A_{\mathscr C}$ for every
$\varepsilon\in[0,1]$.
Moreover, convexity of $D_h(\cdot\mid\nu_0)$ and
Assumption~\ref{ass:H2} give
\[
 \mathcal D_h(\pi^\varepsilon\mid\nu_0)
 \le (1-\varepsilon)\mathcal D_h(\pi\mid\nu_0)
 +\varepsilon\mathcal D_h(\pi'\mid\nu_0)<\infty.
\]
Thus the state equation is well posed along the whole segment
$(\pi^\varepsilon)_{0\leq\varepsilon\leq1}$.

\medskip
\noindent\emph{Step 1: Linearized MKV SDE.}

Let $V$ be the unique strong solution to the following linear MKV SDE 
\begin{equation}\label{eq:Vsde}
\begin{aligned}
\d V_t
={}&
\Bigg[
(D_xb_t)(X_t,\mu_t,\pi_t)V_t +
\widetilde{\E}\left[
(\partial_\mu b_t)
(X_t,\mu_t,\pi_t)(\widetilde X_t)
\,\widetilde V_t
\right]
\\
&\qquad
+
\int_A
\frac{\delta b_t}{\delta m}
(X_t,\mu_t,\pi_t,a)
\,\delta\pi_t(\d a)
\Bigg]\d t
\\
&+
\Bigg[
(D_x\sigma_t)(X_t,\mu_t,\pi_t)V_t +
\widetilde{\E}\left[
(\partial_\mu\sigma_t)
(X_t,\mu_t,\pi_t)(\widetilde X_t)
\,\widetilde V_t
\right]
\\
&\qquad
+
\int_A
\frac{\delta\sigma_t}{\delta m}
(X_t,\mu_t,\pi_t,a)
\,\delta\pi_t(\d a)
\Bigg]\d W_t,
\qquad
V_0=0.
\end{aligned}
\end{equation}
Here as usual $(\widetilde X,\widetilde V)$ is an independent copy of
$(X,V)$.

Under either Assumption~\ref{ass:H3} or Assumption~\ref{ass:H3gen},
$D_xb, \, D_x\sigma, \, \partial_\mu b$ and $\partial_\mu\sigma$
are uniformly bounded. Moreover,
Assumption~\ref{ass:measderiv} gives
\begin{align*}
&
\left|
\int_A
\frac{\delta b_t}{\delta m}
(X_t,\mu_t,\pi_t,a)
\delta\pi_t(\d a)
\right|^2
\\
&\qquad+
\left|
\int_A
\frac{\delta\sigma_t}{\delta m}
(X_t,\mu_t,\pi_t,a)
\delta\pi_t(\d a)
\right|^2
\le
2K D_h(\pi'_t\mid\pi_t).
\end{align*}
Since $\mathcal{D}_h(\pi' \mid \pi )\,\d t<\infty$, standard estimates for linear McKean--Vlasov SDEs imply
\[
\E\left[
\sup_{0\le t\le T}|V_t|^2
\right]
<\infty.
\]

\medskip
\noindent\emph{Step 2: Differentiability of the state process.}

We claim that
\begin{equation}\label{eq:Vderivative}
\lim_{\varepsilon\downarrow0}
\E\left[
\sup_{0\le t\le T}
\left|
\frac{X_t^\varepsilon-X_t}{\varepsilon}
-
V_t
\right|^2
\right]
=0.
\end{equation}

Set
\[
V_t^\varepsilon
:=
\frac{X_t^\varepsilon-X_t}{\varepsilon}
-
V_t,
\qquad
\Delta X_t^\varepsilon
:=
X_t^\varepsilon-X_t.
\]
For $\lambda\in[0,1]$, define
\[
X_t^{\varepsilon,\lambda}
:=
X_t+\lambda\Delta X_t^\varepsilon, \quad \mbox{ and } \mu_t^{\varepsilon,\lambda}
:=
\mathcal L\bigl(X_t^{\varepsilon,\lambda}\bigr).
\]
In particular,
$\mu_t^{\varepsilon,0}=\mu_t$ and
$\mu_t^{\varepsilon,1}=\mu_t^\varepsilon$.

We decompose
\begin{align}
b_t
\bigl(
X_t^\varepsilon,
\mu_t^\varepsilon,
\pi_t^\varepsilon
\bigr)
-
b_t(X_t,\mu_t,\pi_t)
&=
\Big[
b_t
\bigl(
X_t^\varepsilon,
\mu_t^\varepsilon,
\pi_t^\varepsilon
\bigr)
-
b_t(X_t,\mu_t,\pi_t^\varepsilon)
\Big]
\nonumber\\
&\quad+
\Big[
b_t(X_t,\mu_t,\pi_t^\varepsilon)
-
b_t(X_t,\mu_t,\pi_t)
\Big].
\label{eq:b-split}
\end{align}

For the first bracket, we write
\begin{align}
& b_t
\bigl(
X_t^\varepsilon,
\mu_t^\varepsilon,
\pi_t^\varepsilon
\bigr)
-
b_t(X_t,\mu_t,\pi_t^\varepsilon)
\nonumber \\
& =
\int_0^1
(D_xb_t)
\bigl(
X_t^{\varepsilon,\lambda},
\mu_t^{\varepsilon,\lambda},
\pi_t^\varepsilon
\bigr)
\Delta X_t^\varepsilon
\,\d\lambda
\nonumber\\
&\quad+
\int_0^1
\widetilde{\E}\left[
(\partial_\mu b_t)
\bigl(
X_t^{\varepsilon,\lambda},
\mu_t^{\varepsilon,\lambda},
\pi_t^\varepsilon
\bigr)
\bigl(
\widetilde X_t^{\varepsilon,\lambda}
\bigr)
\Delta\widetilde X_t^\varepsilon
\right]
\,\d\lambda.
\label{eq:b-state-law-expansion}
\end{align}

For the second bracket,
\begin{align}
\frac{
b_t(X_t,\mu_t,\pi_t^\varepsilon)
-
b_t(X_t,\mu_t,\pi_t)
}{
\varepsilon
}
&=
\int_0^1\int_A
\frac{\delta b_t}{\delta m}
\bigl(
X_t,\mu_t,
\pi_t^{\lambda\varepsilon},
a
\bigr)
\delta\pi_t(\d a)
\,\d\lambda.
\label{eq:b-control-expansion}
\end{align}

Define
\[
A_t^{b,\varepsilon}
:=
\int_0^1
(D_xb_t)
\bigl(
X_t^{\varepsilon,\lambda},
\mu_t^{\varepsilon,\lambda},
\pi_t^\varepsilon
\bigr)
\,\d\lambda,
\]
and, for a copy variable $\widetilde U$,
\[
\mathcal B_t^{b,\varepsilon}(\widetilde U)
:=
\int_0^1
\widetilde{\E}\left[
(\partial_\mu b_t)
\bigl(
X_t^{\varepsilon,\lambda},
\mu_t^{\varepsilon,\lambda},
\pi_t^\varepsilon
\bigr)
\bigl(
\widetilde X_t^{\varepsilon,\lambda}
\bigr)
\widetilde U
\right]
\,\d\lambda.
\]
Also set
\[
Q_t^{b,\varepsilon}
:=
\int_0^1\int_A
\frac{\delta b_t}{\delta m}
\bigl(
X_t,\mu_t,
\pi_t^{\lambda\varepsilon},
a
\bigr)
\delta\pi_t(\d a)
\,\d\lambda,
\]
and
\[
Q_t^{b,0}
:=
\int_A
\frac{\delta b_t}{\delta m}
(X_t,\mu_t,\pi_t,a)
\delta\pi_t(\d a).
\]

Then, with the above notations, \eqref{eq:b-split} together with \eqref{eq:b-state-law-expansion} and \eqref{eq:b-control-expansion} gives
\begin{align}
&
\frac{
b_t
\bigl(
X_t^\varepsilon,
\mu_t^\varepsilon,
\pi_t^\varepsilon
\bigr)
-
b_t(X_t,\mu_t,\pi_t)
}{
\varepsilon
}
\nonumber\\
&\quad
-
(D_xb_t)(X_t,\mu_t,\pi_t)V_t
\nonumber\\
&\quad
-
\widetilde{\E}\left[
(\partial_\mu b_t)
(X_t,\mu_t,\pi_t)(\widetilde X_t)
\widetilde V_t
\right]
-
Q_t^{b,0}
\nonumber\\
&=
A_t^{b,\varepsilon}V_t^\varepsilon
+
\mathcal B_t^{b,\varepsilon}
\bigl(
\widetilde V_t^\varepsilon
\bigr)
+
R_t^{b,\varepsilon},
\label{eq:b-remainder}
\end{align}
where
\begin{align*}
R_t^{b,\varepsilon}
={}&
\Big[
A_t^{b,\varepsilon}
-
(D_xb_t)(X_t,\mu_t,\pi_t)
\Big]V_t
\\
&+
\int_0^1
\widetilde{\E}\Bigg[
\Bigg(
(\partial_\mu b_t)
\bigl(
X_t^{\varepsilon,\lambda},
\mu_t^{\varepsilon,\lambda},
\pi_t^\varepsilon
\bigr)
\bigl(
\widetilde X_t^{\varepsilon,\lambda}
\bigr)
\\
&\hspace{42mm}
-
(\partial_\mu b_t)
(X_t,\mu_t,\pi_t)
(\widetilde X_t)
\Bigg)
\widetilde V_t
\Bigg]
\,\d\lambda
\\
&+
Q_t^{b,\varepsilon}
-
Q_t^{b,0}.
\end{align*}

The same construction with $\sigma$ in place of $b$ yields
$A_t^{\sigma,\varepsilon}$,
$\mathcal B_t^{\sigma,\varepsilon}$,
$Q_t^{\sigma,\varepsilon}$, $Q_t^{\sigma,0}$ and
$R_t^{\sigma,\varepsilon}$.

Therefore
\begin{align*}
\d V_t^\varepsilon
={}&
\Big[
A_t^{b,\varepsilon}V_t^\varepsilon
+
\mathcal B_t^{b,\varepsilon}
(\widetilde V_t^\varepsilon)
+
R_t^{b,\varepsilon}
\Big]\d t
\\
&+
\Big[
A_t^{\sigma,\varepsilon}V_t^\varepsilon
+
\mathcal B_t^{\sigma,\varepsilon}
(\widetilde V_t^\varepsilon)
+
R_t^{\sigma,\varepsilon}
\Big]\d W_t.
\end{align*}

Since the linear coefficients are uniformly bounded, the BDG and
Gr\"onwall inequalities yield
\begin{equation}\label{eq:Vepsbound}
\E\left[
\sup_{0\le t\le T}
|V_t^\varepsilon|^2
\right]
\le
C
\E\int_0^T
\left(
|R_t^{b,\varepsilon}|^2
+
|R_t^{\sigma,\varepsilon}|^2
\right)
\d t.
\end{equation}

By Lemma~\ref{lem:Xstab},
\[
\E\left[
\sup_{0\le t\le T}
|X_t^\varepsilon-X_t|^2
\right]
\le
C\mathcal D_h(\pi^\varepsilon\mid\pi).
\]
Since $\pi \mapsto \mathcal{D}_h(\pi \mid\pi_t)$ is convex,
\[
\mathcal D_h(\pi^\varepsilon\mid\pi)
\le
\varepsilon\mathcal D_h(\pi'\mid\pi).
\]
Hence
\[
\E\left[
\sup_{0\le t\le T}
|X_t^\varepsilon-X_t|^2
\right]
\longrightarrow0.
\]

Moreover,
\[
W_2^2
\bigl(
\mu_t^{\varepsilon,\lambda},
\mu_t
\bigr)
\le
\E
\left|
X_t^{\varepsilon,\lambda}
-
X_t
\right|^2
\longrightarrow0.
\]

By Assumption~\ref{ass:deriv}, the first two components of
$R_t^{b,\varepsilon}$ and $R_t^{\sigma,\varepsilon}$
converge to zero in
$L^2(\Omega\times[0,T])$.

For the control remainder, fix $\varphi\in\{b,\sigma\}$ and set
\[
 \Gamma_t^\varphi(s)
 :=\int_A
 \frac{\delta\varphi_t}{\delta m}
 \bigl(X_t,\mu_t,\pi_t^s,a\bigr)
 \delta\pi_t(\d a),
 \qquad s\in[0,1].
\]
Then
\[
 Q_t^{\varphi,\varepsilon}
 =\int_0^1\Gamma_t^\varphi(\lambda\varepsilon)\,\d\lambda,
 \qquad
 Q_t^{\varphi,0}=\Gamma_t^\varphi(0).
\]
The continuity along line segments in Assumption~\ref{ass:measderiv}
therefore gives
$\Gamma_t^\varphi(\lambda\varepsilon)\to\Gamma_t^\varphi(0)$
for every $\lambda\in[0,1]$, almost everywhere in $(\omega,t)$. Moreover,
the first bound in that assumption yields
\[
 |\Gamma_t^\varphi(s)|^2
 \le K D_h(\pi'_t\mid\pi_t),
 \qquad s\in[0,1].
\]
Hence, by Jensen's inequality,
\[
 |Q_t^{\varphi,\varepsilon}-Q_t^{\varphi,0}|^2
 \le
 \int_0^1
 |\Gamma_t^\varphi(\lambda\varepsilon)-\Gamma_t^\varphi(0)|^2
 \,\d\lambda
 \le 4K D_h(\pi'_t\mid\pi_t).
\]
Since the last term is integrable, the dominated convergence theorem gives
\begin{equation}\label{eq:Q-control-remainder}
 \E\int_0^T
 |Q_t^{\varphi,\varepsilon}-Q_t^{\varphi,0}|^2\,\d t
 \longrightarrow0,
 \qquad \varphi\in\{b,\sigma\}.
\end{equation}
%Thus this part of the argument uses only the first-order flat derivative,
%its continuity along control segments, and its directional bound from
%Assumption~\ref{ass:measderiv}.

Consequently,
\[
\E\int_0^T
\left(
|R_t^{b,\varepsilon}|^2
+
|R_t^{\sigma,\varepsilon}|^2
\right)
\d t
\longrightarrow0.
\]
By~\eqref{eq:Vepsbound}, this proves
\eqref{eq:Vderivative}.

\medskip
\noindent\emph{Step 3: First variation of the cost.}

Write
\begin{align*}
&
f_t
\bigl(
X_t^\varepsilon,
\mu_t^\varepsilon,
\pi_t^\varepsilon
\bigr)
-
f_t(X_t,\mu_t,\pi_t)
\\
&=
\Big[
f_t
\bigl(
X_t^\varepsilon,
\mu_t^\varepsilon,
\pi_t^\varepsilon
\bigr)
-
f_t(X_t,\mu_t,\pi_t^\varepsilon)
\Big]
\\
&\quad+
\Big[
f_t(X_t,\mu_t,\pi_t^\varepsilon)
-
f_t(X_t,\mu_t,\pi_t)
\Big].
\end{align*}
The chain rule together with definition of the flat derivative give
\begin{align*}
&
\frac{
f_t
\bigl(
X_t^\varepsilon,
\mu_t^\varepsilon,
\pi_t^\varepsilon
\bigr)
-
f_t(X_t,\mu_t,\pi_t)
}{
\varepsilon
}
\\
&=
\int_0^1
\Bigg[
\left\langle
(D_xf_t)
\bigl(
X_t^{\varepsilon,\lambda},
\mu_t^{\varepsilon,\lambda},
\pi_t^\varepsilon
\bigr),
V_t^\varepsilon+V_t
\right\rangle
\\
&\qquad+
\widetilde{\E}\left[
\left\langle
(\partial_\mu f_t)
\bigl(
X_t^{\varepsilon,\lambda},
\mu_t^{\varepsilon,\lambda},
\pi_t^\varepsilon
\bigr)
\bigl(
\widetilde X_t^{\varepsilon,\lambda}
\bigr),
\widetilde V_t^\varepsilon+\widetilde V_t
\right\rangle
\right]
\Bigg]\d\lambda
\\
&\quad+
\int_0^1\int_A
\frac{\delta f_t}{\delta m}
\bigl(
X_t,\mu_t,
\pi_t^{\lambda\varepsilon},
a
\bigr)
\delta\pi_t(\d a)
\,\d\lambda.
\end{align*}

By Step~2, $V^\varepsilon+V\to V$ in $S^2$ and, in particular,
there exists $\varepsilon_0>0$ such that
$\sup_{0<\varepsilon\leq\varepsilon_0}
\|V^\varepsilon+V\|_{S^2}<\infty$. 
%Moreover,
%Lemma~\ref{lem:Xstab}, together with
%$\mathcal D_h(\pi^\varepsilon\mid\pi)\leq
%\varepsilon\mathcal D_h(\pi'\mid\pi)$ imply
%\[
% \sup_{0\leq\varepsilon,\lambda\leq1}
% \E\left[\sup_{0\leq t\leq T}
% |X_t^{\varepsilon,\lambda}|^2\right]<\infty.
%\]
%In the linear-growth regime, we also have
%\[
% M_{2,A}(\pi_t^\varepsilon)^2
% =(1-\varepsilon)M_{2,A}(\pi_t)^2
% +\varepsilon M_{2,A}(\pi'_t)^2.
%\]
%Consequently, Assumption~\ref{ass:H3gen} gives the same square-integrability
%along the segment as the boundedness assumption~\ref{ass:H3}. More precisely,
Let
\[
 F_t^{\varepsilon,\lambda}
 :=D_xf_t(X_t^{\varepsilon,\lambda},
          \mu_t^{\varepsilon,\lambda},\pi_t^\varepsilon), \qquad
 \overline F_t^{\varepsilon,\lambda}
 :=(\partial_\mu f_t)(X_t^{\varepsilon,\lambda},
          \mu_t^{\varepsilon,\lambda},\pi_t^\varepsilon)
          (\widetilde X_t^{\varepsilon,\lambda}),
\]
then, in either regime,
\begin{equation}\label{eq:cost-derivatives-L2-segment}
 \sup_{0\leq\varepsilon,\lambda\leq1}
 \E\int_0^T\left(
 |F_t^{\varepsilon,\lambda}|^2
 +\widetilde\E|\overline F_t^{\varepsilon,\lambda}|^2
 \right)\,\d t<\infty.
\end{equation}
Assumption~\ref{ass:deriv}, Step~2, and
$\mathcal D_h(\pi^\varepsilon\mid\pi)\to0$ also yield
\begin{align*}
 &\int_0^1\E\int_0^T
 |F_t^{\varepsilon,\lambda}-F_t^{0,0}|^2\,\d t\,\d\lambda
 \longrightarrow0, \qquad \int_0^1\E\widetilde\E\int_0^T
 |\overline F_t^{\varepsilon,\lambda}
   -\overline F_t^{0,0}|^2\,\d t\,\d\lambda
 \longrightarrow0.
\end{align*}
Then, by Cauchy--Schwarz's inequality
\begin{align*}
 &\E\int_0^T\int_0^1
 \left|
 \langle F_t^{\varepsilon,\lambda},V_t^\varepsilon + V_t\rangle
 -\langle F_t^{0,0},V_t\rangle
 \right|\,\d\lambda\,\d t\\
 &\quad\le
 \left(\int_0^1\E\int_0^T
 |F_t^{\varepsilon,\lambda}|^2\,\d t\,\d\lambda\right)^{1/2}
 \|V^\varepsilon\|_{H^2}\\
 &\qquad+
 \left(\int_0^1\E\int_0^T
 |F_t^{\varepsilon,\lambda}-F_t^{0,0}|^2\,\d t\,\d\lambda
 \right)^{1/2}
 \|V\|_{H^2},
\end{align*}
which tends to zero as $\varepsilon$ goes to zero. The term involving $\partial_\mu f_t$ is treated in the
same way on the product space. Finally, applying the argument leading to
\eqref{eq:Q-control-remainder} with $\varphi=f$ shows that
\[
 \int_0^1\int_A
 \frac{\delta f_t}{\delta m}
 (X_t,\mu_t,\pi_t^{\lambda\varepsilon},a)
 \delta\pi_t(\d a)\,\d\lambda
 \longrightarrow
 \int_A\frac{\delta f_t}{\delta m}
 (X_t,\mu_t,\pi_t,a)\delta\pi_t(\d a)
\]
in $L^2(\Omega\times[0,T])$. We have therefore proved
\begin{align}
\frac{
f_t
\bigl(
X_t^\varepsilon,
\mu_t^\varepsilon,
\pi_t^\varepsilon
\bigr)
-
f_t(X_t,\mu_t,\pi_t)
}{
\varepsilon
}
&\longrightarrow
\langle
D_xf_t(X_t,\mu_t,\pi_t),
V_t
\rangle +
\widetilde{\E}\left[
\left\langle
(\partial_\mu f_t)
(X_t,\mu_t,\pi_t)
(\widetilde X_t),
\widetilde V_t
\right\rangle
\right]
\nonumber\\
&\quad+
\int_A
\frac{\delta f_t}{\delta m}
(X_t,\mu_t,\pi_t,a)
\delta\pi_t(\d a)
\label{eq:f-first-variation}
\end{align}
in $L^1(\Omega\times[0,T])$.

For the terminal cost, the chain rule gives
\begin{align*}
&\frac{g(X_T^\varepsilon,\mu_T^\varepsilon)
-g(X_T,\mu_T)}{\varepsilon}
\\
&\quad=\int_0^1\Bigg[
\left\langle
D_xg(X_T^{\varepsilon,\lambda},\mu_T^{\varepsilon,\lambda}),
V_T^\varepsilon+V_T
\right\rangle
\\
&\hspace{37mm}+
\widetilde\E\left[
\left\langle
(\partial_\mu g)(X_T^{\varepsilon,\lambda},
\mu_T^{\varepsilon,\lambda})(\widetilde X_T^{\varepsilon,\lambda}),
\widetilde V_T^\varepsilon+\widetilde V_T
\right\rangle
\right]\Bigg]\,\d\lambda.
\end{align*}
Set
\[
 G_T^{\varepsilon,\lambda}
 :=D_xg(X_T^{\varepsilon,\lambda},\mu_T^{\varepsilon,\lambda}),
 \qquad
 \overline G_T^{\varepsilon,\lambda}
 :=(\partial_\mu g)(X_T^{\varepsilon,\lambda},
 \mu_T^{\varepsilon,\lambda})(\widetilde X_T^{\varepsilon,\lambda}).
\]
Since
$\mathcal D_h(\pi^\varepsilon\mid\pi)
\leq\varepsilon\mathcal D_h(\pi'\mid\pi)$,
Lemma~\ref{lem:Xstab} implies
\[
 \sup_{0\leq\varepsilon,\lambda\leq1}
 \E|X_T^{\varepsilon,\lambda}|^2<\infty,
 \qquad
 \sup_{0\leq\varepsilon,\lambda\leq1}
 M_2(\mu_T^{\varepsilon,\lambda})^2<\infty.
\]
Consequently, the boundedness assumption~\ref{ass:H3}, or the linear-growth
bounds of Assumption~\ref{ass:H3gen}, respectively, gives
\begin{equation}\label{eq:terminal-derivatives-L2-segment}
 \sup_{0\leq\varepsilon,\lambda\leq1}
 \E\left[
 |G_T^{\varepsilon,\lambda}|^2
 +\widetilde\E|\overline G_T^{\varepsilon,\lambda}|^2
 \right]<\infty.
\end{equation}
Moreover, Assumption~\ref{ass:deriv} and the convergence established in
Step~2 yield
\begin{align}
 &\int_0^1\E
 |G_T^{\varepsilon,\lambda}-G_T^{0,0}|^2\,\d\lambda
 \longrightarrow0,
 \nonumber\\
 &\int_0^1\E\widetilde\E
 |\overline G_T^{\varepsilon,\lambda}
 -\overline G_T^{0,0}|^2\,\d\lambda
 \longrightarrow0.
 \label{eq:terminal-derivatives-L2-convergence}
\end{align}
More precisely, Assumption~\ref{ass:deriv} gives
\begin{align*}
 \E|G_T^{\varepsilon,\lambda}-G_T^{0,0}|^2
 &\le K\left(
 \E|X_T^{\varepsilon,\lambda}-X_T|^2
 +W_2(\mu_T^{\varepsilon,\lambda},\mu_T)^2
 \right),
 \\
 \E\widetilde\E
 |\overline G_T^{\varepsilon,\lambda}-\overline G_T^{0,0}|^2
 &\le K\left(
 \E|X_T^{\varepsilon,\lambda}-X_T|^2
 +W_2(\mu_T^{\varepsilon,\lambda},\mu_T)^2
 \right.
 \\
 &\hspace{30mm}\left.
 +\E\widetilde\E
 |\widetilde X_T^{\varepsilon,\lambda}-\widetilde X_T|^2
 \right),
\end{align*}
and all terms on the right-hand sides converge to zero by Step~2.

To conclude, the Cauchy--Schwarz inequality gives
\begin{align*}
&\E\int_0^1
\left|
\langle G_T^{\varepsilon,\lambda},V_T^\varepsilon+V_T\rangle
-\langle G_T^{0,0},V_T\rangle
\right|\,\d\lambda
\\
&\quad\leq
\left(\int_0^1\E|G_T^{\varepsilon,\lambda}|^2\,\d\lambda\right)^{1/2}
\|V_T^\varepsilon\|_{L^2}
+
\left(\int_0^1\E
|G_T^{\varepsilon,\lambda}-G_T^{0,0}|^2\,\d\lambda\right)^{1/2}
\|V_T\|_{L^2},
\end{align*}
and, using Jensen's inequality for the independent-copy expectation,
\begin{align*}
&\E\int_0^1
\left|
\widetilde\E\left[
\langle\overline G_T^{\varepsilon,\lambda},
\widetilde V_T^\varepsilon+\widetilde V_T\rangle
-\langle\overline G_T^{0,0},\widetilde V_T\rangle
\right]
\right|\,\d\lambda
\\
&\quad\leq
\left(\int_0^1\E\widetilde\E
|\overline G_T^{\varepsilon,\lambda}|^2\,\d\lambda\right)^{1/2}
\|V_T^\varepsilon\|_{L^2}
\\
&\qquad+
\left(\int_0^1\E\widetilde\E
|\overline G_T^{\varepsilon,\lambda}
-\overline G_T^{0,0}|^2\,\d\lambda\right)^{1/2}
\|V_T\|_{L^2}.
\end{align*}
By~\eqref{eq:Vderivative},
\eqref{eq:terminal-derivatives-L2-segment}, and
\eqref{eq:terminal-derivatives-L2-convergence}, both right-hand sides
converge to zero. Therefore,
\begin{align}
&
\frac{
g
\bigl(
X_T^\varepsilon,
\mu_T^\varepsilon
\bigr)
-
g(X_T,\mu_T)
}{
\varepsilon
}
\nonumber\\
&\longrightarrow
\langle
D_xg(X_T,\mu_T),
V_T
\rangle
+
\widetilde{\E}\left[
\left\langle
(\partial_\mu g)
(X_T,\mu_T)
(\widetilde X_T),
\widetilde V_T
\right\rangle
\right]
\label{eq:g-first-variation}
\end{align}
in $L^1(\Omega)$.

Consequently,
\begin{equation}\label{eq:Jderiv}
\begin{aligned}
\lim_{\varepsilon \downarrow 0} \frac{ J^0\bigl(\pi+\varepsilon(\pi'-\pi)\bigr) - J^0\bigl(\pi \bigr)} {\varepsilon}
={}&
\E\int_0^T
\Bigg[
\langle
D_xf_t,
V_t
\rangle
\\
&\qquad+
\widetilde{\E}\left[
\left\langle
(\partial_\mu f_t)
(X_t,\mu_t,\pi_t)
(\widetilde X_t),
\widetilde V_t
\right\rangle
\right]
\\
&\qquad+
\int_A
\frac{\delta f_t}{\delta m}
(X_t,\mu_t,\pi_t,a)
\delta\pi_t(\d a)
\Bigg]\d t
\\
&+
\E
\langle
D_xg(X_T,\mu_T),
V_T
\rangle
\\
&+
\E\widetilde{\E}\left[
\left\langle
(\partial_\mu g)
(X_T,\mu_T)
(\widetilde X_T),
\widetilde V_T
\right\rangle
\right].
\end{aligned}
\end{equation}

\medskip
\noindent\emph{Step 4: Duality with the adjoint equation.}

Let $(Y,Z)=(Y^{\pi},Z^{\pi})$ denote the adjoint process associated
with $\pi$. By Proposition~\ref{prop:wellposed}, $(Y,Z)\in S^2\times H^2$
in either regime. Define
\[
n_t^b
:=
\widetilde{\E}\left[
(\partial_\mu b_t)
(\widetilde X_t,\mu_t,\widetilde\pi_t)
(X_t)^{\trans}
\widetilde Y_t
\right],
\]
\[
n_t^\sigma
:=
\widetilde{\E}\left[
(\partial_\mu\sigma_t)
(\widetilde X_t,\mu_t,\widetilde\pi_t)
(X_t)^{\trans}
\widetilde Z_t
\right],
\]
\[
n_t^f
:=
\widetilde{\E}\left[
(\partial_\mu f_t)
(\widetilde X_t,\mu_t,\widetilde\pi_t)
(X_t)
\right],
\]
and
\[
n_T^g
:=
\widetilde{\E}\left[
(\partial_\mu g)
(\widetilde X_T,\mu_T)
(X_T)
\right].
\]

Then the adjoint BSDE reads
\begin{equation}\label{eq:adjoint-gateaux}
\begin{cases}
\d Y_t
=
-\Big[
(D_xb_t)^{\trans}Y_t
+
(D_x\sigma_t)^{\trans}Z_t
+
D_xf_t
+
n_t^b
+
n_t^\sigma
+
n_t^f
\Big]\d t
+
Z_t\,\d W_t,
\\[1mm]
Y_T
=
D_xg(X_T,\mu_T)
+
n_T^g.
\end{cases}
\end{equation}

We now use the independent-copy identities. By Fubini's theorem
and exchange of $(\omega,\widetilde\omega)$,
\begin{align}
&
\E\left\langle
Y_t,
\widetilde{\E}\left[
(\partial_\mu b_t)
(X_t,\mu_t,\pi_t)
(\widetilde X_t)
\widetilde V_t
\right]
\right\rangle
\nonumber\\
&=
\E\widetilde{\E}\left[
\left\langle
(\partial_\mu b_t)
(X_t,\mu_t,\pi_t)
(\widetilde X_t)^{\trans}
Y_t,
\widetilde V_t
\right\rangle
\right]
\nonumber\\
&=
\E\widetilde{\E}\left[
\left\langle
(\partial_\mu b_t)
(\widetilde X_t,\mu_t,\widetilde\pi_t)
(X_t)^{\trans}
\widetilde Y_t,
V_t
\right\rangle
\right]
\nonumber\\
&=
\E\langle n_t^b,V_t\rangle.
\label{eq:copy-b}
\end{align}
Similarly,
\begin{equation}\label{eq:copy-sigma}
\E\left\langle
Z_t,
\widetilde{\E}\left[
(\partial_\mu\sigma_t)
(X_t,\mu_t,\pi_t)
(\widetilde X_t)
\widetilde V_t
\right]
\right\rangle_F
=
\E\langle n_t^\sigma,V_t\rangle,
\end{equation}

\begin{equation}\label{eq:copy-f}
\E\widetilde{\E}\left[
\left\langle
(\partial_\mu f_t)
(X_t,\mu_t,\pi_t)
(\widetilde X_t),
\widetilde V_t
\right\rangle
\right]
=
\E\langle n_t^f,V_t\rangle,
\end{equation}
and,
\begin{equation}\label{eq:copy-g}
\E\widetilde{\E}\left[
\left\langle
(\partial_\mu g)
(X_T,\mu_T)
(\widetilde X_T),
\widetilde V_T
\right\rangle
\right]
=
\E\langle n_T^g,V_T\rangle.
\end{equation}

To justify the duality argument, let $\Gamma^V$ denote the diffusion
coefficient in~\eqref{eq:Vsde}:
\[
 \Gamma_t^V
 :=(D_x\sigma_t)(X_t,\mu_t,\pi_t)V_t
 +\widetilde\E\left[
   (\partial_\mu\sigma_t)(X_t,\mu_t,\pi_t)(\widetilde X_t)
   \widetilde V_t
 \right]
 +Q_t^{\sigma,0}.
\]
The boundedness of the derivatives of $\sigma$, the fact that $V\in S^2$,
and the first bound in Assumption~\ref{ass:measderiv} imply that
$\Gamma^V\in H^2$ which in turn yields
\begin{align*}
 &\E\left[\left(\int_0^T|Y_t|^2|\Gamma_t^V|^2\,\d t\right)^{1/2}\right]
 \le \|Y\|_{S^2}\|\Gamma^V\|_{H^2}<\infty.
 \end{align*}
 
 Moreover, it holds
 \begin{align*}
 &\E\left[\left(\int_0^T|V_t|^2|Z_t|^2\,\d t\right)^{1/2}\right]
 \le \|V\|_{S^2}\|Z\|_{H^2}<\infty.
\end{align*}
Hence, the local martingale in It\^o's formula for
$\langle Y_t,V_t\rangle$ is an $\mathcal H^1$-martingale and has zero
expectation. Applying that formula and using the independent-copy identities
above yields
\begin{align}
\E\langle Y_T,V_T\rangle
={}&
\E\int_0^T
\Bigg[
-
\langle
D_xf_t+n_t^f,
V_t
\rangle
\nonumber\\
&\qquad+
\left\langle
Y_t,
\int_A
\frac{\delta b_t}{\delta m}
(X_t,\mu_t,\pi_t,a)
\delta\pi_t(\d a)
\right\rangle
\nonumber\\
&\qquad+
\left\langle
Z_t,
\int_A
\frac{\delta\sigma_t}{\delta m}
(X_t,\mu_t,\pi_t,a)
\delta\pi_t(\d a)
\right\rangle_F
\Bigg]\d t.
\label{eq:itoYV}
\end{align}

By the terminal condition and~\eqref{eq:copy-g},
\begin{align*}
&
\E\langle
D_xg(X_T,\mu_T),
V_T
\rangle+
\E\widetilde{\E}\left[
\left\langle
(\partial_\mu g)
(X_T,\mu_T)
(\widetilde X_T),
\widetilde V_T
\right\rangle
\right]
=
\E\langle Y_T,V_T\rangle.
\end{align*}

Substituting~\eqref{eq:itoYV} into~\eqref{eq:Jderiv} and using
\eqref{eq:copy-f}, all terms involving $V$ cancel. Hence
\begin{align*}
\frac{\d}{\d\varepsilon}
J^0\bigl(\pi+\varepsilon(\pi'-\pi)\bigr)
\bigg|_{\varepsilon=0}
={}&
\E\int_0^T\int_A
\Bigg[
\frac{\delta b_t}{\delta m}
(X_t,\mu_t,\pi_t,a)^{\trans}
Y_t +
\operatorname{tr}\left(
\frac{\delta\sigma_t}{\delta m}
(X_t,\mu_t,\pi_t,a)^{\trans}
Z_t
\right)
\\
&\qquad+
\frac{\delta f_t}{\delta m}
(X_t,\mu_t,\pi_t,a)
\Bigg]
\delta\pi_t(\d a)
\,\d t \\
& =
\E\int_0^T\int_A
\frac{\delta H_t^0}{\delta m}
\bigl(
\Theta_t(\pi),
\pi_t,a
\bigr)
(\pi'_t-\pi_t)(\d a)
\,\d t.
\end{align*}
This completes the proof.
\end{proof}

\begin{Lemma}[Regularity of the Hamiltonian gradient]\label{lem:reg}
Suppose Assumptions~\ref{ass:H1},~\ref{ass:H2},~\ref{ass:deriv},
\ref{ass:measderiv} and~\ref{ass:relative-smoothness-regularity} hold. Assume,
in addition, that one of the following two alternative regimes is satisfied:
\begin{enumerate}
\item[(i)] \emph{Bounded-cost regime:}
Assumption~\ref{ass:H3} holds;
\item[(ii)] \emph{Linear-growth affine regime:}
Assumptions~\ref{ass:H3gen} and~\ref{ass:affine-dynamics} hold, and, for every
$(t,x,\mu)$, the map $m\mapsto b_t(x,\mu,m)$ is affine on $\mathscr C$ or equivalently $\mathcal R_t^b(x,\mu;m'\mid m)=0$ for all $m,m'\in\mathscr C$ with the same notations as in Assumption~\ref{ass:relative-smoothness-regularity}.

\end{enumerate}
Then there exists a constant $C<\infty$ such that, for all $\pi,\pi'\in\AC$ with
$\mathcal D_h(\pi'\mid\pi)<\infty$, setting
$\pi^\varepsilon=(1-\varepsilon)\pi+\varepsilon\pi'$ for
$\varepsilon\in[0,1]$, we have
\begin{align*}
\Bigg|\E\int_0^1\int_0^T\int_A
&\left[
\frac{\delta H^0_t}{\delta m}
 \bigl(\Theta_t^{\pi^\varepsilon},\pi_t^\varepsilon,a\bigr)
-\frac{\delta H^0_t}{\delta m}
 \bigl(\Theta_t^\pi,\pi_t,a\bigr)
\right]
(\pi'_t-\pi_t)(\d a)\,\d t\,\d\varepsilon\Bigg|
\le C\mathcal D_h(\pi'\mid\pi).
\end{align*}
\end{Lemma}

\begin{proof}
Write $\mu_t:=\mathcal L(X_t^\pi)$,
$\mu_t^\varepsilon:=\mathcal L(X_t^{\pi^\varepsilon})$, and
$q_t:=\pi'_t-\pi_t$. The expression inside the absolute value is
$I_1+I_2$, where
\begin{align*}
 I_1&:=\E\int_0^1\int_0^T\int_A
 \left[
 \frac{\delta H_t^0}{\delta m}
  (\Theta_t^{\pi^\varepsilon},\pi_t^\varepsilon,a)
 -\frac{\delta H_t^0}{\delta m}
  (\Theta_t^\pi,\pi_t^\varepsilon,a)
 \right]q_t(\d a)\,\d t\,\d\varepsilon,\\
 I_2&:=\E\int_0^1\int_0^T\int_A
 \left[
 \frac{\delta H_t^0}{\delta m}
  (\Theta_t^\pi,\pi_t^\varepsilon,a)
 -\frac{\delta H_t^0}{\delta m}
  (\Theta_t^\pi,\pi_t,a)
 \right]q_t(\d a)\,\d t\,\d\varepsilon.
\end{align*}

\noindent\emph{Common estimate for the state-adjoint contribution.}
For $\varphi\in\{b,\sigma,f\}$, set
\[
 \Delta_t^{\varphi,\varepsilon}
 :=\int_A\left[
 \frac{\delta\varphi_t}{\delta m}
  (X_t^{\pi^\varepsilon},\mu_t^\varepsilon,
   \pi_t^\varepsilon,a)
 -\frac{\delta\varphi_t}{\delta m}
  (X_t^\pi,\mu_t,\pi_t^\varepsilon,a)
 \right]q_t(\d a).
\]
Assumption~\ref{ass:relative-smoothness-regularity} gives
\begin{equation}\label{eq:delta-control-derivative-state-law}
 |\Delta_s^{\varphi,\varepsilon}|^2
 \le K\left(
 |X_s^{\pi^\varepsilon}-X_s^\pi|^2
 +W_2(\mu_s^\varepsilon,\mu_s)^2
 \right)D_h(\pi'_s\mid\pi_s).
\end{equation}
Expanding the Hamiltonian derivative and adding and subtracting the terms
evaluated at $(X_s^\pi,\mu_s,\pi_s^\varepsilon)$ gives
$|I_1|\le |J_1|+J_2+J_3+J_4$, where
\begin{align*}
 J_1&:=\int_0^1\E\int_0^T\int_A
 \Bigg[
 \frac{\delta b_s}{\delta m}
  (X_s^\pi,\mu_s,\pi_s^\varepsilon,a)^{\trans}
  (Y_s^{\pi^\varepsilon}-Y_s^\pi)\\
 &\hspace{31mm}
 +\operatorname{tr}\!\left(
 \frac{\delta\sigma_s}{\delta m}
  (X_s^\pi,\mu_s,\pi_s^\varepsilon,a)^{\trans}
  (Z_s^{\pi^\varepsilon}-Z_s^\pi)
 \right)\Bigg]q_s(\d a)\,\d s\,\d\varepsilon,\\
 J_2&:=\int_0^1\E\int_0^T
 |Y_s^{\pi^\varepsilon}|\,|\Delta_s^{b,\varepsilon}|
 \,\d s\,\d\varepsilon,\\
 J_3&:=\int_0^1\E\int_0^T
 |Z_s^{\pi^\varepsilon}|\,|\Delta_s^{\sigma,\varepsilon}|
 \,\d s\,\d\varepsilon,\\
 J_4&:=\int_0^1\E\int_0^T
 |\Delta_s^{f,\varepsilon}|\,\d s\,\d\varepsilon.
\end{align*}

Set $\mathcal D:=\mathcal D_h(\pi'\mid\pi)$. By
Lemma~\ref{lem:YZlipc}, Assumption~\ref{ass:measderiv}, and the convexity of
$D_h(\cdot\mid\pi_s)$,
\begin{align}
 |J_1|
 &\le C\int_0^1
 \left(\E\int_0^T D_h(\pi_s^\varepsilon\mid\pi_s)\,\d s\right)^{1/2}
 \mathcal D^{1/2}\,\d\varepsilon
 \le C\mathcal D,
 \label{eq:J1-two-regimes}
\end{align}
\noindent where we used the fact that
$D_h(\pi_s^\varepsilon\mid\pi_s)
\le\varepsilon D_h(\pi'_s\mid\pi_s)$. Moreover, Lemma~\ref{lem:Xstab}
and the canonical coupling of $\mu_s^\varepsilon$ and $\mu_s$ give
\begin{equation}\label{eq:state-segment-bound-reg}
 \E\left[\sup_{0\leq s\leq T}
 |X_s^{\pi^\varepsilon}-X_s^\pi|^2\right]
 +\sup_{0\leq s\leq T}W_2(\mu_s^\varepsilon,\mu_s)^2
 \le C\varepsilon \mathcal D_h(\pi'\mid\pi).
\end{equation}
In particular, the Cauchy--Schwarz inequality,
\eqref{eq:delta-control-derivative-state-law}, and
\eqref{eq:state-segment-bound-reg} imply, in both regimes, that
\begin{equation}\label{eq:J4-two-regimes}
 J_4\le C\int_0^1\varepsilon^{1/2}\,\d\varepsilon\,\mathcal D_h(\pi'\mid\pi)
 \le C\mathcal D_h(\pi'\mid\pi).
\end{equation}

\smallskip
\noindent\emph{Regime (i).}
Lemma~\ref{lem:Ybound} gives uniform bounds for
$Y^{\pi^\varepsilon}$ in $H^\infty$ and for
$Z^{\pi^\varepsilon}\cdot W$ in BMO. For $J_2$, the Cauchy--Schwarz inequality,
\eqref{eq:delta-control-derivative-state-law}, and
\eqref{eq:state-segment-bound-reg} give
\begin{align*}
 J_2
 &\le C\int_0^1
 \left(\E\int_0^T
 \bigl(|X_s^{\pi^\varepsilon}-X_s^\pi|^2
 +W_2(\mu_s^\varepsilon,\mu_s)^2\bigr)\,\d s\right)^{1/2}
 \mathcal D^{1/2}\,\d\varepsilon\\
 &\le C\int_0^1\varepsilon^{1/2}\,\d\varepsilon\,\mathcal D_h(\pi'\mid\pi)
 \le C \mathcal D_h(\pi'\mid\pi).
\end{align*}
For $J_3$, Fefferman's inequality first gives
\begin{align*}
 J_3
 &\le C\int_0^1
 \|Z^{\pi^\varepsilon}\cdot W\|_{\mathrm{BMO}}
 \E\left[\left(\int_0^T
 |\Delta_s^{\sigma,\varepsilon}|^2\,\d s\right)^{1/2}\right]
 \d\varepsilon\\
 &\le C\int_0^1
 \left(\E\left[\sup_{0\le s\le T}
 \bigl(|X_s^{\pi^\varepsilon}-X_s^\pi|^2
 +W_2(\mu_s^\varepsilon,\mu_s)^2\bigr)\right]\right)^{1/2}
 \mathcal D_h(\pi'\mid\pi)^{1/2}\,\d\varepsilon,
\end{align*}
where the second inequality follows from
\eqref{eq:delta-control-derivative-state-law} and Cauchy--Schwarz. Hence,
by~\eqref{eq:state-segment-bound-reg},
\[
 J_3\le C\int_0^1\varepsilon^{1/2}\,\d\varepsilon\,D_h(\pi'\mid\pi)
 \le C\mathcal D_h(\pi'\mid\pi).
\]
Combining the previous estimates eventually gives $|I_1|\le C D_h(\pi'\mid\pi)$ in regime~(i).

\smallskip
\noindent\emph{Regime (ii).}
Assumption~\ref{ass:affine-dynamics} implies that, up to terms depending only
on $(t,m)$, the coefficients $b_t$ and $\sigma_t$ are affine in $x$ and in the
mean of $\mu$. Their flat derivatives with respect to $m$ are therefore
independent of $(x,\mu)$. Hence
\[
 \Delta_t^{b,\varepsilon}=\Delta_t^{\sigma,\varepsilon}=0,
 \qquad J_2=J_3=0.
\]
Estimates~\eqref{eq:J1-two-regimes} and~\eqref{eq:J4-two-regimes} again give
$|I_1|\le C \mathcal D_h(\pi'\mid\pi)$, now without using any $H^\infty$ or BMO estimate.

\medskip
\noindent\emph{Control-measure remainder.}
In $I_2$, the variables $(X_t^\pi,\mu_t,Y_t^\pi,Z_t^\pi)$ are fixed and
only the control-measure argument varies. The fundamental theorem of calculus
for flat derivatives and the definition of
$\mathcal R^\varphi$ therefore give the exact identity
\begin{align}
 I_2=\E\int_0^T\Big[&
 \mathcal R_t^b(X_t^\pi,\mu_t;\pi'_t\mid\pi_t)^{\trans}Y_t^\pi
 +\operatorname{tr}\!\left(
 \mathcal R_t^\sigma(X_t^\pi,\mu_t;\pi'_t\mid\pi_t)^{\trans}Z_t^\pi
 \right)\nonumber\\
 &+\mathcal R_t^f(X_t^\pi,\mu_t;\pi'_t\mid\pi_t)
 \Big]\,\d t.
 \label{eq:I2-finite-remainders}
\end{align}
By Assumption~\ref{ass:relative-smoothness-regularity},
$\mathcal R^\sigma=0$ in both regimes. In regime~(i), the uniform $H^\infty$-bound on
$Y^\pi$ and the bounds on $\mathcal R^b$ and $\mathcal R^f$ give
$|I_2|\le C \mathcal D_h(\pi'\mid\pi)$. In regime~(ii), the additional affinity of
$m\mapsto b_t(x,\mu,m)$ gives $\mathcal R^b=0$, while the bound on
$\mathcal R^f$ directly gives the same conclusion. The proof is now complete.
\end{proof}

We are now in position to prove Theorem \ref{thm:relsmooth}.

\begin{proof}
Fix $\pi, \pi' \in \mathcal{A}_{\mathscr{C}}$.
By Lemma \ref{lem:gateaux} and the fundamental theorem of calculus,
\[
  J^{\tau}(\pi')-J^{\tau}(\pi)= \int_0^1\E\!\int_0^T\!\int_A\frac{\delta H^0_t}{\delta m}\bigl(\Theta_t^{\pi^\eps},\pi^\eps_t,a\bigr)(\pi'_t-\pi_t)(\d a)\,\d t\d \eps +\tau \, \E \int_0^T [h(\pi'_t)-h(\pi_t)] \d t
\]
By the definition of $D_h$ 
\[
  \E \int_0^T [h(\pi'_t)-h(\pi_t)] \d t = \E\int_0^T\Bigl(D_h(\pi'_t|\pi_t) + \int_A \frac{\delta h}{\delta m}(\pi_t,a)(\pi'_t - \pi_t)(\d a)\Bigr)\d t
\]

\noindent It follows by Fubini's theorem and Lemma \ref{lem:reg} that:
\begin{align*}
    &J^{\tau}(\pi')-J^{\tau}(\pi)-\E\!\int_0^T\!\int_A\frac{\delta H^{\tau}_t}{\delta m}\bigl(\Theta_t^{\pi},\pi_t,a\bigr)(\pi'_t-\pi_t)(\d a)\,\d t\\
    &= \int_0^1 \E\int_0^T \int_A\Bigl(\frac{\delta H^0_t}{\delta m}(\Theta_t^{\pi^\eps},\pi^\eps_t,a)-\frac{\delta H_t^0}{\delta m}(\Theta_t^{\pi},\pi_t,a)\Bigr)(\pi'_t-\pi_t)(\d a)\, \d t\, \d \eps + \tau \, \E\int_0^T D_h(\pi'_t|\pi_t) \d t\\
    &\le L\E \int_0^T D_h(\pi'_t|\pi_t)\d t
\end{align*}

\end{proof}

\subsection{Proof of Theorem \ref{thm:dissipation}}\label{sec:proof-dissipation}

Applying Theorem~\ref{thm:relsmooth} for
$\pi'=\pi^{n+1}$ and $\pi=\pi^n$ yields
\[
  J^\tau(\pi^{n+1})-J^\tau(\pi^n)\le\E\!\int_0^T\!\int_A\frac{\delta H^\tau_t}{\delta m}(\Theta_t^{\pi^n},\pi^n_t, a)(\pi^{n+1}_t-\pi^n_t)(\d a)\,\d t
   +L\,\E\!\int_0^T\!D_h(\pi^{n+1}_t\mid\pi^n_t)\,\d t.
\]
On the other hand, by the optimality of $\pi^{n+1}$ in the mirror descent update~\eqref{eq:MDupdate}, for every $s\in[0,T]$,
\begin{align*}
  \int_A\frac{\delta H^\tau_t}{\delta m}&(\Theta_t^{\pi^n},\pi^n_t, a)(\pi^{n+1}_t-\pi^n_t)(\d a)+\lambda D_h(\pi^{n+1}_t\mid\pi^n_t)\\
  &\le\int_A\frac{\delta H^\tau_t}{\delta m}(\Theta_t^{\pi^n},\pi^n_t, a)(\pi^n_t-\pi^n_t)(\d a)+\lambda D_h(\pi^n_t\mid\pi^n_t)=0.
\end{align*}

Combining the two inequalities gives
\begin{equation}\label{eq:dissip}
  J^\tau(\pi^{n+1})-J^\tau(\pi^n)\le -(\lambda-L)\,\E\!\int_0^T\!D_h(\pi^{n+1}_t\mid\pi^n_t)\,\d t,
\end{equation}

\noindent which proves the first statement. Since $D_h\geq0$ and $\lambda\geq L$, this also proves the monotonicity of the sequence $\bigl(J^\tau(\pi^n)\bigr)_{n\geq0}$.

If $\lambda>L$ and $\inf_{\pi\in\AC}J^\tau(\pi)>-\infty$, then summing~\eqref{eq:dissip} from $n=0$ to $m-1$ gives
$$
(\lambda-L)
\sum_{n=0}^{m-1}
\E\Big[\int_0^T
D_h(\pi^{n+1}_s\mid\pi^n_s) \, \d s \Big]
\leq
J^\tau(\pi^0)-J^\tau(\pi^m)
\leq
J^\tau(\pi^0)-\inf{\pi\in\AC}J^\tau(\pi).
$$
Hence
$$
\sum_{n=0}^{\infty}
\E\Big[\int_0^T
D_h(\pi^{n+1}_s\mid\pi^n_s) \, \d s \Big]
<\infty,
$$
\noindent which yields the desired conclusion.

\subsection{Proof of Theorem \ref{thm:relconvex}}

Fix $\pi,\pi'\in\mathcal A_{\mathscr C}$ and set
\[
\delta X_t:=X_t^{\pi'}-X_t^\pi,
\qquad
\delta b_t:=b_t(X_t^{\pi'},\mu_t^{\pi'},\pi_t')
-b_t(X_t^\pi,\mu_t^\pi,\pi_t),
\]
with the analogous notation $\delta\sigma_t$. Let
\[
\mathcal G_t^\pi
:=(D_xH_t^0)(\Theta_t^\pi,\pi_t)
+\widetilde\E\big[
(\partial_\mu H_t^0)
(\widetilde\Theta_t^\pi,\widetilde\pi_t)(X_t^\pi)
\big]
\]
denote the generator of the adjoint equation. By the definition of
$\mathfrak R_g$ and an exchange of the original and copy variables in the
Lions-derivative term,
\begin{equation}\label{eq:terminal-performance-remainder}
\E\big[
g(X_T^{\pi'},\mu_T^{\pi'})-g(X_T^\pi,\mu_T^\pi)
\big]
=\E\big[(Y_T^\pi)^{\trans}\delta X_T\big]
+\mathfrak R_g(X_T^{\pi'}\mid X_T^\pi).
\end{equation}

Apply It\^o's formula to
$t\mapsto(Y_t^\pi)^{\trans}\delta X_t$. The local martingale terms are true
martingales under either of the two regularity regimes in the statement, and
therefore vanish after taking expectations. Since $\delta X_0=0$, we obtain
\begin{align}
\E\big[(Y_T^\pi)^{\trans}\delta X_T\big]
={}&\E\int_0^T\Big[
(Y_t^\pi)^{\trans}\delta b_t
+\operatorname{tr}(\delta\sigma_t^{\trans}Z_t^\pi)
-\langle\mathcal G_t^\pi,\delta X_t\rangle
\Big]\,\d t.
\label{eq:Ito-performance-difference}
\end{align}
Combining~\eqref{eq:terminal-performance-remainder} and
\eqref{eq:Ito-performance-difference} with the definition of $J^0$, and
exchanging the original and copy variables once more in the law-derivative
term, gives the exact performance difference identity
\begin{align}
&J^0(\pi')-J^0(\pi)
-\E\int_0^T\int_A
\frac{\delta H_t^0}{\delta m}
(\Theta_t^\pi,\pi_t,a)
(\pi_t'-\pi_t)(\d a)\,\d t
\nonumber\\
&\qquad=
\mathfrak R_g(X_T^{\pi'}\mid X_T^\pi)
+\int_0^T\mathfrak R_{H,t}^\pi(\pi'\mid\pi)\,\d t.
\label{eq:performance-difference-remainder}
\end{align}
Assumption~\ref{ass:relative-semiconvexity} therefore yields
\begin{align*}
J^0(\pi')-J^0(\pi)
\geq{}&
\E\int_0^T\int_A
\frac{\delta H_t^0}{\delta m}
(\Theta_t^\pi,\pi_t,a)
(\pi_t'-\pi_t)(\d a)\,\d t
-\rho\mathcal D_h(\pi'\mid\pi).
\end{align*}
Finally, the Bregman identity gives
\begin{align*}
\E\int_0^T\bigl(h(\pi_t')-h(\pi_t)\bigr)\,\d t
={}&\E\int_0^T\int_A
\frac{\delta h}{\delta m}(\pi_t,a)
(\pi_t'-\pi_t)(\d a)\,\d t
+\mathcal D_h(\pi'\mid\pi).
\end{align*}
Multiplying by $\tau$, adding the two inequalities, and using the definition
of $H^\tau$ proves~\eqref{eq:relative-semiconvexity-Jtau}.
\subsection{Proof of Theorem \ref{thm:main}}\label{sec:proof-main}

The following standard three-point inequality will be useful in the proof of
Theorem~\ref{thm:main}; see~\cite[Lemma~3]{NEURIPS2022_6e3daaec} or
\cite[Lemma~A.2]{KERIMKULOV2025104765}.

\begin{Lemma}[Three-point inequality]\label{lem:3pts}
Let $\nu\in C$ and let $G\colon C\to\R$ be convex with a flat derivative $\frac{\delta G}{\delta m}$.
If $\bar m\in\arg\min_{m\in C}\big(G(m)+D_h(m\mid\nu)\big)$, then for all $m'\in C$,
\[
  G(m')+D_h(m'\mid\nu)\ \ge\ G(\bar m)+D_h(m'\mid\bar m)+D_h(\bar m\mid\nu).
\]
\end{Lemma}

\noindent \emph{\textbf{Proof of Theorem \ref{thm:main}.}}

Fix $\pi'\in\mathcal{A}_{\mathscr C}$ such that
$\mathcal{D}_h(\pi'\mid\pi^n)<\infty$ for every $n\geq0$. We first derive a recursive estimate that will be used in both cases. Applying Theorem~\ref{thm:relsmooth} with $\pi=\pi^n$ and $\pi'=\pi^{n+1}$ gives
\begin{align*}
J^\tau(\pi^{n+1})
&\leq J^\tau(\pi^n)
+\E\int_0^T\int_A
\frac{\delta H^\tau_s}{\delta m}
\bigl(\Theta_s(\pi^n),\pi^n_s,a\bigr)
(\pi^{n+1}_s-\pi^n_s)(\d a),\d s \, +L\mathcal{D}_h(\pi^{n+1}\mid\pi^n).
\end{align*}

On the other hand, the definition of the mirror descent update and Lemma~\ref{lem:3pts}, applied pointwise in $(\omega,s)$ and then integrated, yield
\begin{align*}
&\E\int_0^T\int_A
\frac{\delta H^\tau_s}{\delta m}
\bigl(\Theta_s(\pi^n),\pi^n_s,a\bigr)
(\pi^{n+1}_s-\pi^n_s)(\d a)\,\d s
+\lambda\mathcal{D}_h(\pi^{n+1}\mid\pi^n)
\\
&\qquad\leq
\E\int_0^T\int_A
\frac{\delta H^\tau_s}{\delta m}
\bigl(\Theta_s(\pi^n),\pi^n_s,a\bigr)
(\pi'_s-\pi^n_s)(\d a)\,\d s
+\lambda\mathcal{D}_h(\pi'\mid\pi^n)
-\lambda\mathcal{D}_h(\pi'\mid\pi^{n+1}).
\end{align*}

Combining the two inequalities and using $L\leq\lambda$, we obtain
\begin{align*}
J^\tau(\pi^{n+1})
&\leq J^\tau(\pi^n)
+\E\int_0^T\int_A
\frac{\delta H^\tau_s}{\delta m}
\bigl(\Theta_s(\pi^n),\pi^n_s,a\bigr)
(\pi'_s-\pi^n_s)(\d a)\,\d s\\
&\qquad +\lambda\mathcal{D}_h(\pi'\mid\pi^n)
-\lambda\mathcal{D}_h(\pi'\mid\pi^{n+1}).
\end{align*}

Set $\alpha:=\tau-\rho$. By the relative convexity of $J^\tau$ given in
Theorem~\ref{thm:relconvex},
\[
J^\tau(\pi^n) + \E\!\int_0^T\!\!\int_A\frac{\delta H^\tau_s}{\delta m}(\Theta_s(\pi^n),\pi^n_s,a)(\pi'_s-\pi^n_s)(\d a)\,\d s
\;\le\; J^\tau(\pi') - \alpha \mathcal{D}_h(\pi'\mid\pi^{n}).
\]
Consequently,
\begin{equation}\label{eq:keyRec}
J^\tau(\pi^{n+1}) - J^\tau(\pi')
\;\le\; \lambda\!\left(\mathcal{D}_h(\pi'\mid\pi^n) - \mathcal{D}_h(\pi'\mid\pi^{n+1})\right)
   - \alpha\, \mathcal{D}_h(\pi'\mid\pi^n).
\end{equation}

\noindent\textbf{Case $\alpha=0$.}
Taking $\pi'=\pi$ in the preceding inequality gives
\[
J^\tau(\pi^{n+1}) - J^\tau(\pi)
\;\le\; \lambda\!\left(\mathcal{D}_h(\pi \mid \pi^n) - \mathcal{D}_h(\pi \mid \pi^{n+1}) \right).
\]
Summing from $n=0$ to $m-1$ gives
\[
\sum_{n=0}^{m-1}\bigl(J^\tau(\pi^{n+1})-J^\tau(\pi)\bigr) \;\le\; \lambda\Bigl( \mathcal{D}_h(\pi \mid\pi^0) - \mathcal{D}_h(\pi \mid \pi^m ) \Bigr).
\]
By Theorem~\ref{thm:dissipation}, the sequence $\{J^\tau(\pi^n)\}$ is
non-increasing. Hence,
\[
m\bigl(J^\tau(\pi^m)-J^\tau(\pi)\bigr) \;\le\; \lambda\,  \mathcal{D}_h(\pi \mid \pi^0).
\]
This proves the asserted $\mathcal{O}(1/m)$ convergence rate.

\medskip
\noindent\textbf{Case $\alpha>0$.}

Let $\pi'=\pi^* \in \arg\min J^\tau$. Since $J^\tau(\pi^{n+1})\ge J^\tau(\pi^*)$, inequality~\eqref{eq:keyRec} implies
\[
0
\leq
(\lambda-\alpha)\mathcal{D}_h(\pi^\star \mid \pi^n)
-\lambda\mathcal{D}_h(\pi^\star \mid \pi^{n+1}).
\]
Thus
\[
\mathcal{D}_h(\pi^* \mid\pi^{n+1}) \;\le\; \Bigl(1-\frac{\alpha}{\lambda}\Bigr)\, \mathcal{D}_h(\pi^*\mid \pi^n).
\]
Iterating this estimate gives
\begin{equation}\label{eq:expoBreg}
\mathcal{D}_h(\pi^* \mid \pi^m )  \;\le\; \Bigl(1-\frac{\alpha}{\lambda}\Bigr)^{m}\,  \mathcal{D}_h(\pi^* \mid \pi^0), \quad m\geq0.
\end{equation}

In particular, the Bregman divergence from the iterates to the optimal control converges to zero at a geometric rate.
Coming back to~\eqref{eq:keyRec},
\[
J^\tau(\pi^{n+1})-J^\tau(\pi^*) \;\le\; (\lambda-\alpha) \mathcal{D}_h(\pi^* \mid \pi^n).
\]
which combined with~\eqref{eq:expoBreg} yields
\[
J^\tau(\pi^{n+1})-J^\tau(\pi^*) \;\le\; \lambda\Bigl(1-\frac{\alpha}{\lambda}\Bigr)^{n+1} \mathcal{D}_h(\pi^*\mid\pi^0).
\]

Together with the optimality of $\pi^*$, this gives the asserted estimate.

\subsection{Proofs for the examples of regularisers}
\label{sec:proofs-regularisers}

\begin{proof}[Proof of Proposition~\ref{prop:KL}]
\leavevmode\par

\noindent \emph{Proof of regime (i).}

The proof is the same as that of
\cite[Proposition~2.17 and Section~3.6]{KERIMKULOV2025104765}. Indeed, the
entropy mirror step has, up to normalisation, the density
\[
\frac{\d\pi_t^{n+1}}{\d\varrho}(a)
\propto
\left(\frac{\d\pi_t^n}{\d\varrho}(a)\right)^{1-\tau/\lambda}
\exp\left(-\frac{\mathscr G_t^{\pi^n}(a)}{\lambda}\right).
\]
The only point to check in the present MKV setting is the
integrability of its logarithm. The assumed bounds on the flat derivatives directly give
\[
\|\mathscr G_t^{\pi^n}\|_{L^\infty(A)}
\leq C\bigl(1+|Y_t^{\pi^n}|+|Z_t^{\pi^n}|\bigr),
\]
and Lemma~\ref{lem:Ybound} makes the right-hand side integrable on
$\Omega\times[0,T]$. The induction argument of the cited proof therefore
applies without further change.

\medskip

\noindent\emph{Proof of regime (ii).}

Set $q:=1-\tau/\lambda\in(0,1)$. Suppose
inductively that $\pi^n\in\mathcal A_{\mathscr C}$ and write
$\pi_t^n=f_t^n\varrho$. The first-order condition for the entropy mirror
step gives
\begin{equation}\label{eq:quadratic-entropy-update}
f_t^{n+1}(a)
=\frac{1}{C_t^n}(f_t^n(a))^q
\exp\left(-\frac{\mathscr G_t^{\pi^n}(a)}{\lambda}\right),
\end{equation}
where $C_t^n$ is the normalising constant. The lower bound in
\eqref{eq:quadratic-action-gradient} and Jensen's inequality give
\[
0<C_t^n
\leq e^{\zeta_t^{\pi^n}/\lambda}
\int_{\mathbb R^k}(f_t^n)^q\,\d\varrho
\leq e^{\zeta_t^{\pi^n}/\lambda}.
\]
Thus~\eqref{eq:quadratic-entropy-update} defines a strictly positive
probability density. Moreover, using $x^q\leq1+x$ and the coercive factor
generated by~\eqref{eq:quadratic-action-gradient}, we obtain
\[
\int_{\mathbb R^k}|a|^2f_t^{n+1}(a)\,\varrho(\d a)<\infty.
\]

The logarithmic recursion is
\begin{equation}\label{eq:quadratic-log-density-update}
\log f_t^{n+1}
=q\log f_t^n-\lambda^{-1}\mathscr G_t^{\pi^n}-\log C_t^n.
\end{equation}
The elementary estimate
$x^q|\log x|\leq C_q(1+x)$, $x>0$, shows that
\[
\int_{\mathbb R^k}|\log f_t^n(a)|\,\pi_t^{n+1}(\d a)<\infty
\]
for almost every $(t,\omega)$. This is the precise cross-logarithmic
integrability needed in~\eqref{eq:quadratic-log-density-update}.

To obtain integrability on $\Omega\times[0,T]$, we do not integrate the
preceding pointwise bound, which would contain $(C_t^n)^{-1}$. Instead, Jensen's
inequality and the two sided-bounds in~\eqref{eq:quadratic-action-gradient} give
\begin{equation}\label{eq:normalisation-constant}
-\frac{\tau}{\lambda}h(\pi_t^n)
-\frac1\lambda\int_{\mathbb R^k}\mathscr G_t^{\pi^n}\,\d\pi_t^n
\leq\log C_t^n\leq\frac{\zeta_t^{\pi^n}}{\lambda}.
\end{equation}
Set
\[
H_t^j:=D_{\rm KL}(\pi_t^j\mid\varrho)
=\int_{\mathbb R^k}f_t^j\log f_t^j\,\d\varrho.
\]
Using $\pi_t^n$ as a competitor in the Gibbs variational identity gives
\begin{align*}
&\int_{\mathbb R^k}\mathscr G_t^{\pi^n}\,
\d(\pi_t^{n+1}-\pi_t^n)
+\tau\left(\int_{\mathbb R^k}\log f_t^n\,\d\pi_t^{n+1}-H_t^n\right)
\\
&\hspace{45mm}
+\lambda D_{\rm KL}(\pi_t^{n+1}\mid\pi_t^n)\leq0.
\end{align*}
Since
\[
D_{\rm KL}(\pi_t^{n+1}\mid\pi_t^n)
=H_t^{n+1}-\int_{\mathbb R^k}\log f_t^n\,\d\pi_t^{n+1}
\]
and the entropy variational inequality gives
$\int\log f_t^n\,\d\pi_t^{n+1}\leq H_t^{n+1}$, and recalling that $\lambda>\tau$, it follows that
\[
\int\mathscr G_t^{\pi^n}\,\d\pi_t^{n+1}+\tau H_t^{n+1}
\leq
\int\mathscr G_t^{\pi^n}\,\d\pi_t^n+\tau H_t^n.
\]
Consequently, the preceding inequality together with the two sided-bounds~\eqref{eq:quadratic-action-gradient} gives
\begin{equation}\label{eq:qent-moment-entropy-step}
\kappa^{-1}M_{2,A}(\pi_t^{n+1})^2+\tau H_t^{n+1}
\leq \kappa M_{2,A}(\pi_t^n)^2+\tau H_t^n+2\zeta_t^{\pi^n}.
\end{equation}
To control the relative entropy between successive iterates, integrate
\eqref{eq:quadratic-log-density-update} against $\pi_t^{n+1}$ and eliminate
$\int\log f_t^n\,\d\pi_t^{n+1}$. This gives
\[
D_{\rm KL}(\pi_t^{n+1}\mid\pi_t^n)
=-\frac{1-q}{q}H_t^{n+1}
-\frac{1}{q\lambda}\int\mathscr G_t^{\pi^n}\,\d\pi_t^{n+1}
-\frac1q\log C_t^n.
\]
Since $H_t^{n+1}\geq0$, the lower bounds \eqref{eq:quadratic-action-gradient} on
$\mathscr G_t^{\pi^n}$ and \eqref{eq:normalisation-constant} on $\log C_t^n$, followed by the upper bound on
$\int\mathscr G_t^{\pi^n}\,\d\pi_t^n$, yield
\begin{equation}\label{eq:qent-successive-kl-step}
D_{\rm KL}(\pi_t^{n+1}\mid\pi_t^n)
\leq \frac{\tau}{q\lambda}H_t^n
+\frac{\kappa}{q\lambda}M_{2,A}(\pi_t^n)^2
+\frac{2}{q\lambda}\zeta_t^{\pi^n}.
\end{equation}
Finally, since $\varrho$ is a probability measure,
$\int f|\log f|\,\d\varrho
\leq D_{\rm KL}(f\varrho\mid\varrho)+2/e$.
Combining~\eqref{eq:qent-moment-entropy-step} and
\eqref{eq:qent-successive-kl-step}, and integrating in $(t,\omega)$, gives
\begin{align*}
&\E\int_0^T\left[
M_{2,A}(\pi_t^{n+1})^2
+\int_{\mathbb R^k}f_t^{n+1}|\log f_t^{n+1}|\,\d\varrho
+D_{\rm KL}(\pi_t^{n+1}\mid\pi_t^n)
\right]\d t
\\
&\qquad\leq C_{q,\kappa,\lambda}\E\int_0^T\left[
1+M_{2,A}(\pi_t^n)^2
+\int_{\mathbb R^k}f_t^n|\log f_t^n|\,\d\varrho
+\zeta_t^{\pi^n}\right]\d t<\infty.
\end{align*}

Finally,
\[
\log\frac{f_t^{n+1}}{f_t^n}
=-\frac{\tau}{\lambda}\log f_t^n
-\frac1\lambda\mathscr G_t^{\pi^n}-\log C_t^n.
\]
The preceding estimates show that the right-hand side is integrable with
respect to $\pi_t^{n+1}(\d a)\,\d t\,\d\mathbb P$. Consequently,
\[
\mathcal D_h(\pi^{n+1}\mid\pi^n)
=\E\int_0^T
D_{\rm KL}(\pi_t^{n+1}\mid\pi_t^n)\,\d t<\infty,
\]
and $\pi^{n+1}\in\mathcal A_{\mathscr C}$. This closes the induction and
proves Assumption~\ref{ass:iterates} in regime (ii).

\medskip

\noindent \emph{Proof of regime (iii).}

The action-dependent part of the unregularised
Hamiltonian is
\[
a^{\trans}N_ta+a^{\trans}\ell_t^n.
\]
Hence the same first-order condition as in regime (ii), but now relative to the Lebesgue measure, reads
\begin{equation}\label{eq:gaussian-density-update-proof}
p_t^{n+1}(a)
\propto
\bigl(p_t^n(a)\bigr)^{1-\tau/\lambda}
\exp\left[-\frac1\lambda
\bigl(a^{\trans}N_ta+a^{\trans}\ell_t^n\bigr)\right].
\end{equation}
Set again $q:=1-\tau/\lambda\in[0,1]$. If
$\pi_t^n=\mathcal N(u_t^n,\Gamma_t^n)$, its density can be written, up to a
normalising constant, as
\[
p_t^n(a)\propto
\exp\left(-\frac12a^{\trans}R_t^na+(\eta_t^n)^{\trans}a\right).
\]
Substitution into~\eqref{eq:gaussian-density-update-proof} and completion of
the square give
\[
p_t^{n+1}(a)\propto
\exp\left(-\frac12a^{\trans}
\left(qR_t^n+\frac2\lambda N_t\right)a
+\left(q\eta_t^n-\frac1\lambda\ell_t^n\right)^{\trans}a\right),
\]
which proves~\eqref{eq:abstract-gaussian-update}. In particular,
$R_t^{n+1}\succeq(2\nu/\lambda)I_k$, so the new density is normalisable and
$\pi_t^{n+1}$ is nondegenerate.

We next verify that the recursion remains in the admissible class. Let
$\overline N:=\sup_{t\leq T}\|N_t\|$. Iterating the precision recursion
gives
\begin{equation}\label{eq:precision-iteration-bound}
R_t^n=q^nR_t^0+\frac2\lambda\sum_{j=0}^{n-1}q^jN_t.
\end{equation}
If $\tau>0$, this identity yields the uniform spectral bounds
\[
\min\left\{\overline\gamma^{-1},\frac{2\nu}{\tau}\right\}I_k
\preceq R_t^n
\preceq
\max\left\{\underline\gamma^{-1},\frac{2\overline N}{\tau}\right\}I_k.
\]
If $\tau=0$, then
$\overline\gamma^{-1}I_k\preceq R_t^n
\preceq(\underline\gamma^{-1}+2n\overline N/\lambda)I_k$.
Thus, for every $n$, both $R^n$ and $\Gamma^n=(R^n)^{-1}$ are bounded in
$(t,\omega)$; the bounds are uniform in $n$ when $\tau>0$.

Assume inductively that $\pi^n$ is admissible. Proposition~\ref{prop:wellposed}
then gives $X^{\pi^n}\in S^2$ and
$(Y^{\pi^n},Z^{\pi^n})\in S^2\times H^2$. Since all LQ coefficients are
bounded,
\begin{equation}\label{eq:lq-linear-energy-integrability}
\E\int_0^T|\ell_t^n|^2\,\d t<\infty.
\end{equation}
Moreover, $\eta^0=R^0u^0$ belongs to
$L^2(\Omega\times[0,T])$. The second recursion in
\eqref{eq:abstract-gaussian-update}, followed by
\eqref{eq:lq-linear-energy-integrability}, therefore implies inductively
that $\eta^n$, and hence $u^n=\Gamma^n\eta^n$, belongs to this space. It
follows that
\[
\E\int_0^T\int_{\mathbb R^k}|a|^2\pi_t^n(\d a)\,\d t
=\E\int_0^T\bigl(|u_t^n|^2+\operatorname{tr}(\Gamma_t^n)\bigr)\,\d t
<\infty,
\]
so $\pi^n$ is admissible and the induction closes.

Finally, the Bregman divergence of negative entropy is relative entropy. For
successive Gaussian iterates,
\begin{align*}
D_h(\pi_t^{n+1}\mid\pi_t^n)
=\frac12\Big(&\operatorname{tr}(R_t^n\Gamma_t^{n+1})
+(u_t^n-u_t^{n+1})^{\trans}R_t^n(u_t^n-u_t^{n+1})
\\
&-k+\log\frac{\det\Gamma_t^n}{\det\Gamma_t^{n+1}}\Big).
\end{align*}
The spectral bounds above and the square integrability of the means show
that its expectation is integrable over $[0,T]$. Consequently,
$\mathcal D_h(\pi^{n+1}\mid\pi^n)<\infty$ for every $n$, which proves
Assumption~\ref{ass:iterates}.
\end{proof}

\begin{proof}[Proof of Proposition~\ref{prop:chi2}]
Write $q_t^n=\d\pi_t^n/\d\varrho$ and let
\[
\mathcal K:=\left\{q\in L^2(\varrho):q\geq0,\
\int_Aq\,\d\varrho=1\right\}.
\]
This is a nonempty closed convex subset of $L^2(\varrho)$. A normalised flat
derivative of $h$ is
\[
 \frac{\delta h}{\delta m}(m,a)
 =q(a)-\int_Aq(a')^2\,\varrho(\d a'),
 \qquad q=\frac{\d m}{\d\varrho}.
\]
Thus, up to additive constants in $a$,
$\delta h/\delta m(m,\cdot)=q$ and
$D_h(m'\mid m)=\frac12\|q'-q\|_{L^2_\varrho}^2$. The standard Hilbert-space
calculation in \cite[Proposition~2.20]{KERIMKULOV2025104765} shows that the
optimality condition for the mirror step is
\begin{equation}\label{eq:chi2-variational-inequality}
 \left\langle
 \frac{\delta H_t^\tau}{\delta m}
 (\Theta_t^{\pi^n},\pi_t^n,\cdot)
 +\lambda(q_t^{n+1}-q_t^n),q-q_t^{n+1}
 \right\rangle_{L^2_\varrho}\geq0,
 \qquad q\in\mathcal K,
\end{equation}
or, equivalently,
\begin{equation}\label{eq:chi2-projection}
q_t^{n+1}
=\operatorname{Proj}_{\mathcal K}\left(
q_t^n-\lambda^{-1}
\frac{\delta H_t^\tau}{\delta m}
(\Theta_t^{\pi^n},\pi_t^n,\cdot)
\right).
\end{equation}
The projection is $1$-Lipschitz. It therefore preserves progressive
measurability and, since $\operatorname{Proj}_{\mathcal K}(q_t^n)=q_t^n$,
\begin{equation}\label{eq:chi2-projection-estimate}
 \|q_t^{n+1}-q_t^n\|_{L^2_\varrho}
 \leq\lambda^{-1}\left\|
 \frac{\delta H_t^\tau}{\delta m}
 (\Theta_t^{\pi^n},\pi_t^n,\cdot)
 \right\|_{L^2_\varrho},
 \qquad
 \|q_t^{n+1}\|_{L^2_\varrho}
 \leq\|q_t^n\|_{L^2_\varrho}
 +\lambda^{-1}\left\|\frac{\delta H_t^\tau}{\delta m}
 (\Theta_t^{\pi^n},\pi_t^n,\cdot)\right\|_{L^2_\varrho}.
\end{equation}

\noindent\emph{Bounded-derivative regime.}
Apart from evaluating the stochastic action gradient along the
MKV FBSDE, this is precisely the setting treated in
\cite[Proposition~2.20]{KERIMKULOV2025104765}. The assumed
$L^2_\varrho$ envelopes and Proposition~\ref{prop:wellposed} give
\[
 \left\|\frac{\delta H_t^\tau}{\delta m}
 (\Theta_t^{\pi^n},\pi_t^n,\cdot)\right\|_{L^2_\varrho}
 \leq C(1+|Y_t^{\pi^n}|+|Z_t^{\pi^n}|)
 +\tau\|q_t^n\|_{L^2_\varrho}.
\]
Consequently, \eqref{eq:chi2-projection-estimate}, integrated in
$(t,\omega)$, closes the induction. We refer to the cited proof for the
remaining standard projection argument.

\medskip

\noindent\emph{Quadratic-growth regime.}
Since additive constants in $a$ are immaterial, we may use the
representative
\[
 \frac{\delta H_t^\tau}{\delta m}
 (\Theta_t^{\pi^n},\pi_t^n,a)
 =\mathscr G_t^{\pi^n}(a)+\tau q_t^n(a).
\]
Writing $M_j(\varrho)=(\int_{\mathbb R^k}|a|^j\varrho(\d a))^{1/j}$,
\eqref{eq:chi2-quadratic-action-gradient} gives
\begin{equation}\label{eq:chi2-quadratic-gradient-L2}
 \left\|\frac{\delta H_t^\tau}{\delta m}
 (\Theta_t^{\pi^n},\pi_t^n,\cdot)\right\|_{L^2_\varrho}
 \leq \tau\|q_t^n\|_{L^2_\varrho}
 +\kappa M_4(\varrho)^2
 +\zeta_t^{\pi^n}\bigl(1+M_2(\varrho)\bigr).
\end{equation}
Thus \eqref{eq:chi2-projection-estimate},
\eqref{eq:chi2-quadratic-gradient-L2}, and the assumed square
integrability of $\zeta^{\pi^n}$ imply
$\E\int_0^T\|q_t^{n+1}\|_{L^2_\varrho}^2\,\d t<\infty$. Moreover,
\[
 \int_{\mathbb R^k}|a|^2\pi_t^{n+1}(\d a)
 \leq
 \left(\int_{\mathbb R^k}|a|^4\varrho(\d a)\right)^{1/2}
 \|q_t^{n+1}\|_{L^2_\varrho},
\]
while $h(\pi_t^{n+1})=\frac12\|q_t^{n+1}-1\|_{L^2_\varrho}^2$ and
\[
 D_h(\pi_t^{n+1}\mid\pi_t^n)
 =\tfrac12\|q_t^{n+1}-q_t^n\|_{L^2_\varrho}^2.
\]
After integration in $(t,\omega)$ these quantities are finite. Induction
therefore proves both admissibility of every iterate and
Assumption~\ref{ass:iterates}. Finally, in the affine-LQ model the
unregularised action gradient is of the form
$a^{\trans}N_ta+a^{\trans}\ell_t^\pi$. Boundedness of $N$ and
Proposition~\ref{prop:wellposed}, which gives
$\E\int_0^T|\ell_t^\pi|^2\,\d t<\infty$, verify
\eqref{eq:chi2-quadratic-action-gradient}. Hence that model is an immediate
special case of regime~(ii).
\end{proof}

\begin{proof}[Proof of Proposition~\ref{prop:eot}]
By \cite[Proposition~2.22 and its proof]{KERIMKULOV2025104765}, $h$ is
finite and continuous on $\mathcal P(A)$, its flat derivative is given by
\eqref{eq:eot-flat-derivative}, and $m\mapsto\phi[m]$ is continuous from
$\mathcal P(A)$, endowed with weak convergence, into $C(A)$. Since
$\mathcal P(A)$ is compact,
\[
\sup_{m\in\mathcal P(A)}|h(m)|<\infty,
\qquad
\sup_{m\in\mathcal P(A)}\|\phi[m]\|_{C(A)}<\infty.
\]

We only detail the points added by the mean-field and linear-growth setting.
For $m,m'\in\mathcal P(A)$, set
\begin{align*}
\mathcal F(t,x,\mu,y,z,m;m')
:={}&\int_A\left[
\frac{\delta H_t^0}{\delta m}(x,\mu,y,z,m,a)
+(\tau-\lambda)\phi[m](a)
\right]m'(\d a)+\lambda h(m').
\end{align*}
Terms independent of $m'$ have been omitted. The preceding continuity
properties, together with the standing joint-measurability assumptions on
the flat derivatives, show that $\mathcal F$ is measurable in
$(t,x,\mu,y,z,m)$ and continuous in $m'$. The measurable maximum theorem on
the compact Polish space $\mathcal P(A)$ therefore provides a measurable
selector of its minimisers. Evaluated at
$(X_t^{\pi^n},\mu_t^{\pi^n},Y_t^{\pi^n},Z_t^{\pi^n},\pi_t^n)$, this selector
defines a progressively measurable solution $\pi^{n+1}$ of the mirror step.

It remains to verify integrability, which is the only substantive change
from \cite[Proposition~2.22]{KERIMKULOV2025104765}. Compactness of $A$, the
growth assumption in the statement, and the definition of the stochastic
action gradient yield
\[
 \|\mathscr G_t^{\pi^n}\|_{C(A)}
 \leq C\bigl(1+|X_t^{\pi^n}|+M_2(\mu_t^{\pi^n})\bigr)
 \bigl(1+|Y_t^{\pi^n}|+|Z_t^{\pi^n}|\bigr).
\]
Proposition~\ref{prop:wellposed} and Cauchy--Schwarz consequently give
\[
\E\int_0^T\|\mathscr G_t^{\pi^n}\|_{C(A)}\,\d t<\infty.
\]
Finally, the uniform bounds above imply, for all $m,m'\in\mathcal P(A)$,
\[
 |h(m)|+D_h(m'\mid m)
 \leq C\left(1+\sup_{\nu\in\mathcal P(A)}\|\phi[\nu]\|_{C(A)}\right).
\]
Thus $\E\int_0^T|h(\pi_t^{n+1})|\,\d t<\infty$ and
$\mathcal D_h(\pi^{n+1}\mid\pi^n)<\infty$. Hence
$\pi^{n+1}\in\mathcal A_{\mathscr C}$, and induction proves
Assumption~\ref{ass:iterates}.
\end{proof}

\begin{proof}[Proof of Corollary~\ref{cor:gaussian-lq-smd}]
The precision recursion in Proposition~\ref{prop:KL}(iii), together with
$N_t\succeq\nu I_k$ and the uniform positive definiteness of $R_t^0$, yields
a constant $c>0$, independent of $n$ and $t$, such that
$R_t^n\succeq cI_k$. Hence
\[
\Gamma_t^n=(R_t^n)^{-1}\preceq \overline\Gamma I_k,
\qquad n\geq0,
\]
for some $\overline\Gamma<\infty$. Notice that no upper bound on the
covariance of the comparator is needed here.

For two Gaussian controls $\pi$ and $\pi'$ with respective means $u,u'$ and
with the covariance of $\pi$ bounded by $\overline\Gamma I_k$, the Gaussian
relative-entropy formula gives
\begin{equation}\label{eq:gaussian-mean-kl-control}
|u_t'-u_t|^2
\leq2\overline\Gamma\,
D_{\mathrm{KL}}(\pi_t'\mid\pi_t).
\end{equation}
The difference of the corresponding LQ state equations is linear. Standard
SDE stability therefore yields
\begin{equation}\label{eq:lq-state-control-mean-stability}
\E\left[\sup_{0\leq t\leq T}|X_t^{\pi'}-X_t^\pi|^2\right]
\leq C\E\int_0^T|u_t'-u_t|^2\,\d t.
\end{equation}
Using \eqref{eq:lq-state-control-mean-stability} in
\eqref{eq:LQ-polynomial}, followed by Young's inequality, gives
\[
|\mathcal Q_{\mathrm{LQ}}(\pi',\pi)|
\leq C\E\int_0^T|u_t'-u_t|^2\,\d t
\leq2C\overline\Gamma\mathcal D_h(\pi'\mid\pi).
\]
On the other hand, direct expansion of the quadratic costs, or equivalently
the performance-difference identity and
Proposition~\ref{prop:lq-semiconvexity}, gives
\[
J^\tau(\pi')-J^\tau(\pi)-\d J^\tau(\pi)(\pi'-\pi)
=\mathcal Q_{\mathrm{LQ}}(\pi',\pi)
+\tau\mathcal D_h(\pi'\mid\pi).
\]
This proves \eqref{eq:gaussian-lq-relative-smoothness} with
$L_{\rm LQ}=\tau+2C\overline\Gamma$. Since $\pi^{n+1}$ is Gaussian,
\eqref{eq:gaussian-lq-relative-smoothness} can first be applied with
$\pi'=\pi^{n+1}$; the proof of Theorem~\ref{thm:dissipation} then applies
verbatim and does not use any lower curvature bound. If the stated lower
bound on $\mathcal Q_{\mathrm{LQ}}$ also holds, it gives relative convexity
with modulus $\tau-\rho$, and the algebraic proof of
Theorem~\ref{thm:main} applies to the invariant Gaussian sequence.
\end{proof}

%\section{Appendix}\label{appendix:sec}

% The bibliography is precompiled because HAL does not transfer BibTeX databases
% to arXiv. Keeping the .bbl explicit also avoids a BibTeX invocation on arXiv.

\end{document}